\documentclass[a4paper,11pt]{amsart}

\usepackage[foot]{amsaddr}
\usepackage{mathrsfs}
\usepackage{enumerate}
\usepackage{dsfont}
\usepackage{mathtools}
\usepackage{amssymb}
\usepackage{bm}

\usepackage[
backend=bibtex,
style=numeric-comp
]{biblatex}
\usepackage[utf8]{inputenc}

\usepackage{amsmath, amsthm}
\usepackage{soul, comment}

\usepackage{rotating}
\usepackage{graphicx}
\usepackage{caption}
\usepackage{subcaption}
\usepackage{float}
\graphicspath{ {./images/} }
\usepackage[super]{nth}

\usepackage[table,xcdraw,svgnames]{xcolor}
\usepackage{hyperref}
\hypersetup{backref,colorlinks=true,citecolor=Blue}

\newtheorem{lemma}{Lemma}[section]

\newtheorem{definition}{Definition}[section]

\newtheorem{theorem}{Theorem}[section]
\newtheorem{remark}{Remark}[section]

\newcommand{\continuation}{??}

\newtheorem{example}{Example}[section]
\newtheorem{assumption}{Assumption}[section]

\newtheorem*{acknow*}{Acknowledgments}

\newcommand{\ignore}[1]{}
\allowdisplaybreaks
\hypersetup{colorlinks=true, linkcolor=blue, filecolor=magenta, urlcolor=cyan}
\newcommand\numberthis{\addtocounter{equation}{1}\tag{\theequation}}
\newcommand\Label[1]{&\refstepcounter{equation}(\theequation)\ltx@label{#1}&}

\newcommand{\R}{\mathbb{R}}

\newcommand{\norm}[1]{\left\lVert #1 \right\rVert}
\newcommand{\abs}[1]{\left\vert #1 \right\rvert}

\newcommand{\bxi}{\bar{\xi}}

\newcommand{\vmu}{\bm{\mu}}
\newcommand{\vxi}{\bm{\xi}}
\newcommand{\vxibar}{\bar{\bm{\xi}}}
\newcommand{\vrho}{\bm{\rho}}
\newcommand{\fsum}[1]{\Gamma^{\ell,j}[#1](\bm{x},\bm{y})}

\allowdisplaybreaks[3]

\author{M. Heldman$^{1,2}$}
\address{$^1$ Department of Mathematics, Virginia Tech, Blacksburg, VA 24061}
\address{$^2$ Department of Mathematics and Statistics, Boston University, Boston, MA 02215}
\email[Max Heldman]{maxh@vt.edu}
\author{S. A. Isaacson$^2$}
\email[Samuel A. Isaacson]{isaacsas@bu.edu}
\author{Q. Liu$^2$}
\email[Qianhan Liu]{liuq19@bu.edu}
\author{K. Spiliopoulos$^2$}
\email[Konstantinos Spiliopoulos]{kspiliop@bu.edu}
\thanks{S. A. I., M. H., and Q. L. were supported by ARO W911NF-20-1-0244 and National Science Foundation DMS-1902854. K.S. was partially supported by the National Science Foundation DMS-2107856, DMS-2311500, SIMONS Foundation Award 672441 and ARO W911NF-20-1-0244. }

\title[Mean field limit for the PBSRDD model with Potential Interactions]{Mean field limits of particle-based stochastic reaction-drift-diffusion models}

\date{\today}

\begin{document}
\begin{abstract}We consider particle-based stochastic reaction-drift-diffusion models where particles move via diffusion and drift induced by one- and two-body potential interactions.  The dynamics of the particles are formulated as measure-valued stochastic processes (MVSPs), which describe the evolution of the singular, stochastic concentration fields of each chemical species. The mean field large population limit of such models is derived and proven, giving coarse-grained deterministic partial integro-differential equations (PIDEs) for the limiting deterministic concentration fields' dynamics. We generalize previous studies on the mean field limit of models involving only diffusive motion, with care to formulating the MVSP representation to ensure detailed balance of reversible reactions in the presence of potentials. Our work illustrates the more general set of PIDEs that arise in the mean field limit, demonstrating that the limiting macroscopic reactive interaction terms for reversible reactions obtain additional nonlinear concentration-dependent coefficients compared to the purely diffusive case. Numerical studies are presented which illustrate that two-body repulsive potential interactions can have a significant impact on the reaction dynamics, and also demonstrate the empirical numerical convergence of solutions to the PBSRDD model to the derived mean field PIDEs as the population size increases.
\end{abstract}

	\maketitle
\section{Introduction}

We consider particle-based stochastic reaction-drift-diffusion (PBSRDD) models where particles move via diffusion and drift induced by one- and two-body potential interactions.  We formulate the dynamics of the particles as measure-valued stochastic processes (MVSPs), which describe the stochastic evolution of the concentration fields of each chemical species as a sum of $\delta-$functions encoding the position and type of each particle. Our goal is to formulate an appropriate MVSP model for such systems, and then rigorously investigate the large population limit of the MVSP dynamics, deriving partial-integral differential equations (PIDEs) that represent the limiting mean-field dynamics.

Particle-based stochastic reaction-diffusion (PBSRD) models have a long history of use in modeling the diffusion of, and reactions between, individual molecules. PBSRDD models are more macroscopic descriptions than millisecond-timescale quantum mechanical or molecular dynamics models of a few molecules \cite{ShawAntonMS2009}, but more microscopic descriptions than deterministic 3D reaction-diffusion PDEs for the average concentration field of each chemical species. One of the most popular PBSRDD models for studying biological processes is the volume reactivity (VR) model of Doi~\cite{TeramotoDoiModel1967,DoiSecondQuantA,DoiSecondQuantB}. In the Doi model, the positions of individual molecules are typically represented as points undergoing Brownian motion. Bimolecular reactions between two substrate molecules occur with a probability per unit time based on their current positions~\cite{DoiSecondQuantA,DoiSecondQuantB}. Unimolecular reactions are typically assumed to represent internal processes, and as such are modeled as occurring with exponentially distributed times based on a specified reaction-rate constant.

In our prior work \cite{IsaacsonMa2022}, we investigated the large population limit of PBSRD models (i.e. in the absence of drift). Allowing drift makes the models more relevant for applications, but complicates the model formulation and the analysis needed to prove the mean field limit. Many models of biological systems involve drift induced by background potential fields and/or by potential interactions. In modeling cellular processes, the former has been used to model how volume exclusion by DNA fibers impedes protein diffusion in the nucleus~\cite{IsaacsonPNAS2011,IsaacsonBmB2013}, and to model membrane-bound protein motion induced by actin contraction in T cell synapses~\cite{DustinHermann2018dw}. Two-body potential interaction fields have been used to more accurately account for attractive or repulsive interactions between molecules~\cite{ReaddyV219}, including to model volume exclusion due to the physical size of molecules~\cite{Bruna2017SIAP,DustinHermann2018dw}. More generally, such interactions can arise in diverse classes of agent-based models including models for the spread of infections or innovations within populations~\cite{Schutte2019}, or for interactions between cells. In addition, as our numerical studies in Section \ref{S:Simulations} demonstrate, the inclusion of potential interactions can have non-trivial effects on the behavior of the system. We therefore now investigate the large population limit PBSRDD models where particles move via diffusion and drift term induced by one-body and two-body potentials.

For simplicity, we will work in free space as in~\cite{IsaacsonMa2022,IMS21a}, and as such, we will study the large population mean field limit via an increasing scaling parameter, $\gamma$, which can physically represent Avogadro's number. The ``large system size" limit $\gamma\rightarrow\infty$ is where one typically obtains more macroscopic coarse-grained partial-integral differential equation (PIDE), PDE, SPIDE, or SPDE models for biological systems that model diffusion and reaction via the evolution of continuous concentration fields \cite{HeldmanSPA24,IsaacsonMa2022,Arnold1980dl,Nolen2019,Schutte2019,Oelschlager89a,VeberAnnals2023}. To determine the limit for PBSRDD models we adopt the classical Stroock-Varadhan Martingale approach \cite{EthierKurtz1986,StroockVaradhan2006} that previously allowed us to rigorously identify and prove the large population limit of the MVSP representation for PBSRDD models~\cite{IsaacsonMa2022}. We note that this method has been successful in many instances to study large population dynamics and general interacting particle systems, see \cite{TypicalDefaults,LargePortfolio,DaiPra2,DaiPra3,Delarue,Inglis,Moynot,Sompolinsky}. We identify a new macroscopic system of partial integro-differential equations (PIDEs) whose solution corresponds to the large population limit of the MVSP, and we rigorously prove the convergence (in a weak sense) of the MVSP to this solution.

We also note here that the bottom-up approach that we take in this paper allows us to derive the new macroscopic PIDEs corresponding to the true population limit of the underlying spatial PBSRDD model. This is to be contrasted with the well-known macroscopic reaction-drift-diffusion PDE models of chemical reactions at a cellular scale, which are often derived by modifying standard ODE models for non-spatial reaction systems via the addition of drift-diffusion operators to give a (phenomenological) spatial model. Note, in the absence of potential interactions (i.e., particles move only via diffusion), in  \cite{IMS21a} we proved that such models can be seen as limits of the rigorous mean field limit we derived in \cite{IsaacsonMa2022} when bimolecular reaction kernels are short-range and averaging.

For our model, consider $\Gamma_j$ to be the set of indices of particles of species $j$. In the absence of reactions, suppose that the $i$th particle is of species $j$ and located at position $Q_t^i$ at time $t$. The $i$th particle then moves according to
\begin{equation}  \label{Eq:ModelMotionFreeSpace}
\begin{aligned}
d Q_t^i &= -\biggl[\nabla v_{j}(Q_t^i) + \frac{1}{\gamma} \nabla \sum_{j'=1, j' \neq j}^J \sum_{k\in \Gamma_{j'}} u_{j,j'} \bigl( \|Q_t^i - Q_t^k\|\bigr)+ \frac{1}{\gamma} \nabla \sum_{k\in \Gamma_{j},k \neq i} u_{j,j} \bigl( \|Q_t^i - Q_t^k\|\bigr)\biggr] d t \\
&\qquad + \sqrt{2D^j} d W_t^i.
\end{aligned}
\end{equation}
Here $v_{j}(x)$ denotes the one-body potential experienced by a particle of type $j$ when at position $x$, and $u_{j,j'}(\|x-y\|) / \gamma$ the two-body potential between particles at locations $x$ and $y$ of types $j$ and $j'$ respectively. $N^{\zeta}(t)$ denotes the total number of particles at time $t$, and $\zeta=(\frac{1}{\gamma}, \eta)$ is a vector consisting of the mean field limit scaling parameter, $\gamma$, and a displacement range parameter, $\eta$. In the large population mean field limit, $\gamma$ plays the role of a system size (e.g., Avogadro's number, or in bounded domains the product of Avogadro's number and the domain volume)\cite{AndersonKurtz2015}. $\eta$ represents a regularization parameter we introduce to rigorously handle $\delta$-function placement densities for reaction products, which are commonly used in many PBSRDD models, see Section \ref{S:MainResult}.  Finally, $D^j$ represents the diffusion coefficient for a particle of species $j$, and $\{W_{t}^{i}\}_{i \in \mathbb{N}_{+}}$ is a countable collection of standard independent Brownian motions in $\mathbb{R}^{d}$.

From~\eqref{Eq:ModelMotionFreeSpace}, we see each particle follows the gradient of a potential landscape, the so-called suitability landscape, and also experiences an additional force derived from the two-body potentials. The scaling $\frac{1}{\gamma}$ in front of the pair-wise potentials is the mean field scaling, which, at least formally, preserves the total strength of the interaction at order one so that we have a well-defined large population mean field limit in the absence of reactions, see~\cite{JabinMeanfield2017}.

From a mathematical point of view, adding drift due to potential interactions makes the rigorous formulation of the particle model more complicated, and gives rise to a number of new terms that need to be appropriately handled for a coarse-grained mean field limit to exist. Specifically, in moving from PBSRD (purely-diffusive) to PBSRDD models (drift due to potential interactions) reactive interaction functions, which determine the probability per time substrates react and produce products at specified positions, may require modification. Such modifications ensure detailed balance of reaction fluxes, i.e., the statistical mechanical property that the pointwise forward and backward reaction fluxes should balance at equilibrium~\cite{IsaacsonZhangDB2022,FrohnerNoe2018,IsaacsonDriftDB}. When formulating PBSRDD reactive interaction terms to be consistent with detailed balance at equilibrium, the needed modifications result in a dependence on the full potential of the system (hence adding a dependence on the positions of non-reactant particles). We specifically formulate these modifications as an extra rejection/acceptance probability appearing within reactive interaction functions, following the approach proposed in~\cite{FrohnerNoe2018}, but note our results should be straightforward to adapt should one instead modify reaction rates or product placement mechanisms to ensure detailed balance. One of the contributions of this work is that we make all of these modifications precise in a general fashion via the MVSP formulation. We illustrate our results for a specific, but important, choice of rejection/acceptance probabilities based on the functional choices proposed by Fr\"{o}hner and No\'e~\cite{FrohnerNoe2018}. The modified reactive interaction functions we employ then give rise to additional nonlinear concentration-dependent coefficients within the limiting reactive terms of the resulting mean field PIDEs, giving rise to a new deterministic, macroscopic reaction-drift-diffusion model.


To illustrate the main result of this paper, consider the reversible $A+B \rightleftarrows C$ reaction as an example. Denote the forward  $A+B\rightarrow C$ reaction by $\mathcal{R}_1$ and the backward  $C \rightarrow A+B$ reaction by $\mathcal{R}_2$. Let us represent by $A(t)$ the stochastic process for the number of species A molecules at time $t$, and label the position of $i$th molecule of species A at time $t$ by the stochastic process $\bm{Q}_{i}^{A(t)} \subset \mathcal{R}^{d}$. The quantity
$$
A^{\gamma}(x, t)=\frac{1}{\gamma} \sum_{i=1}^{A(t)} \delta\bigl(x-Q_{i}^{A(t)}\bigr)
$$
corresponds to the stochastic, singular molar concentration field of species A at point $x$ at time $t$. We define $B^{\gamma}(x, t)$ and $C^{\gamma}(x, t)$ in a similar fashion, and let $\mathbf{S}^{\gamma}(x,t) = (A^{\gamma}(x,t),B^{\gamma}(x,t),C^{\gamma}(x,t))$. For each species we assume a spatially-constant diffusivity, $D^{\mathrm{A}}, D^{\mathrm{B}}$ and $D^{\mathrm{C}}$, respectively.

For $\mathcal{R}_1$, let $K^{\gamma}_{1}(x, y)$ denote the probability per time that one A molecule at $x$ and one $\mathrm{B}$ molecule at $y$ react. The product $C$ molecule is then placed at $z$ following the probability density $m^{\eta}_{1}(z | x, y)$, conditioning on substrates at $x$ and $y$. Here the displacement range parameter $\eta$ represents a mollification parameter for singular $\delta$-function placement densities. We further incorporate a rejection/acceptance mechanism via an acceptance probability $\pi^{\gamma}_1(z|x,y,\bm{q})$, which indicates the probability that the above reaction is accepted, i.e. alllowed to occur, and generates a product at $z$ given the positions of the substrates at $x$ and $y$, and of the non-substrate and non-product particles molecules at $\bm{q}$. We define $K_{2}^{\gamma}(z), m^{\eta}_{2}(x, y | z)$ and $\pi^{\gamma}_2(x,y|z, \bm{q})$ analogously for the reverse reaction $\mathcal{R}_2$. Finally, we assume that the acceptance probabilities can be equivalently rewritten as a function of the substrate positions, product positions, and the pre-reaction concentration fields, $\pi^{\gamma}_1\bigl(z|x,y, \mathbf{S}^{\gamma}(x',t)dx'\bigr)$ and $\pi^{\gamma}_2\bigl(x,y|z, \mathbf{S}^{\gamma}(x',t)dx'\bigr)$. In Section~\ref{S:FrohnerNoeModel} we illustrate that this assumption holds for the specific choices of rejection/acceptance probabilities proposed in~\cite{FrohnerNoe2018}.

Note that $K_1^{\gamma}, m_1^{\eta}$ and $\pi_1^{\gamma}$ all depend on $\zeta=(\frac{1}{\gamma}, \eta)$. In the large population limit that $\gamma \rightarrow \infty$ and $\eta \rightarrow 0$ jointly, denoted as $\zeta \rightarrow 0$, let $K_1, m_1$ and $\pi_1$ denote their respective (possibly rescaled) limits. We will specify our assumptions on the mean field limits and $\zeta$ scalings of $K_1^{\gamma}, m_1^{\eta}$, $\pi_1^{\gamma}$, $K_2^{\gamma}, m_2^{\eta}$ and $\pi_2^{\gamma}$, as well as specific functional examples for them, in Sections~\ref{S:NotationPrelim}, \ref{S:MainAssumptions} and ~\ref{S:FrohnerNoeModel}.

In this work, we derive the large population (thermodynamic) limit and prove, in a weak sense, that as $\zeta \rightarrow 0$,
$$
\bigl(A^{\gamma}(x, t), B^{\gamma}(y, t), C^{\gamma}(z, t)\bigr) \rightarrow \bigl(A(x, t), B(y, t), C(z, t)\bigr),
$$
with $\mathbf{S}(x,t) = (A(x,t),B(x,t),C(x,t))$ representing the limiting deterministic mean-field molar concentration fields. The latter satisfy the system of reaction-drift-diffusion PIDEs that
\begin{align} \label{Eq:ThreeSpeciesReversibleExampleLimit}
\partial_{t} &A(x, t) =D^{A} \Delta_{x} A(x, t)+ \nabla_x \cdot \bigl(\nabla_x v_1 (x) A(x,t)\bigr) \notag \\
&+ \nabla_x \cdot \biggl(A(x,t) \int_{\mathbb{R}^d} \bigl( \nabla u_{1,1}(\norm{x-y})A(y,t) + \nabla u_{1,2}(\norm{x-y})B(y,t) + \nabla u_{1,3}(\norm{x-y})C(y,t)\bigr) dy\biggr)\notag\\
&-\biggl(\int_{\mathbb{R}^{d}} K_{1}(x, y) \biggl(\int_{\mathbb{R}^{d}}m_1(z|x,y)\pi_1\bigl(z|x,y, \mathbf{S}(x',t)dx'\bigr)dz\biggr) B(y, t) d y\biggr) A(x, t)\notag\\
&+\int_{\mathbb{R}^{d}} K_{2}(z)\biggl(\int_{\mathbb{R}^{d}} m_{2}(x, y | z)\pi_{2}\bigl(x,y|z, \mathbf{S}(x',t)dx'\bigr) d y\biggr) C(z, t) d z \notag\\
\partial_{t} &B(y, t)=D^{B} \Delta_{y} B(y, t)+ \nabla_y \cdot \bigl(\nabla_y v_2 (y) B(y,t)\bigr)\notag\\
&+ \nabla_y \cdot \biggl(B(y,t) \int_{\mathbb{R}^d} \bigl(\nabla u_{2,1}(\norm{y-x})A(x,t) + \nabla u_{2,2}(\norm{y-x})B(x,t) + \nabla u_{2,3}(\norm{y-x})C(x,t)\bigr)dx\biggr)\notag\\
&-\biggl(\int_{\mathbb{R}^{d}} K_{1}(x, y) \biggl(\int_{\mathbb{R}^{d}}m_{1}(z| x,y)\pi_{1}\bigl(z|x,y, \mathbf{S}(x',t)dx'\bigr)dz\biggr)A(x, t) d x \biggr)B(y, t)\notag\\
&+\int_{\mathbb{R}^{d}} K_{2}(z)\biggl(\int_{\mathbb{R}^{d}} m_{2}(x, y | z)\pi_{2}\bigl(x,y|z, \mathbf{S}(x',t)dx'\bigr) d x\biggr) C(z, t) d z  \notag\\
\partial_{t} &C(z, t)=D^{C} \Delta_{z} C(z, t)+ \nabla_z \cdot \bigl(\nabla_z v_3 (z) C(z,t)\bigr)\notag\\
&+ \nabla_z \cdot \biggl(C(z,t) \int_{\mathbb{R}^d} \bigl( \nabla u_{3,1}(\norm{z-x})A(x,t) + \nabla u_{3,2}(\norm{z-x})B(x,t)+ \nabla u_{3,3}(\norm{z-x})C(x,t)\bigr)dx\biggr)\notag\\
&+\int_{\mathbb{R}^{d} \times \mathbb{R}^{d}} K_{1}(x, y) m_{1}(z | x, y) \pi_1\bigl(z|x,y, \mathbf{S}(x',t)dx'\bigr)A(x, t) B(y, t) d x d y\notag\\
&-K_{2}(z)\biggl(\int_{\mathbb{R}^{d} \times \mathbb{R}^{d}} m_{2}(x, y|z) \pi_2\bigl(x, y|z, \mathbf{S}(x',t)dx'\bigr)dxdy\biggr)C(z, t).\notag \\
\end{align}

The paper is organized as follows. In Section \ref{S:NotationPrelim}, we go over the notation and definitions that are used throughout this work. We then present the stochastic equation for the PBSRDD model, which describes the evolution of the empirical measure (MVSP) of the chemical species in path space, in Section \ref{S:GeneratorProcessLevelDescription}. In Section \ref{S:MainAssumptions}, we summarize the basic assumptions we make about the form of the reaction rate functions, product placement densities, and acceptance probabilities.
In Section \ref{S:MainResult}, we present our main result on the mean field limit, Theorem \ref{T:MainTheorem}, which describes the evolution equation satisfied by the empirical measures for the molar concentration of each species in the large population limit for general reaction networks. As illustrative examples, we also present the mean field limits for specific chemical systems. In Section \ref{S:FrohnerNoeModel}, we discuss a specific formulation of the acceptance probability in the generalized Fr\"{o}hner-No\'e model \cite{FrohnerNoe2018}, and examine its large-population limit. We also examine a number of reversible reactions and give the corresponding acceptance probabilities. In Section \ref{S:Simulations}, we present numerical studies illustrating the empirical convergence of numerical solutions to the PBSRDD model to the derived mean field PIDEs as $\zeta \rightarrow 0$. Our numerical examples also demonstrate that the potential interactions can have significant impacts on the statistical and dynamical behavior of the system. Finally, in Section \ref{S:ProofMainTheorem}, we give the proof of Theorem \ref{T:MainTheorem}. Appendix \ref{S:Appendix} includes the verification that the Fr\"{o}hner-No\'e type of acceptance probabilities satisfy the relevant assumptions made in Section \ref{S:MainAssumptions}, and provides the proof of a mollification-type result used in Section~\ref{S:ProofMainTheorem}.

\section{Notations and preliminary definitions}\label{S:NotationPrelim}

We consider a collection of particles with $J$ different types, and for the rest of the paper, we will interchangeably use the terms particle or molecule and type or species. Let $\mathcal{S}=\{S_{1}, \cdots, S_{J}\}$ denote the set of different possible types, with $p_{i} \in \mathcal{S}$ the value of the species of the $i$th particle. We also assume an underlying probability triple, $(\Omega, \mathcal{F}, \mathbb{P})$, on which all random variables are defined.

In our model, molecules diffuse in space $\mathbb{R}^{d}$ subject to drift arising from potential interactions, and can undergo $L$ possible types of reactions, denoted as $\mathcal{R}_{1}, \cdots, \mathcal{R}_{L}$. We use non-negative integer stoichiometric coefficients $\{\alpha_{\ell j}\}_{j=1}^{J}$ and $\{\beta_{\ell j}\}_{j=1}^{J}$ to describe the $\mathcal{R}_{\ell}$th reaction, $\ell \in\{1, \ldots, L\}$, as
$$
\sum_{j=1}^{J} \alpha_{\ell j} S_{j} \rightarrow \sum_{j=1}^{J} \beta_{\ell j} S_{j},
$$
and the multi-index vectors $\bm{\alpha}{ }^{(\ell)}=$ $\bigl(\alpha_{\ell 1}, \alpha_{\ell 2}, \cdots, \alpha_{\ell J}\bigr)$ and $\bm{\beta}^{(\ell)}=\bigl(\beta_{\ell 1}, \beta_{\ell 2}, \cdots, \beta_{\ell J}\bigr)$ to collect the coefficients of the $\ell$th reaction. We denote the substrate and product orders of the reaction by $|\bm{\alpha}^{(\ell)}| \doteq \sum_{i=1}^{J} \alpha_{\ell i} \leq 2$ and $|\bm{\beta}^{(\ell)}| \doteq \sum_{j=1}^{J} \beta_{\ell j} \leq 2$. The implicit assumption that all reactions are at most second order is justified by the assumption that the probability that three substrates in a dilute system simultaneously have the proper configuration and energy levels to react is small. It is further justified since reactions of order three and above are often considered to be approximations of sequences of bimolecular reactions in biological models. For subsequent notational purposes, we label the reactions such that the first $\tilde{L}$ reactions correspond to those that have no products, i.e., annihilation reactions of the form
$$
\sum_{j=1}^{J} \alpha_{\ell j} S_{j} \rightarrow \emptyset
$$
for $\ell \in\{1, \ldots, \tilde{L}\}$. We assume that the remaining $L-\tilde{L}$ reactions have one or more product particles.

Let $D^{i}$ label the diffusion coefficient for the $i$th particle, taking values in $\{D_{1}, \ldots, D_{J}\}$, where $D_{j}$ is the diffusion coefficient for species $S_{j}, j=1, \cdots, J$.
We assume drift is imparted to particles via one- and two-body potentials. Let $v_j(x)$ denote a background potential that imparts drift to a particle of species $j$ located at $x$. Similarly, we let $u^{\gamma}_{j,j'}(\norm{x - y}) := u_{j,j'}(\norm{x - y}) / \gamma$ represent a two-body potential experienced between a particle of species $j$ at $x$ and a particle of species $j'$ at $y$. Note, the inverse $\gamma$ scaling here is the standard mean field scaling to keep the total strength of the interaction fixed as more particles are added in the mean field limit~\cite{JabinMeanfield2017}.
We denote by $Q_{t}^{i} \in \mathbb{R}^{d}$ the position of the $i$th particle, $i \in \mathbb{N}_{+}$, at time $t$. In the absence of reactions, the dynamics for $Q_{t}^{i}$ are governed by (\ref{Eq:ModelMotionFreeSpace}). A particle's state can be represented as a vector in $\hat{P}=\mathbb{R}^{d} \times \mathcal{S}$, the combined space encoding particle position and type. This state vector is subsequently denoted by $\hat{Q}_{t}^{i} \stackrel{\text { def }}{=}\bigl(Q_{t}^{i}, p_{i}\bigr)$.

We now formulate our representation for the (number) concentration, equivalently number density, fields of each species. Let $E$ be a complete metric space and $M(E)$ the collection of measures on $E$. \textcolor{black}{Let $\mathcal{M}(E)$ be the subset of $M(E)$ consisting of all finite, non-negative point measures of the form
\begin{align*}
\mathcal{M}(E)&=\left\{\sum_{i=1}^{N}\delta_{Q^{i}}, N\geq 1, Q^{1},\cdots, Q^{N}\in E\right\}.
\end{align*}
}For $f: E \mapsto \mathbb{R}$ and $\mu \in M(E)$, define
$$
\langle f, \mu\rangle_{E}=\int_{x \in E} f(x) \mu(d x) .
$$
We will frequently have $E=\mathbb{R}^{d}$, in which case we omit the subscript $E$ and simply write $\langle f, \mu\rangle$. For each $t \geq 0$, we define the concentration (i.e., number density) of particles in the system at time $t$ by the distribution
\begin{equation}
\nu_{t}=\sum_{i=1}^{N(t)} \delta_{\hat{Q}_{t}^{i}}=\sum_{i=1}^{N(t)} \delta_{Q_{t}^{i}} \delta_{p_{i}}\label{ncen}
\end{equation}
where borrowing notation from \cite{Bansaye2015}, $N(t)=\langle 1, \nu_{t}\rangle_{\hat{P}}$ represents the stochastic process for the total number of particles at time $t$. To investigate the behavior of different types of particles, we denote the marginal distribution on the $j$th type, i.e., the concentration field for species $j$, by
$$
\nu_{t}^{j}(\cdot)=\nu_{t}(\cdot \times\{S_{j}\})
$$
a distribution on $\mathbb{R}^{d}$. $N_{j}(t)=\langle 1, \nu_{t}^{j}\rangle$ will similarly label the total number of particles of type $S_{j}$ at time $t$. \textcolor{black}{Note that in the remainder, in any rigorous calculation $\nu_{t}$ and $\nu_{t}^{j}$ will be measures and treated as such. We will, however, abuse notation and also refer to them as concentration fields, i.e., number densities. Strictly speaking, the latter should refer to the densities associated with such measures, but we ignore this distinction in subsequent discussions.} For $\nu$ any fixed particle distribution of the form \eqref{ncen}, we will also use an alternative representation in terms of the marginal distributions \textcolor{black}{$\nu^j\in \mathcal{M}(\R^d)$ for particles of type $j$,
\begin{equation} \label{eq:densitymeasdefmargrep}
\nu = \sum_{j = 1}^{J} \nu^j\delta_{S_j}  \in \mathcal{M}(\hat{P}).
\end{equation}}

In considering the mean field large population limit, we will take a simultaneous limit in which the population scaling parameter $\gamma \to \infty$, and the (convenience) displacement range parameter $\eta \to 0$ (see Section~\ref{S:MainAssumptions} for the definition of the latter). As mentioned in the introduction, this dual limit is encoded via the vector limit parameter
\begin{equation*}
\zeta := \left(\frac{1}{\gamma}, \eta\right) \to 0.
\end{equation*}
In studying this limit we will work with rescaled measures for each species, denoted by
\begin{equation*}
    \mu_{t}^{\zeta, j} := \frac{1}{\gamma} \nu_{t}^{\zeta, j}, \quad j = 1,\dots,J.
\end{equation*}
When $\gamma$ corresponds to Avogadro's number, $\mu_{t}^{\zeta, j}$ physically corresponds to the measure for which the associated density would represent the molar concentration field for species $j$ at time $t$ (but we will again abuse notation and also refer to $\mu_t^{\zeta,j}$ as the molar concentration field). We similarly let
\begin{equation*}
    \mu_{t}^{\zeta} := \frac{1}{\gamma} \nu_{t}^{\zeta}=\sum_{j=1}^{J} \mu_{t}^{\zeta,j} \delta_{S_{j}},
\end{equation*}
and define the vector of the molar concentrations for each species by
\begin{equation*}
    \vmu^{\zeta}_t := (\mu^{\zeta,1}_t,\dots,\mu^{\zeta,J}_t).
\end{equation*}
In the remainder, we will often write $N^{\zeta}(t) = \langle 1, \gamma \mu_t^{\zeta} \rangle_{\hat{P}}$ and $N^{\zeta}_j(t) = \langle 1, \gamma \mu_t^{\zeta,j}\rangle$ to make explicit that $N$ and $N_j$ depend on $\zeta$.


In addition to having notations for representing particle concentration fields, we will also use state vectors to store the positions of particles of a given type. Define the particle index maps $\{\sigma_{j}(k)\}_{k=1}^{N_{j}(t)}$, which encode a fixed ordering for the positions of particles of species $j$, $Q^{\sigma_{j}(1)} \preceq \cdots \preceq Q^{\sigma_{j}(N_{j}(t))}$, arising from an (assumed) fixed underlying ordering on $\mathbb{R}^{d}$. Following the notation established in \cite{Bansaye2015} (see Section 6.3 therein), we let $\mathbb{N}^{*}=\mathbb{N}\setminus\{0\}$ and let $H=(H^{1},\cdots, H^{k},\cdots ):\mathcal{M}(\mathbb{R}^d)\mapsto \left(\mathbb{R}^{d}\right)^{\mathbb{N}^{*}}$ such that
\begin{equation}
    H\bigl(\nu_{t}^{j}\bigr) \coloneqq \bigl(Q_{t}^{\sigma_{j}(1)}, \cdots, Q_{t}^{\sigma_{j}(N_{j}(t))}, 0,0, \cdots\bigr).
\end{equation}
$H(\nu^j_t)$ represents the position state vector for type $j$ particles. We analogously let $H^{i}\bigl(\nu_{t}^{j}\bigr) \in \mathbb{R}^{d}$ label the $i$th entry of the vector $H\bigl(\nu_{t}^{j}\bigr)$. Note that the zero entries after the $Q_{t}^{\sigma_{j}(N_{j}(t))}$ term merely serve as placeholders. As commented in \cite{Bansaye2015}, this function $H$ allows us to address a notational issue. In particular, choosing a particle of a type $j$ uniformly among all particles in $\nu^j_t \in\mathcal{M}(\mathbb{R}^d)$ amounts to choosing an index uniformly in the set $\{1,\cdots,\langle 1,\nu_t^j \rangle \}$, and then choosing the individual particle from the arbitrary fixed ordering.

As particles of the same type are assumed indistinguishable, there is no ambiguity in the value of $H\bigl(\nu_{t}^{j}\bigr)$ in the case that two particles of type $j$ have the same position. Using this notation we will often write the marginal distribution of species $j$ as
\begin{equation} \label{mu}
    \mu_{t}^{\zeta, j}(d x)
    = \frac{1}{\gamma} \sum_{i=1}^{N^{\zeta}_j(t)} \delta_{H^{i}(\gamma \mu_{t}^{\zeta, j})}(d x)
    = \frac{1}{\gamma} \sum_{i=1}^{\gamma\langle 1, \mu_{t}^{\zeta, j}\rangle} \delta_{H^{i}(\gamma \mu_{t}^{\zeta, j})}(d x),
    \quad j \in\{1, \ldots, J\}.
\end{equation}

With the preceding definitions, and analogously to \cite{IsaacsonMa2022}, we introduce a system of notation to encode substrate and particle positions and configurations that are needed to later specify reaction processes.
\begin{definition}\label{Def:1}
To describe the dynamics of $\nu_t$, we will sample vectors containing the indices of the specific substrate particles participating in a single $\ell$-type reaction from the substrate index space
$$
\mathbb{I}^{(\ell)}=(\mathbb{N} \backslash\{0\})^{|\alpha^{(\ell)}|}.
$$
For the allowable reactions considered in this work, we label the elements of $\mathbb{I}^{(\ell)}$ according to their species types:
\begin{enumerate}
    \item For $\mathcal{R}_{\ell}$ of the form $\varnothing \rightarrow \cdots$
    $$
\mathbb{I}^{(\ell)}=\varnothing.
$$
    \item For $\mathcal{R}_{\ell}$ of the form $S_j \rightarrow \cdots$
    $$
\mathbb{I}^{(\ell)}=\{i_1^{(j)}\in\mathbb{N}\backslash\{0\}\}.
$$
    \item For $\mathcal{R}_{\ell}$ of the form $S_j + S_k \rightarrow \cdots$ with $j<k$
        $$
\mathbb{I}^{(\ell)}=\{(i_1^{(j)},i_1^{(k)})\in (\mathbb{N}\backslash\{0\})^2\}.
$$
    \item For $\mathcal{R}_{\ell}$ of the form $2S_j \rightarrow \cdots$
           $$
\mathbb{I}^{(\ell)}=\{(i_1^{(j)},i_2^{(j)})\in (\mathbb{N}\backslash\{0\})^2\}.
$$
\end{enumerate}
We write a particular sampled set of substrate indices $\bm{i}\in \mathbb{I}^{(\ell)}$ as
$$
\bm{i}=\bigl(i_{1}^{(1)}, \cdots, i_{\alpha_{\ell 1}}^{(1)}, \cdots, i_{1}^{(J)}, \cdots, i_{\alpha_{\ell J}}^{(J)}\bigr).
$$
\end{definition}

\begin{definition}\label{Def:2} We define the substrate particle position space analogously to $\mathbb{I}^{(\ell)}$ as
$$\mathbb{X}^{(\ell)} \in (\mathbb{R}^d)^{|\alpha^{(\ell)}|},
$$
with an element $\bm{x} \in \mathbb{X}^{(\ell)}$ represented by
$\bm{x}=\bigl(x_{1}^{(1)}, \cdots, x_{\alpha_{\ell 1}}^{(1)}, \cdots, x_{1}^{(J)}, \cdots, x_{\alpha_{\ell J}}^{(J)}\bigr)$. For $\bm{x} \in \mathbb{X}^{(\ell)}$, a sampled substrate position configuration for one individual $\mathcal{R}_{\ell}$ reaction, $x_{r}^{(j)}$ then labels the sampled position for the $r$th substrate particle of species $j$ involved in the reaction. Let $d \bm{x}=\bigl(\bigwedge_{j=1}^{J}(\bigwedge_{r=1}^{\alpha_{\ell j}} d x_{r}^{(j)})\bigr)$ be the corresponding volume form on $\mathbb{X}^{(\ell)}$ which also naturally defines an associated Lebesgue measure.
\end{definition}

\begin{definition}\label{Def:3}For reaction $\mathcal{R}_{\ell}$ with $\tilde{L}+1 \leq \ell \leq L$, i.e., having at least one product particle, define the product position space analogously to $\mathbb{X}^{(\ell)}$,
$$
\mathbb{Y}^{(\ell)}\in (\mathbb{R}^d)^{|\beta^{(\ell)}|},
$$
with an element $\bm{y} \in \mathbb{Y}^{(\ell)}$ written as
$\bm{x}=\bigl(y_{1}^{(1)}, \cdots, y_{\beta_{\ell 1}}^{(1)}, \cdots, y_{1}^{(J)}, \cdots, y_{\beta_{\ell J}}^{(J)}\bigr)$. For $\bm{y} \in \mathbb{Y}^{(\ell)}$ a sampled product position configuration for one individual $\mathcal{R}_{\ell}$ reaction, $y_{r}^{(j)}$ then labels the sampled position for the $r$th product particle of species $j$ involved in the reaction. Let $d \bm{y}=\bigl(\bigwedge_{j=1}^{J}(\bigwedge_{r=1}^{\beta_{\ell_{j}}} d y_{r}^{(j)})\bigr)$ be the corresponding volume form on $\mathbb{Y}^{(\ell)}$, which also naturally defines an associated Lebesgue measure.
\end{definition}

\begin{definition}\label{Def:4}Consider a fixed reaction $\mathcal{R}_{\ell}$, with $\bm{i}\in \mathbb{I}^{(\ell)}$ and $\nu$ corresponding to a fixed particle distribution given by \eqref{ncen} with representation \eqref{eq:densitymeasdefmargrep}. We define the $\ell$th projection mapping \textcolor{black}{$\mathcal{P}^{(\ell)} :   \mathcal{M}(\hat{P})\times  \mathbb{I}^{(\ell)} \rightarrow  \mathbb{X}^{(\ell)}$} as
$$
\mathcal{P}^{(\ell)}(\nu, \bm{i})=\bigl(H^{i_{1}^{(1)}}(\nu^{1}), \cdots, H^{i_{\alpha_{\ell 1}}^{(1)}}(\nu^{1}), \cdots, H^{i_{1}^{(J)}}(\nu^{J}), \cdots, H^{i_{\alpha_{\ell J}}^{(J)}}(\nu^{J})\bigr).
$$
When substrates with indices $\bm{i}$ in particle distribution $\nu$ are chosen to undergo a reaction of type $\ell, \mathcal{P}^{(\ell)}(\nu, \bm{i})$ then gives the vector of the corresponding substrate particles' positions. For simplicity of notation, in the remainder, we will sometimes evaluate $\mathcal{P}^{(\ell)}$ with inconsistent particle distributions and index vectors. In all of these cases the inconsistency will occur in terms that are zero, and hence not matter in any practical way.
\end{definition}

\begin{definition}\label{omega} Consider a fixed reaction $\mathcal{R}_{\ell}$, with $\nu$ a fixed particle distribution given by \eqref{ncen} with representation \eqref{eq:densitymeasdefmargrep}. Using the notation of Definition \ref{Def:1}, we define the allowable substrate index sampling space $\Omega^{(\ell)}(\nu) \subset \mathbb{I}^{(\ell)}$ as
$$
\Omega^{(\ell)}(\nu)= \begin{cases}\varnothing, & |\bm{\alpha}^{(\ell)}|=0, \\ \{\bm{i}=i_{1}^{(j)} \in \mathbb{I}^{(\ell)} | i_{1}^{(j)} \leq \langle 1, \nu^{j}\rangle\}, & |\bm{\alpha}^{(\ell)}|=\alpha_{\ell j}=1, \\ \{i=\bigl(i_{1}^{(j)}, i_{2}^{(j)}\bigr) \in \mathbb{I}^{(\ell)} | i_{1}^{(j)}<i_{2}^{(j)} \leq\langle 1, \nu^{j}\rangle\}, & |\bm{\alpha}^{(\ell)}|=\alpha_{\ell j}=2, \\ \{i=(i_{1}^{(j)}, i_{1}^{(k)}) \in \mathbb{I}^{(\ell)} | i_{1}^{(j)} \leq\langle 1, \nu^{j}\rangle, i_{1}^{(k)} \leq\langle 1, \nu^{k}\rangle\}, & |\bm{\alpha}^{(\ell)}|=2, \alpha_{\ell j}=\alpha_{\ell k}=1, j<k .\end{cases}
$$
Note that in the calculations that follow $\Omega^{(\ell)}(\nu)$ will change over time due to the fact that $\nu=\nu_{t}$ changes over time, but this will not be explicitly denoted for notational convenience.
\end{definition}

\begin{definition}\label{Def:6} Consider a fixed reaction $\mathcal{R}_{\ell}$, with $\nu$ any element of $M(\hat{P})$ with the representation \eqref{eq:densitymeasdefmargrep}. We define the $\ell$th substrate measure mapping $\lambda^{(\ell)}[\cdot]: M(\hat{P}) \rightarrow M(\mathbb{X}^{(\ell)})$ evaluated at $\bm{x} \in \mathbb{X}^{(\ell)}$ via $\lambda^{(\ell)}[\nu](dx)= \otimes_{j=1}^{J}\bigl(\otimes_{r=1}^{\alpha_{\ell j}} \nu^{j}(d x_{r}^{(j)})\bigr).$ \label{lambda}
\end{definition}

\begin{definition}\label{xtilde}For reaction $\mathcal{R}_{\ell}$, define a subspace $\tilde{\mathbb{X}}^{(\ell)} \subset \mathbb{X}^{(\ell)}$ by removing all particle substrate position vectors in $\mathbb{X}^{(\ell)}$ for which two particles of the same species have the same position. That is
$$
\tilde{\mathbb{X}}^{(\ell)}=\mathbb{X}^{(\ell)} \backslash\{\bm{x} \in \mathbb{X}^{(\ell)} | x_{r}^{(j)}=x_{k}^{(j)} \textrm { for some } 1 \leq j \leq J, 1 \leq k \neq r \leq \alpha_{\ell j}\}.
$$
\end{definition}
\section{Generator and process level description}\label{S:GeneratorProcessLevelDescription}
To formulate the process-level model, it is necessary to specify more concretely the reaction process between individual particles. We make assumptions on $K_{\ell}, m_{\ell}^{\eta}$  and $\pi_{\ell}^{\gamma}$ that are analogous to what we assumed in the introduction for the  $A+B \rightleftarrows C$ reaction.

For reaction $\mathcal{R}_{\ell}$, denote by $K_{\ell}^{\gamma}(\bm{x})$ the rate (i.e., probability per time) that substrate particles with positions $\bm{x} \in \mathbb{X}^{(\ell)}$ react. As described in the next section, we assume this rate function has a specific scaling dependence on $\gamma$. Let $m_{\ell}^{\eta}(\bm{y}|\bm{x})$ be the placement density when the substrates at $\bm{x} \in \mathbb{X}^{(\ell)}$ react and generate products at $\bm{y} \in \mathbb{Y}^{(\ell)}$. We assume that for each $\bm{x}$ and fixed $\eta>0, m_{\ell}^{\eta}(\cdot | \bm{x})$ is bounded. We let $\pi_{\ell}^{\gamma}(\bm{y}|\bm{x},\bm{q})$ denote the probability that a candidate reaction between the substrates located at $\bm{x} \in \mathbb{X}^{(\ell)}$, given the non-substrate and non-product particles at $\bm{q}$, is accepted and produces products at $\bm{y} \in \mathbb{Y}^{(\ell)}$. We assume that this acceptance probability can equivalently be written in terms of $\bm{x}$, $\bm{y}$, and the molar concentration of each species (rather than the specific positions, $\bm{q},$ of each non-substrate and non-product particle). Therefore, we will usually write the acceptance probability as $\pi_{\ell}^{\gamma} \bigl(\bm{y}|\bm{x}, \vmu_t^{\zeta}(dx')\bigr)$. Note that we assume the acceptance probability may have an explicit dependence on $\gamma$. See Section~\ref{S:FrohnerNoeModel} for explicit examples that demonstrate this dependence.

To describe a reaction $\mathcal{R}_{\ell}$ with no products, i.e., $1 \leq \ell \leq \tilde{L}$, we associate with it a Poisson point measure $d N_{\ell}(s, i, \theta)$ on $\mathbb{R}_{+} \times \mathbb{I}^{(\ell)} \times \mathbb{R}_{+}$. Here $\bm{i} \in \mathbb{I}^{(\ell)}$ gives the sampled substrate configuration, with $i_{r}^{(j)}$ labeling the $r$th sampled index of species $j$. The corresponding intensity measure of $d N_{\ell}$ is given by $d \bar{N}_{\ell}(s, i, \theta)=d s\bigl(\bigwedge_{j=1}^{J}\bigl(\bigwedge_{r=1}^{\alpha_{\ell j}}\bigl(\sum_{k \geq 0} \delta_{k}(i_{r}^{(j)})\bigr)\bigr)\bigr) d \theta$. Analogously, for each reaction $\mathcal{R}_{\ell}$ with products, i.e., $\tilde{L}+1 \leq \ell \leq L$, we associate with it a Poisson point measure $d N_{\ell}(s, \bm{i}, \bm{y}, \theta_{1}, \theta_{2})$ on $\mathbb{R}_{+} \times \mathbb{I}^{(\ell)} \times \mathbb{Y}^{(\ell)} \times \mathbb{R}_{+} \times \mathbb{R}_{+}$. Here $\bm{i} \in \mathbb{I}^{(\ell)}$ gives the sampled substrate configuration, with $i_{r}^{(j)}$ labeling the $r$th sampled index of species $j$. $\bm{y} \in \mathbb{Y}^{(\ell)}$ gives the sampled product configuration, with $y_{r}^{(j)}$ labeling the sampled position for the $r$th newly created particle of species $j$. The corresponding intensity measure is given by $d \bar{N}_{\ell}(s, \bm{i}, \bm{y}, \theta_{1}, \theta_{2})=$ $d s\bigl(\bigwedge_{j=1}^{J}\bigl(\bigwedge_{r=1}^{\alpha_{\ell j}}\bigl(\sum_{k \geq 0} \delta_{k}(i_{r}^{(j)})\bigr)\bigr)\bigr) d \bm{y} d \theta_{1} d \theta_{2}$

The existence of the Poisson point measure follows as the intensity measure is $\sigma$-finite (see Theorem $1.8.1$ in \cite{IkedaWatanabe2014} or Corollary $9.7$ in \cite{Kurtz2001}). Let $d \tilde{N}_{\ell}(s, \bm{i}, \bm{y}, \theta_{1}, \theta_{2})=d N_{\ell}(s, \bm{i}, \bm{y}, \theta_{1}, \theta_{2})-d \bar{N}_{\ell}(s, \bm{i}, \bm{y}, \theta_{1}, \theta_{2})$ be the compensated Poisson measure, for $\tilde{L}+1 \leq \ell \leq L$. For any measurable set $A \in \mathbb{I}^{(\ell)} \times \mathbb{Y}^{(\ell)} \times \mathbb{R}_{+} \times \mathbb{R}_{+}$ \textcolor{black}{ such that $\bar{N}_{\ell}(\cdot,A)<\infty$, which is true if for example A is bounded,} $N_{\ell}(\cdot, A)$ is a Poisson process and $\tilde{N}_{\ell}(\cdot, A)$ is a martingale (see Proposition $9.18$ in \cite{Kurtz2001}). Similarly, we can define $d \tilde{N}_{\ell}(s, i, \theta)=d N_{\ell}(s, i, \theta)-d \bar{N}_{\ell}(s, i, \theta)$, for $1 \leq \ell \leq \tilde{L}$. In this case, given any measurable set $A \in \mathbb{I}^{(\ell)} \times \mathbb{R}_{+}$ \textcolor{black}{{such that $\bar{N}_{\ell}(\cdot,A) < \infty$}}, we then have that $N_{\ell}(\cdot, A)$ is a Poisson process and $\tilde{N}_{\ell}(\cdot, A)$ is a martingale.

\subsection{Process level description.}
We now formulate a weak representation for the time evolution of scaled empirical measures $\mu_{t}^{\zeta, j}$, $j=1, \dots, J$. Denote by $\{W_{t}^{n}\}_{n \in \mathbb{N}_{+}}$ a countable collection of standard independent Brownian motions in $\mathbb{R}^{d}$. For a test function $f \in C_{b}^{2}(\mathbb{R}^{d})$ and for each species $j=1, \cdots, J$, a weak representation of the dynamics of $\mu_t^{\zeta,j}$ are given by (see also \eqref{systemfull}),
\begin{align*}
 \langle f, \mu_{t}^{\zeta, j}\rangle
&=\langle f, \mu_{0}^{\zeta, j}\rangle+\frac{1}{\gamma}\sum_{i \geq 1} \int_{0}^{t} 1_{\{i \leq \gamma\langle 1, \mu_{s^-}^{\zeta, j}\rangle\}} \sqrt{2 D_{j}} \frac{\partial f}{\partial Q}(H^{i}(\gamma\mu_{s^-}^{\zeta, j})\bigr) d W_{s}^{i}\\
&\phantom{=} + \frac{1}{\gamma}\int_{0}^{t} \sum_{i=1}^{\gamma\langle 1, \mu_{s^-}^{\zeta, j}\rangle} \biggl(D_{j} \frac{\partial^{2} f}{\partial Q^{2}}\bigl(H^{i}(\gamma\mu_{s^-}^{\zeta, j})\bigr) - \frac{\partial f}{\partial Q}\bigl(H^{i}(\gamma\mu_{s^-}^{\zeta, j})\bigr) \cdot \nabla v_j \bigl(H^{i}(\gamma\mu_{s^-}^{\zeta, j})\bigr)\biggr) d s \\
&\phantom{=} - \frac{1}{\gamma}\int_{0}^{t} \sum_{i=1}^{\gamma\langle 1, \mu_{s^-}^{\zeta, j}\rangle}\biggl(\frac{\partial f}{\partial Q}\bigl(H^{i}(\gamma\mu_{s^-}^{\zeta, j})\bigr) \cdot \frac{1}{\gamma}\nabla \sum_{j'=1}^{J} \sum_{k=1}^{\gamma\langle 1, \mu_{s^-}^{\zeta, j'}\rangle}u_{j,j'}\bigl(\|H^{i}(\gamma\mu_{s^-}^{\zeta,j}) - H^{k}(\gamma\mu_{s^-}^{\zeta, j'})\|\bigr)\biggr)d s\\
&\phantom{=} + \frac{1}{\gamma}\int_{0}^{t}\sum_{i=1}^{\gamma\langle 1, \mu_{s^-}^{\zeta, j}\rangle} \biggl(\frac{\partial f}{\partial Q}\bigl(H^{i}(\gamma\mu_{s^-}^{\zeta, j})\bigr) \cdot \frac{1}{\gamma}\nabla u_{j,j}\|0\|\biggr)ds\\
&\phantom{=} +\sum_{\ell=1}^{\tilde{L}} \int_{0}^{t} \int_{\mathbb{I}^{(\ell)}} \int_{\mathbb{Y}^{(\ell)}} \int_{\mathbb{R}_{+}^{2}}\biggl(\langle f, \mu_{s^-}^{\zeta, j}-\frac{1}{\gamma} \sum_{r=1}^{\alpha_{\ell j}} \delta_{H^{i_{r}^{(j)}}(\gamma \mu_{s^-}^{\zeta, j})}\rangle-\langle f, \mu_{s^-}^{\zeta, j}\rangle\biggr)\\
&\phantom{=} \qquad\qquad\qquad\qquad\qquad \times 1_{\{\bm{i} \in \Omega^{(\ell)}(\gamma\mu_{s^-}^{\zeta})\}} 1_{\{\theta_1 \leq K_{\ell}^{\gamma}\bigl(\mathcal{P}^{(\ell)}(\gamma\mu_{s^-}^{\zeta}, \bm{i})\bigr)\}} d N_{\ell}(s, \bm{i}, \theta_1, \theta_2)\\
&\phantom{=} +\sum_{\ell=\tilde{L}+1}^{L} \int_{0}^{t} \int_{\mathbb{I}^{(\ell)}} \int_{\mathbb{Y}^{(\ell)}} \int_{\mathbb{R}_{+}^{3}}\biggl(\langle f, \mu_{s^-}^{\zeta, j}-\frac{1}{\gamma} \sum_{r=1}^{\alpha_{\ell j}} \delta_{H^{i_{r}^{(j)}}(\gamma \mu_{s^-}^{\zeta, j})}+\frac{1}{\gamma} \sum_{r=1}^{\beta_{\ell j}} \delta_{y_{r}^{(j)}}\rangle-\langle f, \mu_{s^-}^{\zeta, j}\rangle\biggr)\\
&\phantom{=} \qquad\qquad\qquad\qquad\qquad \times 1_{\{\bm{i} \in \Omega^{(\ell)}(\gamma\mu_{s^-}^{\zeta})\}} 1_{\{\theta_1 \leq K_{\ell}^{\gamma}\bigl(\mathcal{P}^{(\ell)}(\gamma\mu_{s^-}^{\zeta}, \bm{i})\bigr)\}}  1_{\{\theta_2 \leq m_{\ell}^{\eta}\bigl(\bm{y}|\mathcal{P}^{(\ell)}(\gamma\mu_{s^-}^{\zeta}, \bm{i})\bigr)\}}\\
&\phantom{=} \qquad\qquad\qquad\qquad\qquad \times 1_{\{\theta_3 \leq \pi_{\ell}^{\gamma} \bigl(\bm{y}|\bm{x}, \vmu_{s^-}^{\zeta}(dx')\bigr)\}} d N_{\ell}(s, \bm{i, y}, \theta_1, \theta_2, \theta_3).\numberthis\label{system}
\end{align*}


Formula \eqref{system} captures the dynamics of our particle system. \textcolor{black}{Recall that here $N^{\zeta}(s) = \sum_{j=1}^J N^{\zeta,j}(s)$ denotes the total number of particles in the system at time $s$, with $N^{\zeta,j}(s) = \gamma\langle 1, \mu_{s}^{\zeta,j}\rangle$ the number of particles of type $j$ at time $s$.} $D^{i}$ labels the diffusion coefficient for the $i$th molecule, taking values in $\{D_{1}, \ldots, D_{J}\}$, where $D_{j}$ is the diffusion coefficient for species $S_{j}, j=1, \cdots, J$. The drift and diffusion of each particle are modeled by the four integrals on the first four lines of \eqref{system}. The fifth to sixth lines model reactions with no products, while the seventh to ninth lines model reactions with products. The integrals involving the Poisson measures $N_{\ell}$ model the different components of the reaction processes, and correspond to sampling the times at which reactions occur, which substrate particles react, where reaction products are placed, and whether the proposed reaction is accepted.

When the $\ell$th reaction happens for $\ell=\tilde{L}+1, \cdots, L$ (and analogously for $\ell=1, \cdots, \tilde{L}$), with probability per time given by the kernel $K_{\ell}^{\gamma}$, the system loses substrate particles and gains product particles. A sampling of possible reaction occurrences according to $K_{\ell}^{\gamma}$ occurs through the indicator functions on the sixth and eighth lines. The corresponding loss and gain of particles are encoded by the sums of delta functions on the fifth and seventh lines. Product positions are sampled according to the placement density $\textcolor{black}{m_{\ell}^{\eta}}$ through the indicator function on the eighth line. For reactions with products, the reaction $\mathcal{R}_{\ell}$ then fires according to the acceptance probability $\textcolor{black}{\pi_{\ell}^{\gamma}}$ through the indicator function on the ninth line. The indicators over elements of the sets $\Omega^{(\ell)}(\gamma\mu_{s-}^{\zeta})$ ensure that reactions can only occur between particles that correspond to a possible set of substrates. We access particle positions via the state vector $H^{i}$, as the particle labeled by $i$ in \eqref{system} will change dynamically as reactions occur.

For a test function $f \in C_{b}^{2}(\mathbb{R}^{d})$ and for each species $j=1, \cdots, J$, let us define the following generators for the drift-diffusion of particles
\begin{align}
(\mathcal{L}_{j}f)(x)&\coloneqq D_j \Delta_x f(x) - \nabla_x f(x) \cdot \nabla_x v_j (x),\nonumber\\
(\tilde{\mathcal{L}}_{j,j'} f)(x,y) &\coloneqq - \nabla_x f(x) \cdot \nabla u_{j,j'}(\norm{x-y}),\label{Eq:Operators}
\end{align}
and the corresponding formal adjoint operators,
\begin{align}
    (\mathcal{L}^*_jf)(x)&\coloneqq D_j \Delta_x f(x) + \nabla_x \cdot (\nabla_x v_j (x)f(x)\bigr),\nonumber\\
    (\tilde{\mathcal{L}}_{j,j'}^{*} f)(x,y) &\coloneqq  \nabla_x \cdot \bigl(f(x)(\nabla u_{j,j'}(\norm{x-y})\bigr).\label{Eq:AdjointOperators}
\end{align}

We will subsequently assume that $N^{\zeta}(s)=\gamma\langle 1, \mu_{s}^{\zeta}\rangle$ is uniformly bounded in time in Assumption \ref{L:uniformboundofmeasures}. The stochastic integral with respect to Brownian motion in $\eqref{system}$ is then a martingale for a fixed $\zeta$. Taking the expectation, we obtain for the mean that
\begin{align*}
    \mathbb{E}[\langle f, \mu_{t}^{\bm{\zeta}, j}\rangle]
    &= \mathbb{E}[\langle f, \mu_{0}^{\zeta, j}\rangle]+\mathbb{E}\biggl[\int_{0}^{t}\langle(\mathcal{L}_{j} f)(x), \mu_{s^-}^{\zeta, j}(d x)\rangle d s\biggr]\\
    &\phantom{=} +\mathbb{E}\biggl[\int_{0}^{t}\bigl\langle\sum_{j'=1}^{J}\langle(\tilde{\mathcal{L}}_{j,j'} f)(x,y), \mu_{s^-}^{\zeta, j'}(dy)\rangle,\mu_{s^-}^{\zeta, j}(d x)\bigr\rangle d s\biggr]\\
    &\phantom{=} +\frac{1}{\gamma}\mathbb{E}\biggl[\int_{0}^{t} \langle \frac{\partial f}{\partial Q}(x) \cdot \nabla u_{j,j}(\|0\|), \mu_{s^-}^{\zeta, j}(d x)\rangle d s\biggr] \\
    &\phantom{=}  -\sum_{\ell=1}^{\tilde{L}} \mathbb{E}\biggl[\int_{0}^{t} \int_{\tilde{\mathbb{X}}(\ell)} \frac{1}{\bm{\alpha}^{(\ell)}!} K_{\ell}(\bm{x})\bigl(\sum_{r=1}^{\alpha_{\ell j}} f(x_{r}^{(j)})\bigr)\lambda^{(\ell)}[\mu_{s^-}^{\zeta}](d \bm{x}) d s\biggr]\\
    &\phantom{=}  +\sum_{\ell=\tilde{L}+1}^{L} \mathbb{E}\biggl[\int_{0}^{t} \int_{\tilde{\mathbb{X}}(\ell)} \frac{1}{\bm{\alpha}^{(\ell)}!} K_{\ell}(\bm{x})\biggl(\int_{\mathbb{Y}^{(\ell)}}\bigl(\sum_{r=1}^{\beta_{\ell j}} f(y_{r}^{(j)})-\sum_{r=1}^{\alpha_{\ell j}} f(x_{r}^{(j)})\bigr){m_{\ell}^{\eta}(\bm{y} | \bm{x})}\\
    &\phantom{=}\qquad\qquad\qquad\qquad\qquad \times \pi_{\ell}^{\gamma} \bigl(\bm{y}|\bm{x}, \vmu_{s^-}^{\zeta}(dx')\bigr) d \bm{y}\biggr) \lambda^{(\ell)}[\mu_{s^-}^{\zeta}](d \bm{x}) d s\biggr]. 
\end{align*}


\section{Assumptions for the Mean Field Limit and Main Results}\label{S:MainAssumptions}
\subsection{Assumptions for the Mean Field Limit.}
With the introduction of general one and two-body drift terms, we will constrain our choices of the potential functions, the reaction kernels, placement densities, and acceptance probabilities through the following assumptions, along with assuming some basic properties of the underlying reaction network.

\subsubsection{Assumptions on the molar concentration fields}
\begin{assumption}\label{L:uniformboundofmeasures}We assume that the total (molar) population concentration satisfies $\sum_{j=1}^{J}\langle 1, \mu_{t}^{\zeta, j}\rangle \newline \leq$ $C_{\circ}$ for all $t<\infty$, i.e., is uniformly in time bounded by some constant $C_{\circ}<\infty$. 
\end{assumption}

\begin{assumption}
We assume that for all $1 \leq j \leq J$, the initial distribution $\mu_{0}^{\zeta, j} \rightarrow \xi_{0}^{j}$ weakly as $\zeta \rightarrow 0$, where $\xi_{0}^{j}$ is a compactly supported measure with finite mass. \label{ainitial}
\end{assumption}

\subsubsection{Assumptions on potentials}
\begin{assumption}\label{anorm}
For all $1 \leq j,j' \leq J$ and $x,y \in \mathbb{R}^d$, the one-body potential $v_{j}(x)$ and the (unscaled) pairwise potential $u_{j,j'}(x,y)\coloneqq u_{j,j'}(\|x-y\|)$ have bounded $C^2(\mathbb{R}^{d})$ and $C^1(\mathbb{R}^{2d})$ function norms respectively, i.e., there is some constant $C<\infty$ such that for all $j$ and $j'$
\begin{align*}
   \norm{v_j}_{C^2(\mathbb{R}^{d})} &= \displaystyle\sup_{x \in \mathbb{R}^d} |v_j(x)|+\displaystyle\sup_{x \in \mathbb{R}^d} | v_j^{'}(x)|+\displaystyle\sup_{x \in \mathbb{R}^d} | v_j^{''}(x)| \leq C< \infty\\
    \norm{u_{j,j'}}_{C^1(\mathbb{R}^{2d})} &= \displaystyle\sup_{x,y \in \mathbb{R}^d} |u_{j,j'}(x, y)| + \displaystyle\sup_{x,y \in \mathbb{R}^d} |D u_{j,j'}(x, y)| \leq C< \infty.
\end{align*}

\end{assumption}

\subsubsection{Assumptions on reaction functions, placement densities and acceptance probabilities}
 \begin{assumption}\label{A:K_bound}
  We assume that for all $1 \leq \ell \leq L$, the reaction rate kernel $K_{\ell}(\bm{x})$ is uniformly bounded for all $\bm{x} \in \mathbb{X}^{(\ell)}$, and denote generic constants dependent upon this bound by $C(K)$.
 \end{assumption}

 \begin{assumption}
  The reaction kernel is assumed to have the explicit $\gamma$ dependence that
$$
K_{\ell}^{\gamma}(\bm{x})=\gamma^{1-|\bm{\alpha}^{(\ell)}|} K_{\ell}(\bm{x})
$$
for any $\bm{x} \in \mathbb{X}^{(\ell)}, 1 \leq \ell \leq L$. See~\cite{IsaacsonMa2022} for motivation and further details on this choice.
 \end{assumption}

\begin{assumption}\label{A:PlacementMeasure}
We assume that for any $\eta \geq 0, \tilde{L}+1 \leq \ell \leq L, \bm{y} \in \mathbb{Y}^{(\ell)}$ and $\bm{x} \in \mathbb{X}^{(\ell)}$, the placement density $m_{\ell}^{\eta}(\bm{y}|\bm{x})$ is \textcolor{black}{uniformly bounded in $\bm{x}$ and $\bm{y}$, and is a probability density in $\bm{y}$,} i.e., $\int_{\mathbb{Y}^{(\ell)}} m_{\ell}^{\eta}(\bm{y}|\bm{x}) d \bm{y}=1$. \label{m}
\end{assumption}

To define placement densities $m_{\ell}^{\eta}(\bm{y}|\bm{x})$ in terms of delta-functions in a mathematically rigorous way, we introduce the displacement (i.e., smoothing) range parameter $\eta$ in order to mollify the limiting Dirac delta densities $m_{\ell}(\cdot | \bm{x})$ in the standard way.

 \begin{definition}
 For $x \in \mathbb{R}^{d}$, let $G(x)$ denote a standard positive mollifier and $G_{\eta}(x)=\eta^{-d} G(x / \eta)$. That is, $G(x)$ is a smooth function on $\mathbb{R}^{d}$ satisfying the following four requirements
 \begin{enumerate}
     \item $G(x) \geq 0$;
     \item $G(x)$ is compactly supported in $B(0,1)$, the unit ball in $\mathbb{R}^{d}$;
     \item $\int_{\mathbb{R}^{d}} G(x) d x=1$;
     \item $\displaystyle\lim_{\eta \rightarrow 0} G_{\eta}(x)=\displaystyle\lim_{\eta \rightarrow 0} \eta^{-d} G(x / \eta)=\delta_{0}(x)$, where $\delta_{0}(x)$ is the Dirac delta function and the limit is taken in the space of Schwartz distributions.
 \end{enumerate}
 \end{definition}
 The allowable forms of the placement density for each possible reaction are summarized below.
  \begin{assumption}\label{A:PlacementDensity}
The distributional limit of $m_{\ell}^{\eta}(\bm{y}|\bm{x})$ as $\eta \rightarrow 0$ is given by
$m_{\ell}(\bm{y}|\bm{x})$, a linear combination of Dirac delta functions, for any $\bm{y} \in \mathbb{Y}^{(\ell)}$ and $\bm{x} \in \mathbb{X}^{(\ell)}$.
\begin{enumerate}
    \item For a first order reaction $\mathcal{R}_{\ell}$ of the form $S_i \rightarrow S_j$,
 we assume that the placement density $m_{\ell}^{\eta}(y | x)$ takes the mollified form of
$$
m_{\ell}^{\eta}(y | x)=G_{\eta}(y-x),
$$
with the distributional limit as $\eta \rightarrow 0$ given by
$$
m_{\ell}(y | x)=\delta_{x}(y).
$$
\item For a first-order reaction $\mathcal{R}_{\ell}$ of the form $S_i \rightarrow S_j + S_k$,  we assume that the unbinding displacement density is in the mollified form of
$$
m_{\ell}^{\eta}(x, y | z)=\rho(|x-y|) \sum_{i=1}^{I} p_{i} \times G_{\eta}\biggl(z-\bigl(\alpha_{i} x+(1-\alpha_{i}) y\bigr)\biggr)
$$
and the distributional limit as $\eta \rightarrow 0$ is given by
$$
m_{\ell}(x, y | z)=\rho(|x-y|) \sum_{i=1}^{I} p_{i} \times \delta\biggl(z-\bigl(\alpha_{i} x+(1-\alpha_{i}) y\bigr)\biggr),
$$
with $p_i, \alpha_i \in [0,1],$ for $i \in \{1,\cdots,I\}$ and $\displaystyle\sum_{i=1}^{I} p_{i}=1$.
    \item For a second order reaction $\mathcal{R}_{\ell}$ of the form $S_i + S_k \rightarrow S_j,$ we assume that the binding placement density $m_{\ell}(z | x, y)$ takes the mollified form of
$$
m_{\ell}^{\eta}(z | x, y)=\sum_{i=1}^{I} p_{i} \times G_{\eta}\biggl(z-\bigl(\alpha_{i} x+(1-\alpha_{i}) y\bigr)\biggr),
$$
    and the distributional limit as $\eta \rightarrow 0$ is given by
$$
m_{\ell}(z | x, y)=\sum_{i=1}^{I} p_{i} \times \delta\biggl(z-\bigl(\alpha_{i} x+(1-\alpha_{i}) y\bigr)\biggr),
$$
with $p_i, \alpha_i \in [0,1],$ for $i \in \{1,\cdots,I\}$ and $\displaystyle\sum_{i=1}^{I} p_{i}=1$.
    \item For a second order reaction $\mathcal{R}_{\ell}$ of the form $S_i + S_k \rightarrow S_j + S_k$,we assume that the placement density $m_{\ell}(z, w | x, y)$ takes the mollified form of
$$
m_{\ell}^{\eta}(z, w | x, y)=p \times G_{\eta}(x-z) \times G_{\eta}(y-w)+(1-p) \times G_{\eta}(x-w) \times G_{\eta}(y-z),
$$
and the distributional limit as $\eta \rightarrow 0$ is given by
$$
m_{\ell}(z, w | x, y)=p \times \delta_{(x, y)}\bigl((z, w)\bigr)+(1-p) \times \delta_{(x, y)}\bigl((w, z)\bigr),
$$
with $p \in [0,1].$
\end{enumerate}
 \end{assumption}

\begin{remark}
We mention here that even though in Assumption \ref{A:PlacementDensity} the placement density $m_{\ell}^{\eta}$ is assumed to depend only on the regularizing parameter $\eta$, it is physically possible that it can also depend on the scaling parameter $\gamma$. We will see such an instance in Remark  \ref{R:PlacementDensityGammaDepencence}. It will be shown there the dependence is typically such that it is inconsequential for the limit as $\gamma\rightarrow\infty$, and as such, we have chosen not to explicitly denote it for simplicity of exposition.
\end{remark}
\begin{assumption}\label{L:tail}
For \eqref{m} to be true, we will need the probability density $\rho$ to be normalized, i.e.
$$
\int_{\mathbb{R}^{d}} \rho(|w|) d w=1.
$$
Since $\rho$ is a probability density, the previous condition implies that the tail probability
$$
\int_{r>R} \rho(r)r^{d-1}  d r < \varepsilon,
$$
for any $\varepsilon >0$ when $R$ is chosen sufficiently large.
\end{assumption}

Finally, we summarize assumed properties of the pre- and post-limit acceptance probabilities, $\pi_{\ell}^{\gamma}$ and $\pi_{\ell}$ respectively. An explicit example of such probabilities is provided in Section~\ref{S:FrohnerNoeModel}. Let $M_{F}(\mathbb{R}^{d})$ be the space of finite measures endowed with the weak topology and $\mathbb{D}_{M_{F}(\mathbb{R}^{d})}[0, T]$ be the space of c\`adl\`ag paths with values in $M_{F}(\mathbb{R}^{d})$ endowed with Skorokhod topology. As we will work with subsets of $M_F(\mathbb{R}^d)$ in which the measure of $\R^d$ is uniformly bounded, we also let
\begin{equation} \label{eq:Mhatdef}
    \hat{M}_F(\mathbb{R}^d; C_{\circ}) \coloneqq \{\mu \in M_F(\mathbb{R}^d) \,|\, \langle 1, \mu \rangle \leq C_{\circ} \}.
\end{equation}
\begin{definition}
For a complete measurable space $E$, we define the variation norm of finite measures $\|\cdot\|_{M_{F}(E)}$ on $M_{F}(E)$ as
$$\|\nu\|_{M_{F}(E)}=\displaystyle\sup _{f \in L^{\infty}(E),\|f\|_{L} \infty \leq 1}|\langle f, \nu\rangle_{E}|.$$ One can show via a density argument that an equivalent formulation is (see step 4 of Theorem 3.2 of \cite{Jourdain2012}) $$\|\nu\|_{M_{F}(E)}=\displaystyle\sup _{f \in C_{b}^{2}(E),\|f\|_{L} \infty \leq 1}|\langle f, \nu\rangle_{E}|.$$
\end{definition}

\begin{assumption}  \label{lippi}
  We assume that for any $\gamma \geq 0, \tilde{L}+1 \leq \ell \leq L, \bm{y} \in \mathbb{Y}^{(\ell)}, \bm{x} \in \mathbb{X}^{(\ell)}$ and $\vxi \coloneqq (\xi^{1},\xi^{2},\cdots,\xi^{J})$, $\vxibar \coloneqq (\bxi^{1},\bxi^{2},\cdots,\bxi^{J})$ both in $\otimes_{j=1}^J \hat{M}_{F}(\mathbb{R}^{d}; C_{\circ})$, the acceptance probability $\pi_{\ell}\bigl(\bm{y}|\bm{x}, \vxi(dx')\bigr)$ is bounded and Lipschitz continuous with Lipschitz constant $P$, i.e.,
\begin{equation}
     \displaystyle\sup_{\bm{y} \in \mathbb{Y}^{(\ell)}, \bm{x} \in \mathbb{X}^{(\ell)}}\bigl|\pi_{\ell}\bigl(\bm{y}|\bm{x}, \vxi(dx')\bigr)-\pi_{\ell} \bigl(\bm{y}|\bm{x}, \vxibar(dx')\bigl)\bigr|\leq P \sum_{i=1}^J\|\xi^{i}-\bar{\xi}^{i}\|_{M_{F}(\mathbb{R}^{d})}.
\end{equation}
\end{assumption}

\begin{assumption}  \label{lippiiny}
  We assume that for any $\gamma \geq 0, \tilde{L}+1 \leq \ell \leq L, \bm{y}, \bm{y'} \in \mathbb{Y}^{(\ell)}, \bm{x} \in \mathbb{X}^{(\ell)}$ and $\vxi \coloneqq (\xi^{1},\xi^{2},\cdots,\xi^{J})\in \otimes_{j=1}^J \hat{M}_{F}(\mathbb{R}^{d}; C_{\circ})$, the acceptance probability $\pi_{\ell}\bigl(\bm{y}|\bm{x}, \vxi(dx')\bigr)$ is bounded and Lipschitz continuous with Lipschitz constant $\tilde{P} $, i.e.,
\begin{equation}
    \sup_{\substack{\bm{x} \in \mathbb{X}^{(\ell)}\\\vxi \in \hat{M}_{F}(\mathbb{R}^{d}; C_{\circ})}}\bigl|\pi_{\ell}\bigl(\bm{y}|\bm{x}, \vxi(dx')\bigr)-\pi_{\ell} \bigl(\bm{y'}|\bm{x}, \vxi(dx')\bigl)\bigr|\leq \tilde{P} \norm{\bm{y}-\bm{y'}}.
\end{equation}
\end{assumption}

\begin{assumption} \label{convpi}
 We assume that for any $\gamma \geq 0, \tilde{L}+1 \leq \ell \leq L$ and $\vxi \coloneqq (\xi^{1},\xi^{2},\cdots,\xi^{J}) \in \otimes_{j=1}^J \hat{M}_{F}(\mathbb{R}^{d}; C_{\circ})$, $\pi_{\ell}^{\gamma}\bigl(\bm{y}|\bm{x}, \vxi(dx')\bigr)$ converges to $\pi_{\ell}\bigl(\bm{y}|\bm{x}, \vxi(dx')\bigr)$ uniformly as $\gamma \rightarrow \infty$ with respect to $\bm{y} \in \mathbb{Y}^{(\ell)}, \bm{x} \in \mathbb{X}^{(\ell)}$; equivalently,
\begin{equation}
     \sup_{\bm{y} \in \mathbb{Y}^{(\ell)}, \bm{x} \in \mathbb{X}^{(\ell)}}\bigl|\pi_{\ell}\bigl(\bm{y}|\bm{x}, \vxi(dx')\bigr) -\pi_{\ell}^{\gamma}\bigl(\bm{y}|\bm{x}, \vxi(dx')\bigr)\bigr| \xrightarrow[\gamma \rightarrow \infty]{} 0.
\end{equation}
\end{assumption}

\section{Mean field limit and examples.}\label{S:MainResult}
Denote by
\begin{align*}
    \xi_t^j &:= \lim_{\zeta \to 0} \mu^{\zeta,j}_t, &
    \xi_t &:= \lim_{\zeta \to 0} \mu^{\zeta}_t, &
    \vxi_t &:= (\xi_t^1,\dots,\xi_t^J) = \lim_{\zeta \to 0} \vmu^{\zeta}_t
\end{align*}
the limiting measures. Our main result is then
 \begin{theorem}\label{T:MainTheorem}
 (Mean field large population limit) Given Assumptions \ref{L:uniformboundofmeasures}-\ref{convpi}, the sequence of measure-valued processes $\{\vmu_{t}^{\zeta}\}_{t \in[0, T]} \in \mathbb{D}_{\otimes_{j=1}^{J} M_{F}(\mathbb{R}^{d})}([0, T])$ is relatively compact in $\mathbb{D}_{\otimes_{j=1}^{J} M_{F}(\mathbb{R}^{d})}([0, T])$ for each $j=1,2, \cdots, J$. It converges in distribution to $\{\vxi_t\}_{t \in [0,T]} \in C_{\otimes_{j=1}^{J} M_{F}(\mathbb{R}^{d})}([0, T])$ as $\zeta \rightarrow 0$. 
  Each $\xi_{t}^{j}$ is respectively the unique solution to
\begin{equation} \label{evl}
\begin{aligned}
\langle f, \xi_{t}^{j}\rangle &=\langle f, \xi_{0}^{j}\rangle+\int_{0}^{t}\langle(\mathcal{L}_{j} f)(x), \xi_{s}^{j}(d x)\rangle d s +\int_{0}^{t}\bigl\langle\sum_{j'=1}^J\langle(\tilde{\mathcal{L}}_{j,j'} f)(x,y), \xi_s^{j'}(dy)\rangle,\xi_{s}^{j}(d x)\bigr\rangle d s\\
&\phantom{=} -\sum_{\ell=1}^{\tilde{L}}\int_{0}^{t} \int_{\tilde{\mathbb{X}}(\ell)} \frac{1}{\bm{\alpha}^{(\ell)}!} K_{\ell}(\bm{x})\bigl(\sum_{r=1}^{\alpha_{\ell j}} f(x_{r}^{(j)})\bigr)\lambda^{(\ell)}[\xi_{s}](d \bm{x}) d s\\
&\phantom{=} +\sum_{\ell=\tilde{L}+1}^{L}\int_{0}^{t} \int_{\tilde{\mathbb{X}}(\ell)} \frac{1}{\bm{\alpha}^{(\ell)}!} K_{\ell}(\bm{x})\biggl(\int_{\mathbb{Y}^{(\ell)}}\bigl(\sum_{r=1}^{\beta_{\ell j}} f(y_{r}^{(j)})-\sum_{r=1}^{\alpha_{\ell j}} f(x_{r}^{(j)})\bigr) m_{\ell}(\bm{y} | \bm{x})\\
&\phantom{=} \qquad\qquad\qquad\qquad\qquad\qquad\qquad\qquad\qquad \times \pi_{\ell}\bigl(\bm{y}|\bm{x}, \vxi_{s}(dx')\bigr)d \bm{y}\biggr) \lambda^{(\ell)}[\xi_{s}](d \bm{x}) d s.
\end{aligned}
\end{equation}
\end{theorem}
\begin{remark}
If the limiting measures $\vxi_t = (\xi_{t}^{1}(dx), \xi_{t}^{2}(dx), \cdots, \xi_{t}^{J}(dx)\bigr)$ have marginal densities $\vrho(x,t) := \bigl(\rho_1(x,t),\dots,\rho_J(x,t)\bigr)$, then these marginals solve, in a weak sense, the following reaction-diffusion PIDEs
\begin{align*}
\partial_{t} &\rho_{j}(x, t)=D_{j} \Delta_{x} \rho_{j}(x,t) + \nabla_x \cdot \bigl(\nabla_x v_j (x) \rho_j(x,t)\bigr)\\
&\phantom{=} + \nabla_x \cdot \bigl(\rho_j(x,t)\int_{\mathbb{R}^d} \sum_{j'=1}^{J} \nabla u_{j,j'}(\norm{x-y})\rho_{j'}(y,t)dy\bigr)\\
&\phantom{=} -\sum_{\ell=1}^{\tilde{L}}\frac{1}{\bm{\alpha}^{(\ell)}!} \int_{\tilde{\mathbb{X}}^{(\ell)}} K_{\ell}(\bm{x})\bigl(\sum_{r=1}^{\alpha_{\ell j}}\delta_{x}(x_{r}^{(j)}) \bigr) \bigl(\Pi_{k=1}^{J} \Pi_{s=1}^{\alpha_{\ell k}} \rho_{k}(x_{s}^{(k)}, t)\bigr) d \bm{x}\\
&\phantom{=} +\sum_{\ell=\tilde{L}+1}^{L}\frac{1}{\bm{\alpha}^{(\ell)}!} \int_{\tilde{\mathbb{X}}^{(\ell)}} K_{\ell}(\bm{x})\biggl(\int_{\mathbb{Y}^{(\ell)}}\bigl(\sum_{r=1}^{\beta_{\ell j}}\delta_{x}(y_{r}^{(j)}) -\sum_{r=1}^{\alpha_{\ell j}}\delta_{x}(x_{r}^{(j)}) \bigr) m_{\ell}(\bm{y} | \bm{x})\pi_{\ell}\bigl(\bm{y}|\bm{x}, \vrho(x',t)dx'\bigr)d\bm{y}\biggr) \\
&\phantom{=} \qquad\qquad\qquad\qquad\qquad\qquad\qquad\qquad\qquad\qquad\qquad\qquad\qquad\times \bigl(\Pi_{k=1}^{J} \Pi_{s=1}^{\alpha_{\ell k}} \rho_{k}(x_{s}^{(k)}, t)\bigr) d\bm{x}.
\end{align*}
\end{remark}

We now present a few examples to illustrate the limiting PIDEs for basic reaction types:

\begin{example}
Consider a system with three species, $A$, $B$, and $C$ that can undergo the reversible reaction $A+B \rightleftarrows C$. Define the measures for $A$, $B$, and $C$ particles at time $t$ respectively as $\mu_{t}^{\zeta, 1}, \mu_{t}^{\zeta, 2},$ and $\mu_{t}^{\zeta, 3} \in M(\mathbb{R}^{d})$. 

Let $\mathcal{R}_{1}$ be the forward reaction $A+B \rightarrow C$, with $K_{1}^{\gamma}(x, y)$ the probability per unit time one $A$ particle at position $x$ and one $B$ particle at position $y$ bind. Once reaction $\mathcal{R}_{1}$ fires, we generate a new particle $C$ at position $z$ following the placement density $m_{1}^{\eta}(z | x, y)$. Reaction $\mathcal{R}_{1}$ is then accepted with probability $\pi_{1}\bigl(z|x, y, \vmu_{t}^{\zeta}(dx')\bigr)$. For $\mathcal{R}_{1}$, the substrates are particles of species $A$ and $B$, so $\alpha_{11}=\alpha_{12}=1$ and $\alpha_{13}=0$. The product is one particle $C$, so $\beta_{11}=\beta_{12}=0$ and $\beta_{13}=1 .$

Let $\mathcal{R}_{2}$ be the backward reaction $C \rightarrow A+B$, with $K_{2}^{\gamma}(z)$ the probability per time one $C$ particle at position $z$ unbinds. Once reaction $\mathcal{R}_{2}$ fires, we generate a new particle $A$ at position $x$ and a new particle $B$ at position $y$ following the placement density $m_{2}^{\eta}(x,y|z)$. Reaction $\mathcal{R}_{2}$ is accepted with the acceptance probability $\pi_{2}\bigl(z|x, y, \vmu_{t}^{\zeta}(dx')\bigr)$. For $\mathcal{R}_{2}$, the substrate is a $C$ particle, so $\alpha_{21}=\alpha_{22}=0$ and $\alpha_{23}=1$. The products are $A$ and $B$ particles, so $\beta_{21}=\beta_{22}=1$ and $\beta_{23}=0 .$

If the limiting spatially distributed measures for species $A, B$ and $C$ have marginal densities $\bigl(\rho_{1}(x, t), \rho_{2}(x, t), \rho_{3}(x, t)\bigr)$ respectively, they solve the system (\ref{Eq:ThreeSpeciesReversibleExampleLimit}) with the molar concentrations $(A,B,C)=(\rho_1,\rho_2, \rho_{3})$.
\end{example}
\begin{example}
Consider a system with four species, $A, B, C$ and $D$ that can undergo the reversible reaction $A+B \rightleftarrows C +D.$
Define the measures for A, B, C, and D particles at time $t$ respectively as $\mu_{t}^{\zeta, 1}, \mu_{t}^{\zeta, 2}, \mu_{t}^{\zeta, 3}$ and $\mu_{t}^{\zeta, 4} \in M(\mathbb{R}^{d})$.

Let $\mathcal{R}_{1}$ be the forward reaction $A+B \rightarrow C+D$, with $K_{1}^{\gamma}(x, y)$ the probability per time one $A$ particle at position $x$ and one $B$ particle at position $y$ bind. Once reaction $\mathcal{R}_{1}$ fires, we generate one new particle $C$ at position $w$ and one new particle $D$ at position $z$ following the placement density $m_{1}^{\eta}(w, z | x, y)$. Reaction $\mathcal{R}_{1}$ is then accepted with probability $\pi_{1}\bigl(w,z|x, y, \vmu_{t}^{\zeta}(dx')\bigr)$. For $\mathcal{R}_{1}$, the substrates are particles of species $A$ and $B$, so $\alpha_{11}=\alpha_{12}=1$ and $\alpha_{13}=\alpha_{14}=0$. The products are one particle $C$ and one particle $D$, so $\beta_{11}=\beta_{12}=0$ and $\beta_{13}=\beta_{14}=1 .$

Let $\mathcal{R}_{2}$ be the backward reaction $C+D \rightarrow A+B$, with $K_{2}^{\gamma}(w,z)$ the probability per time one $C$ particle at position $w$ and one $D$ particle at position $z$ bind. Once reaction $\mathcal{R}_{2}$ fires, we generate a new particle $A$ at position $x$ and a new particle $B$ at position $y$ following the placement density $m_{2}^{\eta}(x,y|w,z)$. Reaction $\mathcal{R}_{2}$ is accepted with the acceptance probability $\pi_{2}\bigl(w,z|x, y, \vmu_{t}^{\zeta}(dx')\bigr)$. For $\mathcal{R}_{2}$, the substrates are one $C$ particle and one $D$ particle, so $\alpha_{21}=\alpha_{22}=0$ and $\alpha_{23}=\alpha_{24}=1$. The products are one $A$ particle and one $B$ particle, so $\beta_{21}=\beta_{22}=1$ and $\beta_{23}=\beta_{24}=0 .$

If the limiting spatially distributed measures for species $A, B, C$ and $D$ have marginal densities $\bigl(\rho_{1}(x, t), \rho_{2}(x, t), \rho_{3}(x, t), \rho_{4}(x, t)\bigr)$ respectively, they must solve the following reaction-diffusion equations in a weak sense:
\begin{align*}
\partial_{t} \rho_{1}(x, t)&=D_{1} \Delta_{x} \rho_{1}(x, t)+ \nabla_x \cdot \bigl(\nabla_x v_1 (x) \rho_1(x,t)\bigr)\\
&\phantom{=} + \nabla_x \cdot \bigl(\rho_1(x,t)\int_{\mathbb{R}^d} \sum_{j=1}^{4} \nabla u_{1,j}(\norm{x-y})\rho_{j}(y,t)dy\bigr)\\
&\phantom{=} -\biggl(\int_{\mathbb{R}^{d}} K_{1}(x, y) \biggl(\int_{\mathbb{R}^{2d}}m_1(w,z|x,y)\pi_1\bigl(w,z|x,y, \vrho(x',t)dx'\bigr)dw dz\biggr) \rho_{2}(y, t) d y\biggr) \rho_{1}(x, t)\\
&\phantom{=} +\int_{\mathbb{R}^{2d}} K_{2}(w,z)\biggl(\int_{\mathbb{R}^{d}} m_{2}(x, y |w, z)\pi_{2}\bigl(x,y|w,z, \vrho(x',t)dx'\bigr) d y\biggr) \rho_{3}(w, t) d w \rho_{4}(z, t) d z \\
\partial_{t} \rho_{2}(y, t) &= D_{2} \Delta_{y} \rho_{2}(y, t)+ \nabla_y \cdot \bigl(\nabla_y v_2 (y) \rho_2(y,t)\bigr)\\
&\phantom{=} + \nabla_y \cdot \bigl(\rho_2(y,t)\int_{\mathbb{R}^d} \sum_{j=1}^{4} \nabla u_{2,j}(\norm{y-x})\rho_{j}(x,t)dx\bigr) \\
&\phantom{=} -\biggl(\int_{\mathbb{R}^{d}} K_{1}(x, y) \biggl(\int_{\mathbb{R}^{2d}}m_{1}(w,z| x,y)\pi_{1}\bigl(w,z|x,y, \vrho(x',t)dx'\bigr)dwdz\biggr)\rho_{1}(x, t) d x \biggr)\rho_{2}(y, t)\\
&\phantom{=} +\int_{\mathbb{R}^{2d}} K_{2}(w,z)\biggl(\int_{\mathbb{R}^{d}} m_{2}(x, y |w, z)\pi_{2}\bigl(x,y|w,z, \vrho(x',t)dx'\bigr) d x\biggr) \rho_{3}(w, t) d w\rho_{4}(z, t) dz \\
\partial_{t} \rho_{3}(w, t) &= D_{3} \Delta_{w} \rho_{3}(w, t)+ \nabla_w \cdot \bigl(\nabla_w v_3 (w) \rho_3(w,t)\bigr)\\
&\phantom{=} + \nabla_w \cdot \bigl(\rho_3(w,t)\int_{\mathbb{R}^d} \sum_{j=1}^{4} \nabla u_{3,j}(\norm{w-x})\rho_{j}(x,t)dx\bigr)\\
&\phantom{=} +\int_{\mathbb{R}^{2d}} K_{1}(x, y)
\biggl(\int_{\mathbb{R}^{d}}m_{1}(w,z | x, y) \pi_1\bigl(w,z|x,y, \vrho(x',t)dx'\bigr)dz\biggr)\rho_{1}(x, t) \rho_{2}(y, t) d x d y\\
&\phantom{=} -\biggl(\int_{\mathbb{R}^{d}}K_{2}(w,z)\biggl(\int_{\mathbb{R}^{2d}} m_{2}(x, y|w,z) \pi_2\bigl(x, y|w,z, \vrho(x',t)dx'\bigr)dxdy\biggr) \rho_{4}(z, t)dz\biggr)\rho_{3}(w, t)\\
\partial_{t} \rho_{4}(z, t)&=D_{4} \Delta_{z} \rho_{4}(z, t)+ \nabla_z \cdot \bigl(\nabla_z v_4 (z) \rho_4(z,t)\bigr)\\
&\phantom{=} + \nabla_z \cdot \bigl(\rho_4(z,t)\int_{\mathbb{R}^d} \sum_{j=1}^{4} \nabla u_{4,j}(\norm{z-x})\rho_{j}(x,t)dx\bigr) \\
&\phantom{=} +\int_{\mathbb{R}^{2d}} K_{1}(x, y)\biggl(\int_{\mathbb{R}^{d}}m_{1}(w,z | x, y) \pi_1\bigl(w,z|x,y, \vrho(x',t)dx'\bigr)dw\biggr)\rho_{1}(x, t) \rho_{2}(y, t) d x d y\\
&\phantom{=} -\biggl(\int_{\mathbb{R}^{d}}K_{2}(w,z)\biggl(\int_{\mathbb{R}^{2d}} m_{2}(x, y|w,z) \pi_2\bigl(x, y|w,z, \vrho(x',t)dx'\bigr)dxdy\biggr) \rho_{3}(w, t)dw\biggr)\rho_{4}(z, t).
\end{align*}
\end{example}

\begin{example}
Consider a system with two species, $A$ and $B$ that can undergo the reversible dimerization reaction $A+A \rightleftarrows B.$
Define the measures for A and B particles at time $t$ respectively as $\mu_{t}^{\zeta, 1}$ and $\mu_{t}^{\zeta, 2} \in M(\mathbb{R}^{d})$.

Let $\mathcal{R}_{1}$ be the forward reaction $A+A \rightarrow B$, with $K_{1}^{\gamma}(x, y)$ the probability per time one $A$ particle at position $x$ and another $A$ particle at position $y$ bind. Once reaction $\mathcal{R}_{1}$ fires, we generate a new particle $B$ at position $z$ following the placement density $m_{1}^{\eta}(z | x, y)$. Reaction $\mathcal{R}_{1}$ is then accepted with probability $\pi_{1}\bigl(z|x, y, \vmu_{t}^{\zeta}(dx')\bigr)$. For $\mathcal{R}_{1}$, the substrates are particles of species $A$, so $\alpha_{11}=2$ and $\alpha_{12}=0$. The product is one particle $B$, so $\beta_{11}=0$ and $\beta_{12}=1.$

Let $\mathcal{R}_{2}$ be the backward reaction $B \rightarrow A+A$, with $K_{2}^{\gamma}(z)$ the probability per time one $B$ particle at position $z$ unbinds. Once reaction $\mathcal{R}_{2}$ fires, we generate two new $A$ particles at $x$ and $y$ following the placement density $m_{2}^{\eta}(x,y|z)$. Reaction $\mathcal{R}_{2}$ is accepted with the acceptance probability $\pi_{2}\bigl(z|x, y, \vmu_{t}^{\zeta}(dx')\bigr)$. For $\mathcal{R}_{2}$, the substrate is one $B$ particle, so $\alpha_{21}=0$ and $\alpha_{22}=1$. The products are two $A$ particles, so $\beta_{21}=2$ and $\beta_{22}=0$.

If the limiting spatially distributed measures for species $A$ and $B$ have marginal densities $\bigl(\rho_{1}(x, t), \rho_{2}(x, t)\bigr)$ respectively, they must solve the following reaction-diffusion equations in a weak sense:
\begin{align*}
\partial_{t} \rho_{1}(x, t)&=D_{1} \Delta_{x} \rho_{1}(x, t)+ \nabla_x \cdot \bigl(\nabla_x v_1 (x) \rho_1(x,t)\bigr)\\
&\phantom{=} + \nabla_x \cdot \bigl(\rho_1(x,t)\int_{\mathbb{R}^d} \sum_{j=1}^{2} \nabla u_{1,j}(\norm{x-y})\rho_{j}(y,t)dy\bigr)\\
&\phantom{=} -\biggl(\int_{\mathbb{R}^{d}}K_{1}(x, y) \biggl(\int_{\mathbb{R}^{d}}m_1(z|x,y)\pi_1\bigl(z|x,y, \vrho(x',t)dx'\bigr)dz\biggr) \rho_{1}(y, t)dy\biggr)\rho_{1}(x, t) \\
&\phantom{=} +2\int_{\mathbb{R}^{d}} K_{2}(z)\biggl(\int_{\mathbb{R}^{d}}m_{2}(x, y | z)\pi_{2}\bigl(x,y|z, \vrho(x',t)dx'\bigr)dy\biggr)\rho_{2}(z, t) d z \\
\partial_{t} \rho_{2}(y, t)&=D_{2} \Delta_{y} \rho_{2}(y, t)+ \nabla_y \cdot \bigl(\nabla_y v_2 (y) \rho_2(y,t)\bigr)\\
&\phantom{=} + \nabla_y \cdot \bigl(\rho_2(y,t)\int_{\mathbb{R}^d} \sum_{j=1}^{2} \nabla u_{2,j}(\norm{y-x})\rho_{j}(x,t)dx\bigr)\\
&\phantom{=} +\frac{1}{2}\int_{\mathbb{R}^{d} \times \mathbb{R}^{d}} K_{1}(x, y) m_{1}(z | x, y) \pi_1\bigl(z|x,y, \vrho(x',t)dx'\bigr)\rho_{1}(x, t) \rho_{1}(y, t) d x d y\\
&\phantom{=} -K_{2}(z)\biggl(\int_{\mathbb{R}^{d} \times \mathbb{R}^{d}} m_{2}(x, y|z) \pi_2\bigl(x, y|z, \vrho(x',t)dx'\bigr)dxdy\biggr)\rho_{2}(z, t).
\end{align*}
\end{example}

\section{A specific example: Fr\"{o}hner-No\'{e} Model}\label{S:FrohnerNoeModel}
In this section, we study an example based on the specific acceptance probabilities, $\pi(\bm{y}|\bm{x},\vmu_t^{\zeta})$, proposed in Fr\"{o}hner-No\'e in~\cite{FrohnerNoe2018}. We derive a specific formulation of the acceptance probability which preserves the detailed balance condition for general reversible reactions, present the corresponding acceptance probabilities, and illustrate their large population limit for general systems involving one- and two-body interactions.

To illustrate the particular acceptance probability $\pi_{\ell}^{\gamma}$ in the generalized Fr\"{o}hner-No\'e model, consider the $\mathcal{R}_{1}$ reaction $A+B \rightarrow C$, where one $A$ particle at $x$ binds with one $B$ particle at $y$ to produce one $C$ particle at $z$, and the other non-substrate and non-product particles are located at $\bm{q}$. Denote the total potential energy of the system before $\mathcal{R}_{1}$ by $\Phi_1^{-,\gamma}(x,y,\bm{q})$ and after $\mathcal{R}_{1}$ by $\Phi_1^{+,\gamma}(z,\bm{q})$. We assume that the total potential function depends on the system size parameter $\gamma$ and consists of only one- and two-body potentials. We then represent the total potential energy in the system prior to $\mathcal{R}_1$ by
\begin{equation*}
    \Phi_1^{-,\gamma}(x,y,\bm{q}) = \Phi^{\gamma}(\bm{q}) + v_1(x) + v_2(y) + u_1^{\gamma}(x;\bm{q})+u_2^{\gamma}(y;\bm{q}) + u_{1,2}^{\gamma}(x,y),
\end{equation*}
and the total potential energy in the system after $\mathcal{R}_1$ by
\begin{equation*}
    \Phi_1^{+,\gamma}(z,\bm{q}) = \Phi^{\gamma}(\bm{q}) + v_3(z) + u_3^{\gamma}(z;\bm{q}).
\end{equation*}
Here, $\Phi^{\gamma}(\bm{q})$ denotes the total potential interactions between the non-substrate and non-product particles at $\bm{q}$; $v_1(x)$, $v_2(y)$ and $v_3(z)$ represent all one-body interactions involving the substrates at $x$ and $y$ and the product at $z$, respectively; $u_1^{\gamma}(x;\bm{q})$, $u_2^{\gamma}(y;\bm{q})$ and $u_3^{\gamma}(z;\bm{q})$ denote all pairwise interactions between each substrate/product and the non-substrate and non-product particles at $\bm{q}$; finally, $u^{\gamma}_{1,2}(x,y) = u_{1,2}(x,y) / \gamma$ represents the specific two-body interaction between the two substrate particles.

Let $\vmu_t^{\zeta}$ denote the pre-reaction state of the system, consistent with the state corresponding to having the two substrates at $(x,y)$ and the non-reactant particles at $\bm{q}$. The Fr\"{o}hner-No\'{e} acceptance probability of $\mathcal{R}_1$  takes the form
\begin{equation*}
    \pi_{1}^{\gamma}(z|x,y,\vmu_t^{\zeta}(dx')) = \min\left\{1, e^{-\bigl[\Phi_1^{+,\gamma}(z,\bm{q})-\bigl(\Phi_1^{-,\gamma}(x,y,\bm{q})-u^{\gamma}_{1,2}(x,y)\bigr)\bigr]}\right\}.
\end{equation*}
In the next section we will demonstrate why it is appropriate to treat $\pi_1$ as a function of $z$, $x$, $y$, and $\vmu^{\zeta}_t$. The acceptance probability $\pi_{1}^{\gamma}$ always accepts a $\mathcal{R}_1$ reaction where the potential energy, excluding the pairwise potential $u_{1,2}^{\gamma}(x,y)$ between the substrates, decreases from pre-reaction stage to post-reaction stage. On the other hand, if the potential difference is positive, then the $\mathcal{R}_1$ reaction is only accepted a fraction of time. In Subsections \ref{SS:AcceptanceProbNoeModel} and \ref{SS:ExamplesNoeModel} we present the specifics of these constructions and illustrate them in a number of examples.

\subsection{Acceptance probability.}\label{SS:AcceptanceProbNoeModel}

In this section we present the functional form of the acceptance probability in the generalized Fr\"{o}hner-No\'e model, illustrate why it can be written as a function of $\vmu_t^{\eta}$ instead of the positions of the non-reactant particles, and discuss its limiting form as $\gamma \to \infty$. The proof that this generalized Fr\"{o}hner-No\'e acceptance probability satisfies the assumptions of this paper is presented in the appendix.

Let $V(\bm{x})$ represent the total one-body interactions involving each of the substrates at $\bm{x} \in \mathbb{X}^{(\ell)}$, where
\begin{equation*}
    V(\bm{x}) = \sum_{j=1}^{J}\sum_{r=1}^{\alpha_{\ell j}}v_j(x_r^{(j)}).
\end{equation*}
Similarly, denote all one-body interactions involving each of the product particles at $\bm{y} \in \mathbb{Y}^{(\ell)}$ by $V(\bm{y})$, where
\begin{equation*}
 V(\bm{y}) = \sum_{j=1}^{J}\sum_{r=1}^{\beta_{\ell j}}v_j(y_r^{(j)}).
\end{equation*}

Recall that we denote one-body potentials by $v_{j}(x)$ for each particle at $x$ of type $j$ and two-body potentials by $u_{j,j'}^{\gamma}(x,y)=\frac{1}{\gamma}u_{j,j'}(x,y)$ between particles at $x$ of type $j$ and at $y$ of type $j'$ for $j, j'=1, \cdots, J$. We slightly abuse notation for the following discussion and use $U^{\gamma}(\bm{x})$ to represent the total pairwise interactions between the substrate particles at $\bm{x} \in \mathbb{X}^{(\ell)}$, where
\begin{equation*}
    U^{\gamma}(\bm{x}) = \sum_{j=1}^{J}\sum_{r=1}^{\alpha_{\ell j}}\biggl[\sum_{j'=j+1}^{J}\sum_{r'=1}^{\alpha_{\ell j'}}\frac{1}{\gamma}u_{j,j'}(x^{(j)}_{r},x^{(j')}_{r'})+\sum_{r'=1}^{r-1}\frac{1}{\gamma}u_{j,j}(x^{(j)}_{r},x^{(j)}_{r'})\biggr].
\end{equation*}
$U^{\gamma}(\bm{y})$ is defined analogously to represent the total pairwise interactions between the product particles at $\bm{y}\in \mathbb{Y}^{(\ell)}$, where
\begin{equation*}
    U^{\gamma}(\bm{y}) = \sum_{j=1}^{J}\sum_{r=1}^{\beta_{\ell j}}\biggl[\sum_{j'=j+1}^{J}\sum_{r'=1}^{\beta_{\ell j'}}\frac{1}{\gamma}u_{j,j'}(y^{(j)}_{r},y^{(j')}_{r'})+\sum_{r'=1}^{r-1}\frac{1}{\gamma}u_{j,j}(y^{(j)}_{r},y^{(j)}_{r'})\biggr].
\end{equation*}
Here, the first term represents the pairwise potential between substrates/products of different types, and the second term accounts for the two-body interactions between substrates/products of the same type. The lower limit of the first summation, $j+1$, and the upper limit of the second summation, $r-1$, prevent counting the two-body potential between the substrates/products twice. Note that both $U^{\gamma}(\bm{x})$ and $U^{\gamma}(\bm{y})$ converge to $0$ as $\gamma \rightarrow \infty$ due to the assumed boundedness of the two-body potentials.

When the system state is given by $\vmu_{t^-}^{\zeta}$, for substrates at $\bm{x} \in \mathbb{X}^{(\ell)}$ we denote by $\bm{q}^{j} \in \mathbb{R}^{d\times \bigl(N_{j}(s^-)-\alpha_{\ell j}\bigr)}$ the position vector for the non-reactant particles of type $j$. That is, $\bm{q}^j$ corresponds to the particle positions within
\begin{align*}
    H\left(\gamma \mu_{t^-}^{\zeta,j}-\sum_{r=1}^{\alpha_{\ell j}}\delta_{H^{i_r^{(j)}}(\gamma \mu_{t^-}^{\zeta, j})}\right)
\end{align*}
where $\bm{i} \in \mathbb{I}^{(\ell)}$ are the subset of particle indices that are represented within $\bm{x}$. We then have $\bm{q} = ( \bm{q}^{1},\cdots, \bm{q}^{J})$.

With some abuse of notation, we write the sum of all pairwise interactions between the substrates at $\bm{x} \in \mathbb{X}^{(\ell)}$ and the non-substrate and non-product particles at $\bm{q}$ as $U^{\gamma}(\bm{x};\bm{q})$. To encode information about particle types, we let $U_{j,j'}^{\gamma}(\bm{x}^{j};\bm{q}^{j'})$ denote all pairwise interactions between the sampled substrates of type $j$ at $\bm{x}^{j} = (x_{1}^{(j)},\cdots,x_{\alpha_{\ell j}}^{(j)})$ and the non-reactant particles of species $j'$ at $\bm{q}^{j'}$, with $j, j'=1,\cdots,J$. Therefore
\begin{align*}
    U^{\gamma}(\bm{x};\bm{q}) &:= \sum_{j=1}^{J}\sum_{j'=1}^{J}U_{j,j'}^{\gamma}(\bm{x}^{j};\bm{q}^{j'})\\
    &= \sum_{j=1}^J\sum_{r=1}^{\alpha_{\ell j}}\sum_{j'=1}^J \biggl[\int_{\mathbb{R}^d}u_{j,j'}(x^{(j)}_{r},x)\mu_{t^-}^{\zeta, j'}(d x)-\sum_{r'=1}^{\alpha_{\ell j'}}\frac{1}{\gamma}u_{j,j'}(x^{(j)}_{r},x^{(j')}_{r'})\biggr],
\end{align*}
demonstrating that the two-body interactions between a set of substrates and non-reactant particles can be written solely in terms of the substrate positions and the components of $\vmu_{t^-}^{\zeta}$.

We analogously denote by $U^{\gamma}(\bm{y};\bm{q})$ the total two-body interactions between the products at $\bm{y} \in \mathbb{Y}^{(\ell)}$ and the non-substrate and non-product particles at $\bm{q}$. To specify particle types, we let $U_{j,j'}^{\gamma}(\bm{y}^{j};\bm{q}^{j'})$ denote all pairwise interactions between the products of type $j$ sampled from $\bm{y}^{j} = (y_{1}^{(j)},\cdots,y_{\beta_{\ell j}}^{(j)})$ and the non-substrate and non-product particles of species $j'$ sampled from $\bm{q}^{j'}$, with $j, j'=1,\cdots,J$. We then obtain
\begin{align*}
    U^{\gamma}(\bm{y};\bm{q}) &:= \sum_{j=1}^{J}\sum_{j'=1}^{J}U_{j,j'}^{\gamma}(\bm{y}^{j};\bm{q}^{j'})\\
    &= \sum_{j=1}^J\sum_{r=1}^{\beta_{\ell j}}\sum_{j'=1}^J \biggl[\int_{\mathbb{R}^d}u_{j,j'}(y^{(j)}_{r},x)\mu_{t^-}^{\zeta, j'}(d x)-\sum_{r'=1}^{\alpha_{\ell j'}}\frac{1}{\gamma}u_{j,j'}(y^{(j)}_{r},x^{(j')}_{r'})\biggr],
\end{align*}
demonstrating we can write these two-body interactions in terms of the substrate positions, the product positions, and the \emph{pre-reaction} state measure, $\vmu_{t^-}^{\zeta}$.

With the preceding definitions, we can represent the total potential for the system before $\mathcal{R}_{\ell}$ for $\bm{x} \in \mathbb{X}^{(\ell)}$ by
\begin{align*}
\Phi_{\ell}^{-,\gamma}(\bm{x},\bm{q})&= \Phi^{\gamma}(\bm{q}) + V(\bm{x}) + U^{\gamma}(\bm{x};\bm{q})+  U^{\gamma}(\bm{x})
\end{align*}
and the total potential for the system after $\mathcal{R}_{\ell}$ for $\bm{y} \in \mathbb{Y}^{(\ell)}$ by
\begin{align*}
   \Phi_{\ell}^{+,\gamma}(\bm{y},\bm{q})
   &= \Phi^{\gamma}(\bm{q}) + V(\bm{y}) + U^{\gamma}(\bm{y};\bm{q})+  U^{\gamma}(\bm{y}).
\end{align*}
\begin{remark} \label{rmk:potentialwithmeasures}
    Examining the preceding definitions, we see that we could alternatively write the potentials as functions of $\bm{x}$, $\bm{y}$, and $\vmu^{\zeta}_{t^-}$, i.e., $\Phi_{\ell}^{-,\gamma}(\bm{x},\vmu^{\zeta}_{t^-})$ and $\Phi_{\ell}^{+,\gamma}(\bm{y},\vmu^{\zeta}_{t^-}; \bm{x})$ respectively. We will subsequently make use of this representation to extend the pre-limit Fr\"{o}hner-No\'e acceptance probabilities to be functions of general finite measures in the Appendix. We use the $\bm{q}$ notation in this section as it is more consistent with how potentials are written in the modeling literature.
\end{remark}

For reversible reactions with differing numbers of substrates and products, for instance the reversible reactions $A+B \rightleftarrows C$ and $A+A \rightleftarrows B$, we denote the acceptance probability for the binding reaction $\mathcal{R}_1$ with $\bm{x} \in \mathbb{X}^{(1)}$ and $\bm{y} \in \mathbb{Y}^{(1)}$ by
\begin{align}\label{pduf}
\pi_{1}^{\gamma} \big(\bm{y}|\bm{x},\vmu_{t}^{\zeta}(dx')\bigr)= \min \left\{1, e^{-\bigl[\Phi_1^{+,\gamma}(\bm{y},\bm{q})-\bigl(\Phi_1^{-,\gamma}(\bm{x},\bm{q})-U^{\gamma}(\bm{x})\bigr)\bigr]}\right\},
\end{align}
and for the unbinding reaction $\mathcal{R}_2$ with $\bm{x} \in \mathbb{X}^{(2)}$ and $\bm{y} \in \mathbb{Y}^{(2)}$ by
\begin{align} \label{pdub}
\pi_{2}^{\gamma} \big(\bm{y}|\bm{x},\vmu_{t}^{\zeta}(dx')\bigr)= \min \left\{1, e^{-\bigl[\bigl(\Phi_2^{+,\gamma}(\bm{y},\bm{q})-U^{\gamma}(\bm{y})\bigr)-\Phi_2^{-,\gamma}(\bm{x},\bm{q})\bigr]}\right\},
\end{align}
in order to satisfy the detailed balance condition and preserve symmetry for the reversible reaction~\cite{FrohnerNoe2018,IsaacsonDriftDB}.

For other allowable reversible reaction types such as the reversible reactions $A+B \rightleftarrows C+D$ and $A+B \rightleftarrows A+C$ with products always placed at the positions of the substrates, we do not subtract the specific pairwise potential term $U^{\gamma}(\bm{x})$ between the substrates at $\bm{x} \in \mathbb{X}^{(1)}$ from the total potential energy $\Phi_{1}^{-,\gamma}(\bm{x},\bm{q})$ in the system prior to the forward reaction, nor subtract the specific pairwise potential term $U^{\gamma}(\bm{y})$ between the products at $\bm{y} \in \mathbb{Y}^{(2)}$ from the total potential energy $\Phi_{2}^{+,\gamma}(\bm{y},\bm{q})$ in the system after the backward reaction, for the detailed balance condition to hold. Instead, we consider the acceptance probability of the form
\begin{align} \label{pd}
\pi_{\ell}^{\gamma} \big(\bm{y}|\bm{x},\vmu_{t}^{\zeta}(dx')\bigr)= \min \left\{1, e^{-\bigl((\Phi_{\ell}^{+,\gamma}(\bm{y},\bm{q})-\Phi_{\ell}^{-,\gamma}(\bm{x},\bm{q})\bigr)}\right\},
\end{align}
with $\bm{x} \in \mathbb{X}^{(\ell)}$ and $\bm{y} \in \mathbb{Y}^{(\ell)}$, see~\cite{FrohnerNoe2018,IsaacsonDriftDB}.

When $\zeta$ goes to zero, the pairwise potentials $U^{\gamma}(\bm{x})$ between substrates at $\bm{x} \in \mathbb{X}^{(\ell)}$ and $U^{\gamma}(\bm{y})$ between the products at $\bm{y} \in \mathbb{Y}^{(\ell)}$ both vanish. The total two-body interactions between the substrates (products) and the non-reactant particles with $\bm{x} \in \mathbb{X}^{(\ell)}$ and $\bm{y} \in \mathbb{Y}^{(\ell)}$ are then
\begin{equation*}
    U(\bm{x}; \vxi_t) := \lim_{\zeta \to 0} U^\gamma(\bm{x}; \bm{q})
    = \sum_{j=1}^J\sum_{r=1}^{\alpha_{\ell j}}\sum_{j'=1}^J\int_{\mathbb{R}^d}u_{j,j'}(x^{(j)}_{r},x)\xi_{t}^{j'}(d x),
\end{equation*}
and
\begin{equation*}
    U(\bm{y}; \vxi_t) := \lim_{\zeta \to 0} U^{\gamma} (\bm{y};\bm{q})
    = \sum_{j=1}^J\sum_{r=1}^{\beta_{\ell j}}\sum_{j'=1}^J\int_{\mathbb{R}^d}u_{j,j'}(y^{(j)}_{r},x)\xi_{t}^{j'}(d x),
\end{equation*}
where $\xi_t^{j}$ and $\xi_t^{j'}$  denote the corresponding large population limits of $\mu_t^{\zeta,j}$ and $\mu_t^{\zeta,j'}$ for $j,j' = 1,\cdots J$. Substituting in these formulas, we obtain that the mean field limits of the three forms of acceptance probabilities coincide. More specifically, the mean field limit of $\Phi_{\ell}^{-,\gamma}(\bm{x},\bm{q})$ for $\bm{x} \in \mathbb{X}^{(\ell)}$ is
\begin{align*}
\Phi_{\ell}^{-}(\bm{x}, \vxi_t)&:= \sum_{j=1}^{J}\sum_{r=1}^{\alpha_{\ell j}}v_j(x^{(j)}_{r})+\sum_{j=1}^J\sum_{r=1}^{\alpha_{\ell j}} \sum_{j'=1}^J\int_{\mathbb{R}^d}u_{j,j'}(x^{(j)}_{r},x)
\xi_{t}^{j'}(d x),
\end{align*}
and the mean field limit of $\Phi_{\ell}^{+,\gamma}(\bm{y},\bm{q})$ for $\bm{y} \in \mathbb{Y}^{(\ell)}$ is
\begin{equation*}
       \Phi_{\ell}^{+}(\bm{y}, \vxi_t ) := \sum_{j=1}^{J}\sum_{r=1}^{\beta_{\ell j}}v_j(y_{r}^{(j)})+\sum_{j=1}^J\sum_{r=1}^{\beta_{\ell j}}\sum_{j'=1}^J\int_{\mathbb{R}^d}u_{j,j'}(y_r^{(j)},x)\xi_{t}^{j'}(d x).
\end{equation*}
As such, the mean field limit of all forms of the acceptance probabilities, 
$\pi_{\ell}^{\gamma} \big(\bm{y}|\bm{x},\vmu_{t^-}^{\zeta}(dx')\bigr)$, are
\begin{equation}\label{pdmfl}
\pi_{\ell} \big(\bm{y}|\bm{x},\vxi_{t}(dx')\bigr)= \min \left\{1, e^{-\bigl((\Phi_{\ell}^{+}(\bm{y},\vxi_t)-\Phi_{\ell}^{-}(\bm{x},\vxi_t)\bigr)}\right\},
\end{equation}
with $\bm{x} \in \mathbb{X}^{(\ell)}, \bm{y} \in \mathbb{Y}^{(\ell)}, \{\vmu_{t}^{\zeta}\}_{t \in[0, T]} \in \mathbb{D}_{\otimes_{j=1}^{J} M_{F}(\mathbb{R}^{d})}([0, T])$ and $\{\vxi_{t}\}_{t \in [0, T]} \in C_{\otimes_{j=1}^{J} M_{F}(\mathbb{R}^{d})}([0, T])$.

\begin{remark}
The total potential between the non-substrate and non-product particles $\Phi^{\gamma}(\bm{q})$ never appears in the formulae \eqref{pduf}, \eqref{pdub}, and \eqref{pd} of the acceptance probability, as this term remains the same before and after the reaction and vanishes in the potential difference.
\end{remark}

Previously, the position vectors $\bm{x}$ and $\bm{y}$ referred to locations of the substrate and product particles in the $\ell$th reaction. We now directly compare the forward and backward directions in one reversible reaction cycle where a set of substrates at $\bm{x}$ are replaced by a set of products at $\bm{y}$ and vice versa, with the non-reactant particles at $\bm{q}$. As such, in the forward reaction $\mathcal{R}_1$, the substrates are placed at $\bm{x} \in \mathbb{X}^{(1)}$ with the products located at $\bm{y} \in \mathbb{Y}^{(1)}$, and in the backward reaction $\mathcal{R}_2$, the substrates are placed at $\bm{y} \in \mathbb{X}^{(2)}$ while the products are located at $\bm{x} \in \mathbb{Y}^{(2)}$. In this case, we have that
\begin{align*}
    \Phi_1^{-,\gamma}(\bm{x},\bm{q}) &= \Phi_2^{+,\gamma}(\bm{x},\bm{q}),\\
    \Phi_1^{+,\gamma}(\bm{y},\bm{q}) &= \Phi_2^{-,\gamma}(\bm{y},\bm{q}),
\end{align*}
where $\bm{x} \in \mathbb{X}^{(1)} = \mathbb{Y}^{(2)}$ and $\bm{y} \in \mathbb{Y}^{(1)}=\mathbb{X}^{(2)}$.

\begin{remark}\label{R:PlacementDensityGammaDepencence}
For such reversible reactions where the number of substrates and products differs, for example $A+B \rightleftarrows C$ and $A+A \rightleftarrows B$, the placement densities for $\mathcal{R}_2$ may be chosen to take the form of $m_{2}^{\zeta}(\bm{x}|\bm{y}) \coloneqq \frac{1}{Z^{\zeta}}m_{1}^{\eta}(\bm{y}|\bm{x})e^{-U^{\gamma}(\bm y)}$, which converges to $m_{2}(\bm{x}|\bm{y}) \coloneqq \frac{1}{Z}m_{1}(\bm y | \bm x)$ in the mean field limit as $\zeta \rightarrow 0$, where $Z^{\zeta}$ and $Z$ are the normalizing constants. We provide specific formulations of the placement densities for the reversible $A+B \rightleftarrows C$ reaction in Example \ref{E:ABC} below. We again have $\bm{x} \in \mathbb{X}^{(1)} = \mathbb{Y}^{(2)}$ and $\bm{y} \in \mathbb{Y}^{(1)}=\mathbb{X}^{(2)}$. See \cite{FrohnerNoe2018} and \cite{IsaacsonDriftDB} for more details.
\end{remark}

\subsection{Examples.}\label{SS:ExamplesNoeModel}
\begin{example}\label{E:ABC}
Consider a system with three species, $A$, $B$, and $C$ that can undergo the reversible reaction $A+B \rightleftarrows C$. Let $\mathcal{R}_{1}$ be the forward reaction $A+B \rightarrow C$, where one $A$ particle at position $x$ and one $B$ particle at position $y$ bind to generate one $C$ particle at position $z$. Assume the non-reactant particles that are unchanged by the reaction are located at $\bm{q}$ as in the last section.

We denote the specific two-body interaction between the substrates by
\begin{equation*}
    U^{\gamma}(x,y) = u_{1,2}^{\gamma}(x,y)=\frac{1}{\gamma}u_{1,2}(x,y),
\end{equation*}
and represent the total potential of the system before $\mathcal{R}_1$ occurs by
\begin{align*}
\Phi_1^{-,\gamma}(x,y,\bm{q})=& \Phi^{\gamma}(\bm{q})+  v_1(x)+v_2(y)
   +\sum_{j=1}^{3}\int_{\mathbb{R}^d}u_{1,j}(x,y')\mu_{s^-}^{\zeta, j}(d y')-\frac{1}{\gamma}u_{1,1}(x,x)\\
   +&\sum_{j=1}^{3}\int_{\mathbb{R}^d}u_{2,j}(y,x') \mu_{s^-}^{\zeta, j}(d x')-\frac{1}{\gamma}u_{2,1}(y,x)-\frac{1}{\gamma}u_{2,2}(y,y),
\end{align*}
where $\vmu_{s^-}^{\zeta}$ represents the system's state before an $\mathcal{R}_1$ reaction at time $s$ (i.e., with an $A$ particle that will react at $x$, a $B$ particle that will react at $y$, and non-reactant particles at $\bm{q}$). The total potential energy of the system after $\mathcal{R}_1$ is denoted by
\begin{align*}
   \Phi_1^{+,\gamma}(z,\bm{q})=& \Phi^{\gamma}(\bm{q})+ v_3(z)+\sum_{j=1}^{3}\int_{\mathbb{R}^d}u_{3,j}(z,x')\mu_{s^-}^{\zeta, j}(d x')-\frac{1}{\gamma}u_{3,1}(z,x)-\frac{1}{\gamma}u_{3,2}(z,y),
\end{align*}
where we have written the potential in terms of the pre-reaction state, $\vmu^{\zeta}_{s^-}$. The acceptance probability for $\mathcal{R}_1$ is thus
\begin{equation*}
   \pi_{1}^{\gamma} \big(z|x,y,\vmu_{s}^{\zeta}(dx')\bigr) = \min \left\{1, e^{-\bigl[\Phi_1^{+,\gamma}(z,\bm{q})-\bigl(\Phi_1^{-,\gamma}(x,y,\bm{q})-u_{1,2}^{\gamma}(x,y)\bigr)\bigr]}\right\}.
\end{equation*}
The mean field limit of
$\Phi_1^{-,\gamma}(x,y,\bm{q})$ is
\begin{equation*}
    \Phi_1^{-}(x, y,\vxi_{s}) := v_1(x)+v_2(y)
   +\sum_{j=1}^{3}\int_{\mathbb{R}^d}u_{1,j}(x,y')\xi_{s}^{j}(d y')+\sum_{j=1}^{3}\int_{\mathbb{R}^d}u_{2,j}(y,x') \xi_{s}^{j}(d x'),
\end{equation*}
and the mean field limit of $\Phi_1^{+,\gamma}(z,\bm{q})$ is
\begin{equation*}
  \Phi_1^{+}(z,\vxi_s) := v_3(z)+\sum_{j=1}^{3}\int_{\mathbb{R}^d}u_{3,j}(z,x')\xi_{s}^{j}(d x').
\end{equation*}
Note that the pairwise potential term $U^{\gamma}(x,y)=\frac{1}{\gamma}u_{1,2}(x,y)$ between the substrates converges to zero as $\gamma \rightarrow \infty$.
Therefore, the mean field limit of the acceptance probability for $\mathcal{R}_1$ simplifies to
\begin{equation*}
\pi_{1} \big(z|x,y,\vxi_{s}(dx')\bigr)= \min \left\{1, e^{-\bigl((\Phi_1^{+}(z,\vxi_s)-\Phi_1^{-}(x,y,\vxi_s)\bigr)}\right\}.
\end{equation*}

Let $\mathcal{R}_{2}$ denote the backward reaction $C \rightarrow A+B$, where one $C$ particle at position $z$ unbinds to generate one $A$ particle at position $x$ and one $B$ particle at position $y$ and the non-reactant particles are again assumed to be at $\bm{q}$. Letting $\vmu_{s^-}^{\zeta}$ represent the state before the $\mathcal{R}_2$ reaction, with a substrate $C$ particle at $z$ and the remaining non-reactant particles at $\bm{q}$, the total potential of the system before $\mathcal{R}_2$ is defined to be
\begin{align*}
\Phi_2^{-,\gamma}(z,\bm{q})&= \Phi^{\gamma}(\bm{q}) +v_3(z)
   +\sum_{j=1}^{3}\int_{\mathbb{R}^d}u_{3,j}(z,x')\mu_{s^-}^{\zeta, j}(d x')-\frac{1}{\gamma}u_{3,3}(z,z),
\end{align*}
and we denote the total potential energy of the system after $\mathcal{R}_2$ by
\begin{align*}
   \Phi_2^{+,\gamma}(x,y,\bm{q})&= \Phi^{\gamma}(\bm{q})+ v_1(x)+v_2(y)+\sum_{j=1}^{3}\int_{\mathbb{R}^d}u_{1,j}(x,y')\mu_{s^-}^{\zeta, j}(d y')-\frac{1}{\gamma}u_{1,3}(x,z)\\
   &+\sum_{j=1}^{3}\int_{\mathbb{R}^d}u_{2,j}(y,x')\mu_{s^-}^{\zeta, j}(d x')-\frac{1}{\gamma}u_{2,3}(y,z)+\frac{1}{\gamma}u_{1,2}(x,y).
\end{align*}
The acceptance probability for $\mathcal{R}_2$ is thus
\begin{equation*}
   \pi_{2}^{\gamma} \big(x,y|z,\vmu_{s}^{\zeta}(dx')\bigr) = \min \left\{1, e^{-\bigl[\bigl(\Phi_2^{+,\gamma}(x,y,\bm{q})-u_{1,2}^{\gamma}(x,y)\bigr)-\Phi_2^{-,\gamma}(z,\bm{q})\bigr]}\right\},
\end{equation*}
so that the mean field limit of $\Phi_2^{-,\gamma}(z,\bm{q})$ is
\begin{equation*}
    \Phi_2^{-}(z,\vxi_s) := v_3(z) +\sum_{j=1}^{3}\int_{\mathbb{R}^d}u_{3,j}(z,x')\xi_{s}^{j}(d x'),
\end{equation*}
and the mean field limit of $\Phi_2^{+,\gamma}(x,y,\bm{q})$ is
\begin{equation*}
  \Phi_2^{+}(x,y,\xi_s) := v_1(x)+v_2(y)+\sum_{j=1}^{3}\int_{\mathbb{R}^d}u_{1,j}(x,y')\xi_{s}^{j}(d y')+\sum_{j=1}^{3}\int_{\mathbb{R}^d}u_{2,j}(y,x')\xi_{s}^{j}(d x').
\end{equation*}
As the pairwise potential term $U^{\gamma}(x,y)$ between the products converges to zero as $\gamma \rightarrow \infty$, the mean field limit of the acceptance probability for $\mathcal{R}_2$ is
\begin{equation*}
\pi_{2} \big(x,y|z,\vxi_{s}(dx')\bigr)= \min \left\{1, e^{-\bigl((\Phi_2^{+}(x,y,\vxi_s)-\Phi_2^{-}(z,\vxi_s)\bigr)}\right\}.
\end{equation*}

Finally, for completeness we give specific formulas for the placement densities of the reversible reaction $A+B \rightleftarrows C$ that are consistent with detailed balance holding, see \cite{IsaacsonDriftDB}. For simplicity, we let the placement density $m_{1}^{\eta}(z|x,y)$ for $\mathcal{R}_1$ take the form
\begin{align*}
    m_{1}^{\eta}(z|x,y) &= G_{\eta}\biggl(z - \bigr(\alpha x+(1-\alpha)y\bigr)\biggr),
    \end{align*}
    and as $\eta \rightarrow 0, m_{1}^{\eta}(z|x,y)$ then converges to
    \begin{align*}
         m_{1}(z|x,y) &= \delta\biggl(z - \bigr(\alpha x+(1-\alpha)y\bigr)\biggr),
    \end{align*}
   for $\alpha \in [0,1]$.
To ensure detailed balance of reaction fluxes at equilibrium, see \cite{IsaacsonDriftDB}, and also maintain symmetry for reversible reactions, we then chose
\begin{align*}
    m_{2}^{\zeta}(x,y|z)&= \frac{1}{Z_{1,2}^{\eta}} \mathbf{1}_{B_{\varepsilon}(\bm{0})}(x-y) G_{\eta}\biggl(z - \bigl(\alpha x+(1-\alpha)y\bigr)\biggr) e^{-u_{1,2}^{\gamma}(x,y)},
\end{align*}
where
\begin{align*}
    Z_{1,2}^{\eta} =
    \int_{\mathbb{R}^{2d}}\mathbf{1}_{B_{\varepsilon}(\bm{0})}(x-y) G_{\eta}\biggl(z - \bigl(\alpha x+(1-\alpha)y\bigr)\biggr) e^{-u_{1,2}^{\gamma}(x,y)} dxdy.
\end{align*}
$m_{2}^{\zeta}(x,y|z)$ then converges as a distribution to $m_{2}(x,y|z) \coloneqq \frac{1}{Z_{1,2}}\mathbf{1}_{B_{\varepsilon}(\bm{0})}(x-y)m_{1}(z|x,y)$ in the mean field limit as $\zeta \rightarrow 0$ with
\begin{align*}
    Z_{1,2}
    &= B_{\varepsilon}(\bm{0}).
\end{align*}
\end{example}

\begin{example}
Consider a system with four species, $A, B, C$ and $D$ that can undergo the reversible reaction $A+B \rightleftarrows C+D$. Let $\mathcal{R}_{1}$ be the forward reaction $A+B \rightarrow C+D$, where one $A$ particle at position $x$ and one $B$ particle at position $y$ react to generate one $C$ particle at position $w$ and one $D$ particle at position $z$.

Analogous to the last section, we represent the total potential energy of the system before $\mathcal{R}_1$ by
\begin{align*}
\Phi_1^{-,\gamma}(x,y,\bm{q})=& \Phi^{\gamma}(\bm{q})+v_1(x)+v_2(y)
   +\sum_{j=1}^{4}\int_{\mathbb{R}^d}u_{1,j}(x,y')\mu_{s^-}^{\zeta, j}(d y')-\frac{1}{\gamma}u_{1,1}(x,x)\\
   +&\sum_{j=1}^{4}\int_{\mathbb{R}^d}u_{2,j}(y,x') \mu_{s^-}^{\zeta, j}(d x')-\frac{1}{\gamma}u_{2,1}(y,x)-\frac{1}{\gamma}u_{2,2}(y,y),
\end{align*}
and the total potential energy of the system after $\mathcal{R}_1$ by
\begin{align*}
   \Phi_1^{+,\gamma}(w,z,\bm{q})&=\Phi^{\gamma}(\bm{q})+ v_3(w)+v_4(z)+\sum_{j=1}^{4}\int_{\mathbb{R}^d}u_{3,j}(w,x')\mu_{s^-}^{\zeta, j}(d x')-\frac{1}{\gamma}u_{3,1}(w,x)-\frac{1}{\gamma}u_{3,2}(w,y)\\
   &+\sum_{j=1}^{4}\int_{\mathbb{R}^d}u_{4,j}(z,x')\mu_{s^-}^{\zeta, j}(d x')-\frac{1}{\gamma}u_{4,1}(z,x)-\frac{1}{\gamma}u_{4,2}(z,y).
\end{align*}
The acceptance probability for $\mathcal{R}_1$ is thus
\begin{equation*}
   \pi_{1}^{\gamma} \big(w,z|x,y,\vmu_{s}^{\zeta}(dx')\bigr) = \min \left\{1, e^{-\bigl[\Phi_1^{+,\gamma}(w,z,\bm{q})-\Phi_1^{-,\gamma}(x,y,\bm{q})\bigr]}\right\}.
\end{equation*}
The mean field limit of $\Phi_1^{-,\gamma}(x,y,\bm{q})$ is
\begin{equation*}
    \Phi_1^{-}(x, y,\vxi_s) := v_1(x)+v_2(y)
   +\sum_{j=1}^{4}\int_{\mathbb{R}^d}u_{1,j}(x,y')\xi_{s}^{j}(d y')+\sum_{j=1}^{4}\int_{\mathbb{R}^d}u_{2,j}(y,x') \xi_{s}^{j}(d x'),
\end{equation*}
and the mean field limit of $\Phi_1^{+,\gamma}(w,z,\bm{q})$ is
\begin{equation*}
  \Phi_1^{+}(w,z,\vxi_s) := v_3(w)+v_4(z)+\sum_{j=1}^{4}\int_{\mathbb{R}^d}u_{3,j}(w,x')\xi_{s}^{j}(d x')+\sum_{j=1}^{4}\int_{\mathbb{R}^d}u_{4,j}(z,x')\xi_{s}^{j}(d x').
\end{equation*}
Therefore, the mean field limit of the acceptance probability for $\mathcal{R}_1$ simplifies to
\begin{equation*}
\pi_{1} \big(w,z|x,y,\vxi_{s}(dx')\bigr)= \min \left\{1, e^{-\bigl((\Phi_1^{+}(w,z,\vxi_s)-\Phi_1^{-}(x,y,\vxi_s)\bigr)}\right\}.
\end{equation*}

Let $\mathcal{R}_{2}$ be the backward reaction $C+D \rightarrow A+B$, where one $C$ particle at position $w$ and one $D$ particle at position $z$ react to generate one $A$ particle at position $x$ and one $B$ particle at position $y$. The total potential energy of the system before $\mathcal{R}_2$ is defined to be
\begin{align*}
\Phi_2^{-,\gamma}(w,z,\bm{q})&=\Phi^{\gamma}(\bm{q})+ v_3(w)+v_4(z)
   +\sum_{j=1}^{4}\int_{\mathbb{R}^d}u_{3,j}(w,x')\mu_{s^-}^{\zeta, j}(d x')-\frac{1}{\gamma}u_{3,3}(w,w)\\ &+\sum_{j=1}^{4}\int_{\mathbb{R}^d}u_{4,j}(z,x')\mu_{s^-}^{\zeta, j}(d x')-\frac{1}{\gamma}u_{4,3}(z,w)-\frac{1}{\gamma}u_{4,4}(z,z),
\end{align*}
and we denote the total potential energy of the system after $\mathcal{R}_2$ by
\begin{align*}
   \Phi_2^{+,\gamma}(x,y,\bm{q})=&\Phi^{\gamma}(\bm{q})+ v_1(x)+v_2(y)+\sum_{j=1}^{4}\int_{\mathbb{R}^d}u_{1,j}(x,y')\mu_{s^-}^{\zeta, j}(d y')-\frac{1}{\gamma}u_{1,3}(x,w)-\frac{1}{\gamma}u_{1,4}(x,z)\\
   &+\sum_{j=1}^{4}\int_{\mathbb{R}^d}u_{2,j}(y,x')\mu_{s^-}^{\zeta, j}(d x')-\frac{1}{\gamma}u_{2,3}(y,w)-\frac{1}{\gamma}u_{2,4}(y,z)+\frac{1}{\gamma}u_{1,2}(x,y).
\end{align*}
The acceptance probability for $\mathcal{R}_2$ is thus
\begin{equation*}
   \pi_{2}^{\gamma} \big(x,y|w,z,\vmu_{s}^{\zeta}(dx')\bigr) = \min \left\{1, e^{-\bigl[\Phi_2^{+,\gamma}(x,y,\bm{q})-\Phi_2^{-,\gamma}(w,z,\bm{q})\bigr]}\right\}.
\end{equation*}
The mean field limit of $\Phi_2^{-,\gamma}(z,\bm{q})$ is
\begin{equation*}
    \Phi_2^{-}(w,z,\vxi_s) := v_3(w)+v_4(z)
   +\sum_{j=1}^{4}\int_{\mathbb{R}^d}u_{3,j}(w,x')\xi_{s}^{j}(d x') +\sum_{j=1}^{4}\int_{\mathbb{R}^d}u_{4,j}(z,x')\xi_{s}^{j}(d x'),
\end{equation*}
and the mean field limit of $\Phi_2^{+,\gamma}(x,y,\bm{q})$ is
\begin{equation*}
  \Phi_2^{+}(x,y,\vxi_s) := v_1(x)+v_2(y)+\sum_{j=1}^{4}\int_{\mathbb{R}^d}u_{1,j}(x,y')\xi_{s}^{j}(d y')+\sum_{j=1}^{4}\int_{\mathbb{R}^d}u_{2,j}(y,x')\xi_{s}^{j}(d x').
\end{equation*}
Therefore, the mean field limit of the acceptance probability for $\mathcal{R}_2$ is
\begin{equation*}
\pi_{2} \big(x,y|w,z,\vxi_{s}(dx')\bigr)= \min \left\{1, e^{-\bigl((\Phi_2^{+}(x,y,\vxi_s)-\Phi_2^{-}(w,z,\vxi_s)\bigr)}\right\}.
\end{equation*}

Finally, we note that the placement densities, $m_{1}^{\eta}(w,z|x,y)$ and $m_{2}^{\eta}(x,y|w,z)$, and their respective mean field limits, $m_{1}(w, z|x,y)$ and $m_{2}(x,y|w,z)$, take the same forms as in Assumption \ref{A:PlacementDensity}.

\end{example}

\begin{example}
Consider a system with two species, $A$ and $B$ that can undergo the reversible reaction $A+A \rightleftarrows B$. Let $\mathcal{R}_{1}$ be the forward reaction $A+A \rightarrow B$, where two $A$ particles at $x$ and $y$ bind to generate one $B$ particle at position $z$.

Denote the specific two-body interaction between the substrates by
\begin{equation*}
    U^{\gamma}(x,y) = u_{1,1}^{\gamma}(x,y)=\frac{1}{\gamma}u_{1,1}(x,y).
\end{equation*}
We represent the total potential energy of the system before $\mathcal{R}_1$ by
\begin{align*}
\Phi_1^{-,\gamma}(x,y,\bm{q})=& \Phi^{\gamma}(\bm{q})+ v_1(x)+v_1(y)
   +\sum_{j=1}^{2}\int_{\mathbb{R}^d}u_{1,j}(x,y')\mu_{s^-}^{\zeta, j}(d y')-\frac{1}{\gamma}u_{1,1}(x,x)-\frac{1}{\gamma}u_{1,1}(x,y)\\
   +&\sum_{j=1}^{2}\int_{\mathbb{R}^d}u_{1,j}(y,x') \mu_{s^-}^{\zeta, j}(d x')-\frac{1}{\gamma}u_{1,1}(y,y),
\end{align*}
and the total potential energy of the system after $\mathcal{R}_1$ by
\begin{align*}
   \Phi_1^{+,\gamma}(z,\bm{q})=&\Phi^{\gamma}(\bm{q})+ v_2(z)+\sum_{j=1}^{2}\int_{\mathbb{R}^d}u_{2,j}(z,x')\mu_{s^-}^{\zeta, j}(d x')-\frac{1}{\gamma}u_{2,1}(z,x)-\frac{1}{\gamma}u_{2,1}(z,y).
\end{align*}
The acceptance probability for $\mathcal{R}_1$ is thus
\begin{equation*}
   \pi_{1}^{\gamma} \big(z|x,y,\vmu_{s}^{\zeta}(dx')\bigr) = \min \left\{1, e^{-\bigl[\Phi_1^{+,\gamma}(z,\bm{q})-\bigl(\Phi_1^{-,\gamma}(x,y,\bm{q})-u_{1,1}^{\gamma}(x,y)\bigr)\bigr]}\right\}.
\end{equation*}
The mean field limit of $\Phi_1^{-,\gamma}(x,y,\bm{q})$ is
\begin{equation*}
    \Phi_1^{-}(x, y,\vxi_s) := v_1(x)+v_1(y)
   +\sum_{j=1}^{2}\int_{\mathbb{R}^d}u_{1,j}(x,y')\xi_{s}^{j}(d y')+\sum_{j=1}^{2}\int_{\mathbb{R}^d}u_{1,j}(y,x') \xi_{s}^{j}(d x'),
\end{equation*}
and the mean field limit of $\Phi_1^{+,\gamma}(z,\bm{q})$ is
\begin{equation*}
  \Phi_1^{+}(z,\vxi_s) := v_2(z)+\sum_{j=1}^{2}\int_{\mathbb{R}^d}u_{2,j}(z,x')\xi_{s}^{j}(d x').
\end{equation*}
The pairwise potential term $U^{\gamma}(x,y)$ between the substrates converges to zero as $\gamma \rightarrow \infty$. Therefore, the mean field limit of the acceptance probability for $\mathcal{R}_1$ simplifies to
\begin{equation*}
\pi_{1} \big(z|x,y,\vxi_{s}(dx')\bigr)= \min \left\{1, e^{-\bigl((\Phi_1^{+}(z,\vxi_s)-\Phi_1^{-}(x,y,\vxi_s)\bigr)}\right\}.
\end{equation*}

Let $\mathcal{R}_{2}$ be the backward reaction $B \rightarrow A+A$, where one $B$ particle at position $z$ unbinds to generate two $A$ particles at $x$ and $y$. The total potential energy of the system before $\mathcal{R}_2$ is defined to be
\begin{align*}
\Phi_2^{-,\gamma}(z,\bm{q})=&\Phi^{\gamma}(\bm{q})+ v_2(z)
   +\sum_{j=1}^{2}\int_{\mathbb{R}^d}u_{2,j}(z,x')\mu_{s^-}^{\zeta, j}(d x')-\frac{1}{\gamma}u_{2,2}(z,z),
\end{align*}
and we denote the total potential energy of the system after $\mathcal{R}_2$ by
\begin{align*}
   \Phi_2^{+,\gamma}(x,y,\bm{q})=&\Phi^{\gamma}(\bm{q})+ v_1(x)+v_1(y)+\sum_{j=1}^{2}\int_{\mathbb{R}^d}u_{1,j}(x,y')\mu_{s^-}^{\zeta, j}(d y')-\frac{1}{\gamma}u_{1,2}(x,z)\\
   &+\sum_{j=1}^{2}\int_{\mathbb{R}^d}u_{1,j}(y,x')\mu_{s^-}^{\zeta, j}(d x')-\frac{1}{\gamma}u_{1,2}(y,z)+\frac{1}{\gamma}u_{1,1}(x,y).
\end{align*}
The acceptance probability for $\mathcal{R}_2$ is thus
\begin{equation*}
   \pi_{2}^{\gamma} \big(x,y|z,\vmu_{s}^{\zeta}(dx')\bigr) = \min \left\{1, e^{-\bigl[\bigl(\Phi_2^{+,\gamma}(x,y,\bm{q})-u_{1,1}^{\gamma}(x,y)\bigr)-\Phi_2^{-,\gamma}(z,\bm{q})\bigr]}\right\}.
\end{equation*}
The mean field limit of $\Phi_2^{-,\gamma}(z,\bm{q})$ is
\begin{equation*}
    \Phi_2^{-}(z,\vxi_s) := v_2(z)
   +\sum_{j=1}^{2}\int_{\mathbb{R}^d}u_{2,j}(z,x')\xi_{s}^{j}(d x'),
\end{equation*}
and the mean field limit of $\Phi_2^{+,\gamma}(x,y,\bm{q})$ is
\begin{equation*}
  \Phi_2^{+}(x,y,\vxi_s) := v_1(x)+v_1(y)+\sum_{j=1}^{2}\int_{\mathbb{R}^d}u_{1,j}(x,y')\xi_{s}^{j}(d y')+\sum_{j=1}^{2}\int_{\mathbb{R}^d}u_{1,j}(y,x')\xi_{s}^{j}(d x').
\end{equation*}
The pairwise potential term $U^{\gamma}(x,y)$  between the products converges to zero as $\gamma \rightarrow \infty$. Therefore, the mean field limit of the acceptance probability for $\mathcal{R}_2$ is
\begin{equation*}
\pi_{2} \big(x,y|z,\vxi_{s}(dx')\bigr)= \min \left\{1, e^{-\bigl((\Phi_2^{+}(x,y,\vxi_s)-\Phi_2^{-}(z,\vxi_s)\bigr)}\right\}.
\end{equation*}

Finally, we note that the placement densities, $m_{1}^{\eta}(z|x,y)$ and $m_{2}^{\eta}(x,y|z)$, and their respective mean field limits, $m_{1}(z|x,y)$ and $m_{2}(x,y|z)$, take similar forms as in Example \ref{E:ABC} and Assumption \ref{A:PlacementDensity}.
\end{example}


\section{Simulations}\label{S:Simulations}
We numerically solve the $A+B \rightleftarrows C$ reaction for a periodic one-dimensional system to compare our derived PIDEs and the underlying PBSRDD model. We subsequently call these PIDEs the mean field model (MFM). The MFM is solved using a Fourier spectral method. We discretize the PBSRDD model in space to obtain a (convergent) jump process approximation to the
stochastic process associated with the PBSRDD model via the Convergent Reaction Diffusion Master Equation (CRDME)~\cite{Isaacson2013, IsaacsonZhang2018} using the approach developed for systems with drift due to interaction potentials in~\cite{HeldmanCRDME,HeldmanThesis}.

We first present the model problem in Subsection \ref{S:numerics_example_def} and then discuss the discretization schemes we employed for the MFM and PBSRDD model in the next two subsections. Finally, we present the numerical results and demonstrate that for the total molar mass (i.e., the integral or $L^1$-norm of the molar concentration field) of the type $C$ particles, the mean field process gives an increasingly accurate approximation of the PBSRDD model as $\gamma$ increases. For $\gamma=1000$, the largest value of $\gamma$ that we consider, the means for the two models agree up to statistical error. We also compare with the purely-diffusive case to demonstrate that including drift induced by potential interactions affects the behavior of the system in non-trivial ways.


\subsection{Description of model problem} \label{S:numerics_example_def}
 Our model follows the general form of Example \ref{E:ABC}, with some modifications. We restrict the reaction system to the periodic domain $\Omega = [0,L]$ with $L = 2 \pi$. As in the example, we prescribe harmonic two-body potentials between each pair of particles which we choose as
\begin{equation}
u_{s,s'}(x,y) = u_{s,s'}(|x-y|) = \kappa \max\left\{0,3(r_s + r_{s'}) - |x-y|\right\}^2,
\end{equation}
where $s,s'\in\{A,B,C\}$. Here and throughout this section, $|x-y|$ denotes the periodic distance between $x,y\in\Omega$, i.e.,
\begin{equation*}
  |x-y| = \min\{|x-y|,L-|x-y|\}.
\end{equation*}

 The parameters $r_j$, which control the interaction distance, and $\kappa$, which controls the interaction strength, are chosen large enough to make a clear contrast with the $\kappa = 0$ (i.e., no potentials) case. Note that for simplicity, we do not include single-body potentials in this model.

 The drift-diffusion transport operator for each particle type $s\in\{A,B,C\}$ is also scaled by a diffusion constant $D_s$. We emphasize that this diffusion constant plays a slightly different role than in earlier sections; in particular, it scales \textit{both} the drift and diffusion terms (see \eqref{eq:hoppingrate}
 and \eqref{eq:MFM_PDEs}), and compare the transport operators in the latter system of PIDEs with, e.g., those in \eqref{Eq:ThreeSpeciesReversibleExampleLimit}).

For reactions, we follow Example \ref{E:ABC} except that we replace the Doi reaction kernel $1_{B_\varepsilon(\mathbf{0})}(x - y)$ with a normalized Gaussian
\begin{equation*}
K(x,y) = \frac{1}{Z}\frac{e^{-\frac{|x-y|^2}{2\sigma^2}}}{\sqrt{2\pi\sigma^2}},
\end{equation*}
with $\sigma$ the kernel width and $Z$ a normalization constant,
\begin{equation*}
Z = \frac{1}{\sqrt{2\pi\sigma^2}} \int_{0}^{2 \pi}  e^{-\frac{|x-y|^2}{2\sigma^2}} dx.
\end{equation*}
We also replace the placement density $m_1(z|x,y)$ with the combination of $\delta$-functions: \begin{equation}\label{eq:deltafunctionplacement}m_1(z|x,y) = \frac{1}{2}\delta(x - z) + \frac{1}{2}\delta(y-z),\end{equation} so that, e.g., in the event of an $A+B\to C$ reaction the product $C$ is placed at the location of the $A$ or the $B$ with equal probability $\frac{1}{2}$. The backward reaction placement density is likewise modified to incorporate our changes to the forward placement density and the reaction kernel, i.e., $$m_2(x,y|z) = \frac{1}{Z_{AB}} K(x,y) m_1(z|x,y) e^{-u^\gamma_{A,B}(x,y)},$$ where $$Z_{AB} = \int_{0}^{2\pi} K(x,0)e^{-u_{A,B}^\gamma(x,0)}dx = \int_{0}^{2\pi} K(0,y)e^{-u_{A,B}^\gamma(0,y)}dy,$$ recalling the notation $u^{\gamma}_{s,s'}(x,y) = \frac{u_{s,s'}(x,y)}{\gamma}$ for $s,s'\in\{A,B,C\}$. We note that, in the following simulations, we do not regularize this placement density, i.e., in the actual numerical implementation we work with the $\delta$ function densities directly.

We choose parameters $\lambda$ and $\mu$ which control the relative rates at which forward and backward reactions occur. The rates are chosen so that prior to applying the detailed-balance enforcing rejection-acceptance mechanism, the forward rate for an $A$ at $x\in\Omega$ and a $B$ at $y\in\Omega$ to react is $\lambda K(x,y)$ and the rate for a $C$ at $z$ to unbind is $\mu$ in the particle model.

Finally, we specify the values of the various parameters described above: $L = 2\pi$, $r_A = r_B = 0.05, r_C = 0.1,$ $D_A = D_B=0.25$, $D_C=0.5$, $\sigma = 0.15, \lambda = 1,$ and $\mu = 0.05$. We choose the potential strength parameter $\kappa=200$, also making qualitative comparisons with the pure diffusion ($\kappa=0$) case. Initial conditions are set proportionally to \begin{align*}
    A(x,0) &=e^{-5 |x-0.75\pi|^2},\\
    B(x,0) &= e^{-5|x-1.25\pi|^2},\\
    C(x,0) &= 0,
\end{align*}
where again we use periodic distances.

\subsection{Discretization of particle models.}

To solve the particle model, we use the CRDME, a convergent spatial discretization of the forward Kolmogorov equation associated with the PBSRDD model \cite{Isaacson2013, IsaacsonZhang2018}. The CRDME corresponds to the forward equation for a system of continuous-time jump processes on a mesh, and we therefore simulate the particle system via simulations of these jump processes using optimized versions of the stochastic simulation algorithm (SSA), also known as Gillespie's method or Kinetic Monte Carlo \cite{Gillespie1976}. The PBSRDD particles' Brownian motions are then approximated by continuous-time random walks on a grid and their reactive interactions by jump processes that depend on the relative positions of reactants on the mesh. As discussed in \cite{Isaacson2013, IsaacsonZhang2018, HeldmanThesis, HeldmanCRDME}, statistics obtained from simulations of the CRDME should then converge to those of the underlying PBSRDD model as the mesh spacing is taken to zero.

For these simulations, we use a uniform mesh with nodes $\{x_i\}_{i=1}^N \subseteq \Omega$, where $x_i = (i-1)h$, $i=1,...,N$, and $h = \frac{2\pi}{N}$. We denote the compartments, or voxels, that particles hop between by $V_i = (x_i-\frac{h}{2}\mod 2\pi,x_i+\frac{h}{2})$ for $i = 1,2,\dots,N$. To initialize the particle positions at the start of each simulation, we first choose the number of $A$ and $B$ particles as $\frac{\gamma}{2}$, for $\gamma \in \{50, 100, 150, 200, 250, 350, 500, 1000\}$ . Then, we sample the position of each individual particle from the (unnormalized) discrete distributions $\{A(x_i,0)\}_{i=1}^N$, $\{B(x_i,0)\}_{i=1}^N$. We note that our choice of initial distribution implies that the initial error between the mean field model and the mean of the particle model is zero, up to discretization error.

Let $\xi_s^\gamma$ be concentration measures for the nodal locations of the particles of type $s \in \{A,B,C\}$, so that if there are $s_i$ particles of type $s$ in voxel $V_i$, $$\xi^\gamma_s = \frac{1}{\gamma}\sum_{i=1}^N s_i\delta(x_i - x).$$ Then, the total potential difference induced by a hop of one particle of type $s$ from voxel $V_i$ to voxel $V_j$ is given by:
$$\Delta_{j|i}^{s} = \sum_{s'\in\{A,B,C\}}\biggl(\int_{\mathbb{R}} u_{s,s'}(x_j,y)\xi_s^\gamma(dy) - \int_{\mathbb{R}} u_{s,s'}(x_i,y)\xi_s^\gamma(dy) \biggr)- u_{s,s}^\gamma(x_j,x_i) +  u_{s,s}^\gamma(x_i,x_i),$$
where the latter terms remove self-interactions. The hopping rate for each particle of species $s\in \{A,B,C\}$ in voxel $V_i$ to a neighboring voxel $V_j$ is then given by \begin{equation}\label{eq:hoppingrate}\frac{D_s}{h^2} \frac{\Delta_{j|i}^s}{e^{\Delta_{j|i}^s} - 1}.\end{equation} The above formula can be derived by using a specific combination of quadrature rules on the transport operator appearing in the forward Kolmogorov equation of the particle model, as described in~\cite{HeldmanThesis, HeldmanCRDME}.

Proposal reaction rates between a molecule of type $A$ in voxel $V_i$ and a molecule of type $B$ in voxel $V_j$ are computed as $\lambda K^\gamma(x_i,y_j) = \frac{K(x_i,y_j)}{\gamma}$. The resulting $C$ is placed in voxel $V_i$ or $V_j$ with probability $\frac{1}{2}$ each. We then accept the proposed reaction with probability computed by the formula given in Example \ref{E:ABC}, where $x$, $y$, and $z$ are replaced with the proposed mesh nodes $x_i,y_j,z_k$ and particle positions $\mathbf{q}$ are replaced with mesh nodes associated with their respective voxel sites.

The total rate of proposed unbinding reactions in the CRDME model at each voxel site $V_k$ is given by the discrete marginal integral of the reaction kernel,
\begin{equation*}
\frac{\mu}{Z_{AB}}\sum_{i=1}^N hK(x_i,z_k)e^{-u_{A,B}^\gamma(x_i,z_k)} = \frac{\mu}{Z_{AB}}\sum_{j=1}^N hK(z_k,y_j)e^{-u_{A,B}^\gamma(z_k,y_j)}.
\end{equation*}
We note that for our chosen parameters, the above sum is essentially $\mu$ (i.e., to numerical precision). After the unbinding of a $C$ particle at $V_k$ is proposed, a particle of type $A$ or of type $B$ is tentatively placed in $V_k$ with probability $\frac{1}{2}$. The location of the other product particle is proposed by sampling the (unnormalized) discrete distributions $\{K(x_i,z_k)e^{-u_{A,B}^\gamma(x_i,z_k)}\}_{i=1}^N = \{K(z_k,y_j)e^{-u_{A,B}^\gamma(z_k,y_j)}\}_{j=1}^N$ over the voxel sites. Again, after reactions are proposed using the rates described above, they are accepted or rejected using the mechanism described in Example \ref{E:ABC}, with particle positions localized to mesh nodes.

Specifying the reaction and transport rates as we have just described preserves detailed balance and convergence of the CRDME to the underlying particle model~\cite{HeldmanThesis,HeldmanCRDME}. We note that, although in the simulation methodology described here potentials and reaction rates are always evaluated at mesh nodes, it may improve CRDME convergence in some cases (e.g., discontinuous reaction kernels) to employ other discretization strategies (e.g., involving voxel averages of the reaction kernel as done in~\cite{Isaacson2013, IsaacsonZhang2018}). 

\subsection{Discretization of MFM} Let $\mathbf{S}(x,t) = (A(x,t), B(x,t), C(x,t))$, and define the integral operator $(u*f)(x,t)$ for a function $f(x,t)$ by
\begin{equation*}
  (u * f)(x,t) = \int_{\Omega} u(x,y) f(y,t) \, dy.
\end{equation*}
The MFM for the reversible reaction is given by the system of PIDEs
\begin{equation}\label{eq:MFM_PDEs}
\begin{split}
    \frac{\partial A}{\partial t}(x,t) &= D_A\nabla\cdot \biggl(\nabla A(x,t) + A(x,t)\!\!\!\!\!\!\sum_{S \in \{A,B,C\}} (\nabla u_{A,S}*S)(x,t) \biggr)\\
    &\phantom{=}- \frac{\lambda}{2}A(x, t)\int_{\mathbb{R}^{d}} K(x, y)\left(\pi_{1}\bigl(x|x,y, \mathbf{S}(x',t)dx'\bigr) + \pi_{1}\bigl(y|x,y, \mathbf{S}(x',t)dx'\bigr)\right) B(y,t)dy\\
    &\phantom{=}+ \frac{\mu}{2}C(x, t)\int_{\mathbb{R}^{d}} K(x, y)\pi_{2}\bigl(x,y|x, \mathbf{S}(x',t)dx'\bigr)dy\\
&\phantom{=}+ \frac{\mu}{2}\int_{\mathbb{R}^{d}} K(x, y)\pi_{2}\bigl(x,y|y, \mathbf{S}(x',t)dx'\bigr)C(y, t) d y,\\
    \frac{\partial B}{\partial t}(y,t) &= D_B\nabla\cdot\biggl( \nabla B(y,t) +  B(y,t)\!\!\!\!\!\!\sum_{S \in \{A,B,C\}} (\nabla u_{B,S}*S)(y,t)\biggr)\\
        &\phantom{=}- \frac{\lambda}{2}B(y, t)\int_{\mathbb{R}^{d}} K(x, y)\left(\pi_{1}\bigl(x|x,y, \mathbf{S}(x',t)dx'\bigr) + \pi_{1}\bigl(y|x,y, \mathbf{S}(x',t)dx'\bigr)\right) A(x,t)dx\\
    &\phantom{=}+ \frac{\mu}{2}C(y, t)\int_{\mathbb{R}^{d}} K(x, y)\pi_{2}\bigl(x,y|y, \mathbf{S}(x',t)dx'\bigr)dx\\
    &\phantom{=}+ \frac{\mu}{2}\int_{\mathbb{R}^{d}} K(x, y)\pi_{2}\bigl(x,y|x, \mathbf{S}(x',t)dx'\bigr)C(x, t)dx,\\
    \frac{\partial C}{\partial t}(z,t) &= D_C\nabla\cdot \biggl(\nabla C(z,t) + C(z,t)\!\!\!\!\!\!\sum_{S \in \{A,B,C\}} (\nabla u_{C,S}*S)(z,t)\biggr)\\
    &\phantom{=}+\frac{\lambda}{2}A(z, t) \int_{\mathbb{R}^{d}} K(z, y)\left( \pi_{1}\bigl(z|z,y, \mathbf{S}(x',t)dx'\bigr) \right) B(y, t)dy \\
    &\phantom{=}+\frac{\lambda}{2}B(z, t)\int_{\mathbb{R}^{d}} K(x,z)\pi_{1}\bigl(z|x,z, \mathbf{S}(x',t)dx'\bigr)A(x, t)dx\\
    &\phantom{=}-\frac{\mu}{2}C(z,t) \int_{\mathbb{R}^{d}} \left( K(z, x)\pi_{2}\bigl(z,x|z, \mathbf{S}(x',t)dx'\bigr) + K(x,z)\pi_{2}\bigl(x,z|z, \mathbf{S}(x',t)dx'\bigr) \right) dx
    \end{split}
\end{equation}
where $A$, $B$, and $C$ are the mean field molar concentration fields for the corresponding particle types.  As for the particle model, the acceptance-rejection factors $\pi$ are as described in Example \ref{E:ABC}.

The above PIDEs are solved using a Fourier collocation method \cite{Hesthaven2007} for the spatial discretization, with collocation points $x_i = \frac{iL}{N}, i = 0, \cdots, N-1$. We use a total of $N = 2^9$ basis functions $\{e^{inx}\}_{n=0}^{2^9-1}$ to represent the concentration fields for each of the three species. We convert between Fourier representations and values at collocation points using the \texttt{SciPy} fast Fourier transforms functions, and approximate the drift and reaction integral terms using the trapezoidal rule centered at the collocation points. 

For the resulting spatially discretized reaction-drift-diffusion ODEs, the diffusion and drift terms are stiff whereas the reaction terms are non-stiff. We therefore adopt a one-step implicit-explicit (IMEX) Euler method to solve the system of ODEs arising from our spatial discretization, treating the reaction terms explicitly (forward Euler) and the transport operator implicitly (backward Euler). For the implicit terms, the nonlinear system of equations is solved at each timestep using the \texttt{SciPy} Newton-Krylov solver. To ensure convergence of the nonlinear solver, we dynamically reduce our timestep from an empirically chosen (for accuracy) maximum timestep of $10^{-4}$.

\subsection{Numerical results}
We present the results of our numerical studies in Figure \ref{fig:comparison}. The first three subfigures (top left, top right, and bottom left) each compare the molar mass for the $C$ particles in the prelimit and limiting processes, denoted by $\bar{C}(t)$ for the MFM and $\bar{C}^\gamma(t)$ for the CRDME model, over the time interval $[0,40]$. The bottom-right subfigure compares the corresponding spatial distributions $C(x,t)$ and $C^\gamma(x,t)$ at time $t=4$.

For the MFM, the molar mass $$\bar{C}(t) = \int_{0}^{2\pi} C(x,t)dx$$ is approximated from the spatial distribution $C(x,t)$ at a series of discrete times $t$ by applying the trapezoidal rule to the numerically computed mean field solution. For the CRDME model, we compute $\bar{C}^\gamma(t)$ by first averaging the number of particles in each voxel over 280,000 simulations to obtain $\bar{C}^\gamma_i(t)$, $i=1,...,N$, at a series of discrete times $t$. Then, the concentration fields $C^\gamma(x,t)$ shown in the bottom-right panel of Figure \ref{fig:comparison} can be computed from $\bar{C}^\gamma_i(t)$ as the piecewise-linear interpolant of the grid function $C^\gamma(x_i,t) = \frac{\bar{C}^\gamma_i(t)}{h\gamma}$. Finally, the molar mass $\bar{C}^{\gamma}(t)$ is computed from $C^\gamma(x,t)$ as $$\bar{C}^{\gamma}(t) = \sum_{i=1}^N hC^\gamma(x_i,t).$$

\begin{figure}[ht]
\hspace*{-1cm}
\centering
\begin{subfigure}[b]{.5\textwidth}
  \centering
\centerline{\includegraphics[width=\linewidth]{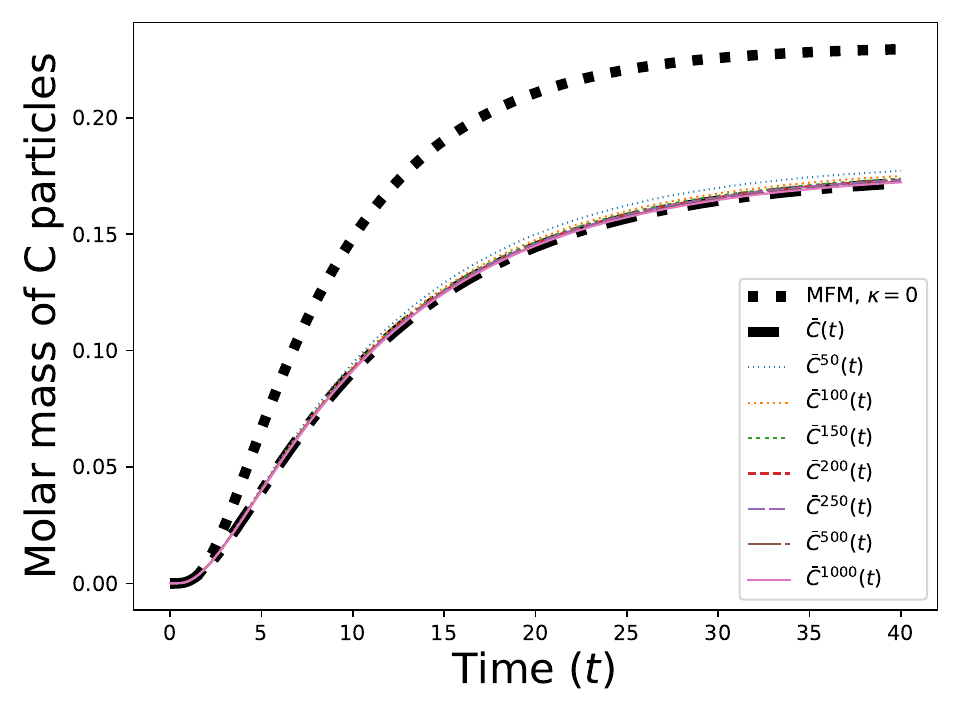}}
  \caption{Convergence of the CRDME molar mass $\bar{C}^\gamma(t)$ to the MFM molar mass $\bar{C}(t)$ as $\gamma$ increases over the time interval $[0,40]$, compared with the MFM molar mass in the no-potentials case ($\kappa = 0$).}
\end{subfigure}%
    \hfill
\begin{subfigure}[b]{.5\textwidth}
  \centering
\centerline{\includegraphics[width=\linewidth]{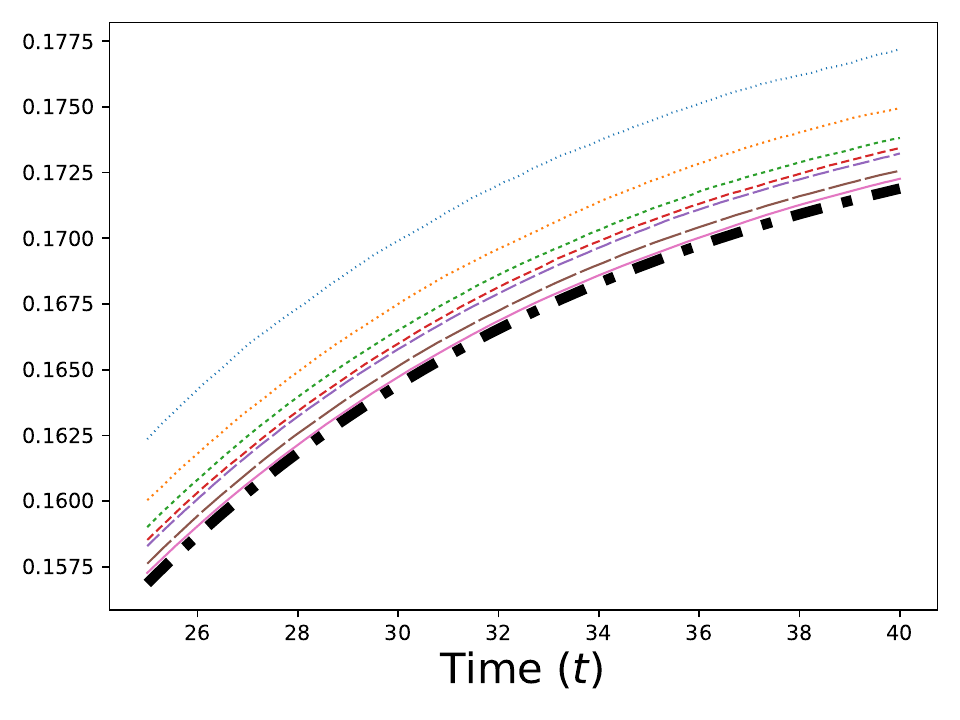}}
  \caption{Convergence of the CRDME molar mass $\bar{C}^\gamma(t)$ to the MFM molar mass $\bar{C}(t)$ as $\gamma$ increases over the restricted time interval $[25,40]$.\newline}
\end{subfigure}\newline
\hspace*{-1cm}
\begin{subfigure}[b]{.5\textwidth}
  \centering
\centerline{\includegraphics[width=\linewidth]{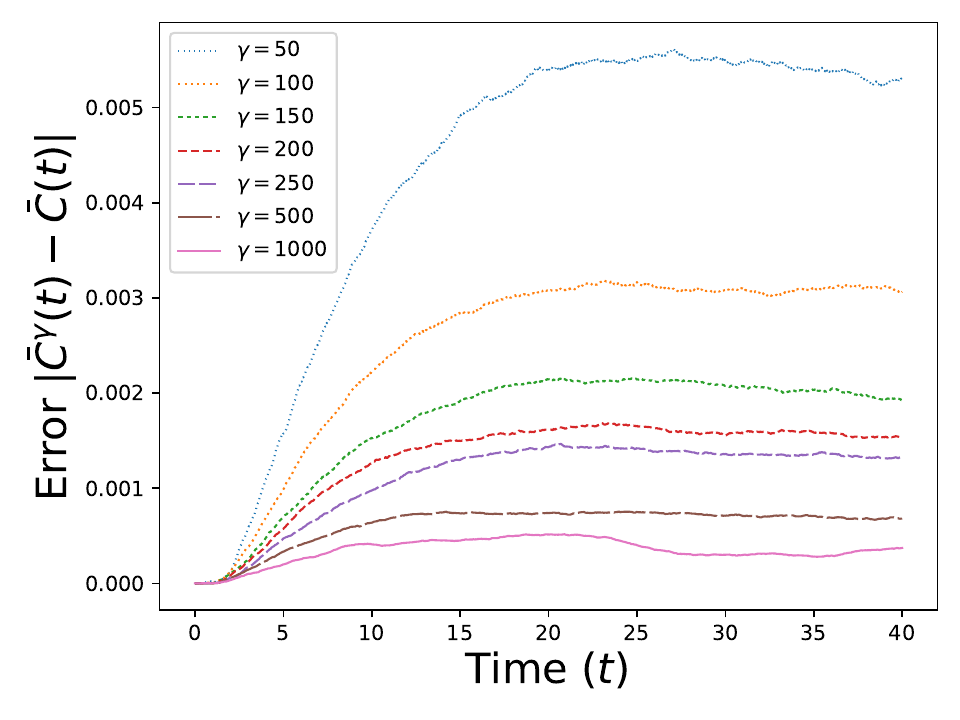}}
  \caption{Absolute differences $|\bar{C}^\gamma(t) - \bar{C}(t)|$ of the CRDME molar mass $\bar{C}^\gamma(t)$ and the MFM molar mass $\bar{C}(t)$ on the interval $[0,40]$ for different values of $\gamma$.\newline}
\end{subfigure}%
    \hfill
\begin{subfigure}[b]{.5\textwidth}
  \centering
\centerline{\includegraphics[width=\linewidth]{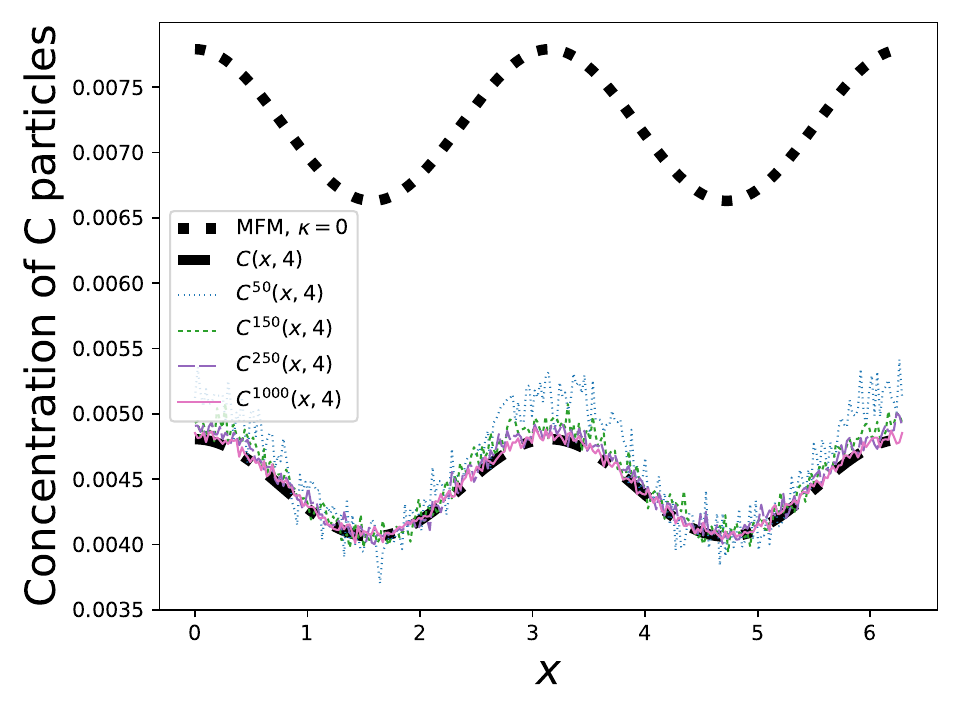}}
  \caption{Comparison of the spatial distribution of the MFM, $C(x,t)$, to the spatial distribution of the CRDME, $C^\gamma(x,t)$, at the time $t = 4$, along with the spatial distribution of the MFM in the purely-diffusive (i.e. no potentials) case.}
\end{subfigure}%
\caption{Comparison of the CRDME and MFM solutions. The CRDME data are obtained by averaging the results of 280,000 simulations.}
\label{fig:comparison}
\end{figure}

Starting with the top-left panel in Figure \ref{fig:comparison}, we observe that, as expected, the CRDME molar masses $\bar{C}^\gamma(t)$ converge to the MFM molar masses as $\gamma\to\infty$. We have also included for comparsion the mean-field solution in the no-potentials case (i.e., $\kappa = 0$), observing that over time, the repulsive pair potential force reduces the effective forward reaction rate (relative to the backward), resulting in a significantly smaller molar mass at the final time.

The top-right and bottom-left panels further compare the molar masses $\bar{C}(t)$ and $\bar{C}^\gamma(t)$ for the MFM and CRDME models in the $\kappa = 200$ case; the top-right panel shows a zoomed-in version of the top-left panel, with the $\kappa = 0$ solution removed and the time interval restricted to $[25,40]$. We see more clearly the convergence of the CRDME molar masses $\bar{C}^\gamma(t)$ to the MFM molar masses $\bar{C}(t)$. Looking directly at the errors in the bottom-left panel, computed as $|\bar{C}(t) - \bar{C}^\gamma(t)|$ for each time $t$, we see that the maximum errors range from $\approx .005$ ($\approx$ 3\%) for $\gamma = 50$ to $\approx .00025$ ($\approx$ .15\%) for $\gamma = 1000$.

Finally, in the bottom right figure, we compare the spatial distribution $C^\gamma(x,t)$ for the CRDME model with the MFM spatial distribution $C(x,t)$. Although the CRDME spatial distributions appear significantly noisier than the molar masses due to the relatively small number of particles in each voxel, we can still see that the MFM gives a good fit for the particle model, especially for larger values of $\gamma$. We have also once again included for comparison the $\kappa = 0$ MFM spatial distribution. 

\section{Proof of Theorem \ref{T:MainTheorem}.}\label{S:ProofMainTheorem}

Without loss of generality, we assume that $\tilde{L}=0$ in this section. The case when $\tilde{L}>0$ follows by similar arguments as we now give in the $\tilde{L}=0$ case.

To rigorously determine the large population limit of the MVSP, we use the martingale problem approach for studying solutions to stochastic differential equations developed by Stroock and Varadhan \cite{EthierKurtz1986,StroockVaradhan2006}. The proof is divided into four subsections. In Subsection \ref{SS:PathLevelDescription}, we provide the path level description of $\mu_{t}^{\zeta, j}$, and we derive equations for its expectation in Subsection \ref{SS:TakingExpectations}. Tightness and identification of the limit are presented in Subsection \ref{SS:LimitIdentification}. 
Lastly, we prove in Subsection \ref{SS:Uniqueness} that the limit equation has a unique solution. Collectively, these results imply Theorem \ref{T:MainTheorem}, see, for example, the proof of Theorem 5.5 in~\cite{IsaacsonMa2022} for details.

\subsection{Path level description.}\label{SS:PathLevelDescription}
For a test function $f \in C_{b}^{2}(\mathbb{R}^{d})$ and for each species $j=1, \cdots, J$, we obtain the coupled system
\begin{align*}
 \langle f, & \mu_{t}^{\zeta, j}\rangle
=\langle f, \mu_{0}^{\zeta, j}\rangle+\frac{1}{\gamma}\sum_{i \geq 1} \int_{0}^{t} 1_{\{i \leq \gamma\langle 1, \mu_{s^-}^{\zeta, j}\rangle\}} \sqrt{2 D_{j}} \frac{\partial f}{\partial Q}(H^{i}(\gamma\mu_{s^-}^{\zeta, j})\bigr) d W_{s}^{i}\\
&\phantom{=} + \frac{1}{\gamma}\int_{0}^{t} \sum_{i=1}^{\gamma\langle 1, \mu_{s^-}^{\zeta, j}\rangle} \biggl(D_{j} \frac{\partial^{2} f}{\partial Q^{2}}\bigl(H^{i}(\gamma\mu_{s^-}^{\zeta, j})\bigr) - \frac{\partial f}{\partial Q}\bigl(H^{i}(\gamma\mu_{s^-}^{\zeta, j})\bigr) \cdot \nabla v_j\bigl(H^{i}(\gamma\mu_{s^-}^{\zeta, j})\bigr)\biggr) d s \\
&\phantom{=}  - \frac{1}{\gamma}\int_{0}^{t} \sum_{i=1}^{\gamma\langle 1, \mu_{s^-}^{\zeta, j}\rangle}\biggl(\frac{\partial f}{\partial Q}\bigl(H^{i}(\gamma\mu_{s^-}^{\zeta, j})\bigr) \cdot \frac{1}{\gamma}\nabla \sum_{j'=1, j' \neq j}^{J} \sum_{k=1}^{\gamma\langle 1, \mu_{s^-}^{\zeta, j'}\rangle}u_{j,j'}\bigl(\|H^{i}(\gamma\mu_{s^-}^{\zeta,j}) - H^{k}(\gamma\mu_{s^-}^{\zeta, j'})\|\bigr)\biggr)d s\\
&\phantom{=}  - \frac{1}{\gamma}\int_{0}^{t} \sum_{i=1}^{\gamma\langle 1, \mu_{s^-}^{\zeta, j}\rangle}\biggl(\frac{\partial f}{\partial Q}\bigl(H^{i}(\gamma\mu_{s^-}^{\zeta, j})\bigr) \cdot \frac{1}{\gamma}\nabla \sum_{k=1, k \neq i}^{\gamma\langle 1, \mu_{s^-}^{\zeta, j}\rangle}u_{j,j}\biggl(\|H^{i}(\gamma\mu_{s^-}^{\zeta,j}) - H^{k}(\gamma\mu_{s^-}^{\zeta, j})\|\biggr)\biggr)d s\\
&\phantom{=}  +\sum_{\ell=1}^{L} \int_{0}^{t} \int_{\mathbb{I}^{(\ell)}} \int_{\mathbb{Y}^{(\ell)}} \int_{\mathbb{R}_{+}^{3}}\biggl(\langle f, \mu_{s^-}^{\zeta, j}-\frac{1}{\gamma} \sum_{r=1}^{\alpha_{\ell j}} \delta_{H^{i_{r}^{(j)}}(\gamma \mu_{s^-}^{\zeta, j})}+\frac{1}{\gamma} \sum_{r=1}^{\beta_{\ell j}} \delta_{y_{r}^{(j)}}\rangle-\langle f, \mu_{s^-}^{\zeta, j}\rangle\biggr) 1_{\{\bm{i} \in \Omega^{(\ell)}(\gamma\mu_{s^-}^{\zeta})\}} \\
&\phantom{=} \qquad \times 1_{\{\theta_1 \leq K_{\ell}^{\gamma}\bigl(\mathcal{P}^{(\ell)}(\gamma\mu_{s^-}^{\zeta}, \bm{i})\bigr)\}} 1_{\{\theta_2 \leq m_{\ell}^{\eta}(\bm{y}|\mathcal{P}^{(\ell)}(\gamma\mu_{s^-}^{\zeta}, \bm{i}))\}} 1_{\{\theta_3 \leq \pi_{\ell}^{\gamma} \bigl(\bm{y}|\bm{x}, \vmu_{s^-}^{\zeta}(dx')\bigr)\}} d N_{\ell}(s, \bm{i, y}, \theta_1, \theta_2, \theta_3)\\
&=\langle f, \mu_{0}^{\zeta, j}\rangle+\frac{1}{\gamma}\sum_{i \geq 1} \int_{0}^{t} 1_{\{i \leq \gamma\langle 1, \mu_{s^-}^{\zeta, j}\rangle\}} \sqrt{2 D_{j}} \frac{\partial f}{\partial Q}(H^{i}(\gamma\mu_{s^-}^{\zeta, j})\bigr) d W_{s}^{i}\\
&\phantom{=}  + \frac{1}{\gamma}\int_{0}^{t} \sum_{i=1}^{\gamma\langle 1, \mu_{s^-}^{\zeta, j}\rangle} \biggl(D_{j} \frac{\partial^{2} f}{\partial Q^{2}}\bigl(H^{i}(\gamma\mu_{s^-}^{\zeta, j})\bigr) - \frac{\partial f}{\partial Q}\bigl(H^{i}(\gamma\mu_{s^-}^{\zeta, j})\bigr) \cdot \nabla v_j \bigl(H^{i}(\gamma\mu_{s^-}^{\zeta, j})\bigr)\biggr) d s \\
&\phantom{=}  - \frac{1}{\gamma}\int_{0}^{t} \sum_{i=1}^{\gamma\langle 1, \mu_{s^-}^{\zeta, j}\rangle}\biggl(\frac{\partial f}{\partial Q}\bigl(H^{i}(\gamma\mu_{s^-}^{\zeta, j})\bigr) \cdot \frac{1}{\gamma}\nabla \sum_{j'=1}^{J} \sum_{k=1}^{\gamma\langle 1, \mu_{s^-}^{\zeta, j'}\rangle}u_{j,j'}\bigl(\|H^{i}(\gamma\mu_{s^-}^{\zeta,j}) - H^{k}(\gamma\mu_{s^-}^{\zeta, j'})\|\bigr)\biggr)d s\\
&\phantom{=}  + \frac{1}{\gamma}\int_{0}^{t}\sum_{i=1}^{\gamma\langle 1, \mu_{s^-}^{\zeta, j}\rangle} \biggl(\frac{\partial f}{\partial Q}\bigl(H^{i}(\gamma\mu_{s^-}^{\zeta, j})\bigr) \cdot \frac{1}{\gamma}\nabla u_{j,j}\|0\|\biggr)ds\\
&\phantom{=}  +\frac{1}{\gamma} \sum_{\ell=1}^{L} \int_{0}^{t} \int_{\mathbb{I}^{(\ell)}} \int_{\mathbb{Y}^{(\ell)}} \int_{\mathbb{R}_{+}^{3}}\biggl(\langle f, -\sum_{r=1}^{\alpha_{\ell j}} \delta_{H^{i_{r}^{(j)}}(\gamma \mu_{s^-}^{\zeta, j})} + \sum_{r=1}^{\beta_{\ell j}} \delta_{y_{r}^{(j)}}\rangle\biggr) 1_{\{\bm{i} \in \Omega^{(\ell)}(\gamma\mu_{s^-}^{\zeta})\}} \\
&\phantom{=} \quad \times 1_{\{\theta_1 \leq K_{\ell}^{\gamma}\bigl(\mathcal{P}^{(\ell)}(\gamma\mu_{s^-}^{\zeta}, \bm{i})\bigr)\}} 1_{\{\theta_2 \leq m_{\ell}^{\eta}(\bm{y}|\mathcal{P}^{(\ell)}(\gamma\mu_{s^-}^{\zeta}, \bm{i}))\}}
1_{\{\theta_3 \leq \pi_{\ell}^{\gamma} \bigl(\bm{y}|\bm{x}, \mu_{s^-}^{\zeta}(dx')\bigr)\}}d N_{\ell}(s, \bm{i, y}, \theta_1, \theta_2, \theta_3).\numberthis \label{systemfull}
\end{align*}
The exchangeability of the sum, the differentiation, and the Lebesgue integral here and in the next subsection are justified by the fact that for fixed $\zeta, \gamma\langle 1, \mu_{s^-}^{\zeta, j}\rangle$ is finite by Assumption \eqref{L:uniformboundofmeasures}, and that $f$ and its partial derivatives are uniformly bounded.

\subsection{Taking expectations.}\label{SS:TakingExpectations} By taking expectation on \eqref{systemfull}, we have
   \begin{align*}
    \mathbb{E}[\langle f, \mu_{t}^{\bm{\zeta}, j}\rangle]
      &=\mathbb{E}[\langle f, \mu_{0}^{\zeta, j}\rangle]+\mathbb{E}\biggl[\int_{0}^{t} \int_{\mathbb{R}^{d}} \frac{1}{\gamma} \sum_{i=1}^{\gamma\langle 1, \mu_{s^-}^{\zeta, j}\rangle}\biggl( D_{j} \frac{\partial^{2} f}{\partial Q^{2}}(x) - \frac{\partial f}{\partial Q}(x) \cdot \nabla v_j(x) \biggr) \delta_{H^{i}(\gamma \mu_{s^-}^{\zeta, j})}(d x) d s\biggr] \\
    &\phantom{=} -\mathbb{E}\biggl[\frac{1}{\gamma}\int_{0}^{t} \int_{\mathbb{R}^{d}}\sum_{i=1}^{\gamma\langle 1, \mu_{s^-}^{\zeta, j}\rangle}\frac{\partial f}{\partial Q}(x)\biggl( \int_{\mathbb{R}^{d}}\frac{1}{\gamma}\nabla \sum_{j'=1}^{J} \sum_{k=1}^{\gamma\langle 1, \mu_{s^-}^{\zeta, j'}\rangle}u_{j,j'}(\norm{x-y})\delta_{H^{k}(\gamma \mu_{s^-}^{\zeta, j'})}(d y)\biggr)\\
    &\phantom{=} \qquad\qquad\qquad\qquad\qquad\qquad\qquad\qquad\qquad\qquad\qquad\qquad\qquad \times \delta_{H^{i}(\gamma \mu_{s^-}^{\zeta, j})}(d x) d s\biggr]\\
    &\phantom{=}  +\mathbb{E}\biggl[\frac{1}{\gamma}\int_{0}^{t} \int_{\mathbb{R}^{d}}\sum_{i=1}^{\gamma\langle 1, \mu_{s^-}^{\zeta, j}\rangle}\frac{\partial f}{\partial Q}(x) \cdot\biggl( \frac{1}{\gamma}\nabla  u_{j,j}(\|0\|)\biggr)\delta_{H^{i}(\gamma \mu_{s^-}^{\zeta, j})}(d x)d s\biggr]\\
    &\phantom{=}  -\sum_{\ell=1}^{L} \mathbb{E}\biggl[\int_{0}^{t} \frac{1}{\gamma} \sum_{i \in \Omega^{(\ell)}(\gamma \mu_{s^-}^{\zeta})}\biggl(\sum_{r=1}^{\alpha_{\ell j}} f\bigl(H^{i_{r}^{(j)}}(\gamma \mu_{s^-}^{\zeta, j})\bigr)\biggr) K_{\ell}^{\gamma}\bigl(\mathcal{P}^{(\ell)}(\gamma \mu_{s^-}^{\zeta}, \bm{i})\bigr)\\
    &\phantom{=} \qquad\qquad\qquad\qquad\qquad\qquad \times\biggl(\int_{\mathbb{Y}^{(\ell)}} m_{\ell}^{\eta}\bigl(\bm{y} | \mathcal{P}^{(\ell)}(\gamma \mu_{s^-}^{\zeta}, \bm{i}
    )\bigr)
     \pi_{\ell}^{\gamma} \bigl(\bm{y}|\bm{x}, \vmu_{s^-}^{\zeta}(dx')\bigr) d\bm{y}\biggr) d s\biggr] \\
    &\phantom{=}  +\sum_{\ell=1}^{L} \mathbb{E}\biggl[\int_{0}^{t} \frac{1}{\gamma} \sum_{i \in \Omega^{(\ell)}(\gamma\mu_{s^-}^{\zeta})}K_{\ell}^{\gamma}\bigl(\mathcal{P}^{(\ell)}(\gamma \mu_{s^-}^{\zeta}, \bm{i})\bigr)\\
    &\phantom{=}  \qquad\qquad\qquad \times\biggl(\int_{\mathbb{Y}^{(\ell)}}\bigl(\sum_{r=1}^{\beta_{\ell j}} f(y_{r}^{(j)})\bigr)m_{\ell}^{\eta}\bigl(\bm{y} | \mathcal{P}^{(\ell)}(\gamma \mu_{s^-}^{\zeta}, \bm{i}
    )\bigr) \pi_{\ell}^{\gamma} \bigl(\bm{y}|\bm{x}, \vmu_{s^-}^{\zeta}(dx')\bigr) d\bm{y}\biggr) d s\biggr] \\
    &= \mathbb{E}[\langle f, \mu_{0}^{\zeta, j}\rangle]+\mathbb{E}\biggl[\int_{0}^{t}\langle D_j\frac{\partial^2 f}{\partial Q^2}(x) - \frac{\partial f}{\partial Q} \cdot \nabla v_j (x), \mu_{s^-}^{\zeta, j}(d x)\rangle d s\biggr]\\
    &\phantom{=}  +\mathbb{E}\biggl[\int_{0}^{t}\bigl\langle \sum_{j'=1}^{J}\langle -\frac{\partial f}{\partial Q}(x) \cdot  \nabla u_{j,j'}(\norm{x-y}), \mu_{s^-}^{\zeta, j'}(d y) \rangle, \mu_{s^-}^{\zeta, j}(d x)\bigr\rangle d s\biggr] \\
    &\phantom{=}  +\frac{1}{\gamma}\mathbb{E}\biggl[\int_{0}^{t} \langle \frac{\partial f}{\partial Q}(x) \cdot \nabla u_{j,j}(\|0\|), \mu_{s^-}^{\zeta, j}(d x)\rangle d s\biggr] \\
    &\phantom{=}  -\sum_{\ell=1}^{L} \mathbb{E}\biggl[\int_{0}^{t} \frac{1}{\gamma}\int_{\mathbb{X}^{(\ell)}}\sum_{i \in \Omega^{(\ell)}(\gamma\mu_{s^-}^{\zeta})} \bigl(\sum_{r=1}^{\alpha_{\ell j}} f(x_{r}^{(j)}) \bigr)K^{\gamma}_{\ell}(\bm{x})\\
    &\phantom{=} \qquad\qquad\qquad\qquad\qquad \times\biggl(\int_{\mathbb{Y}^{(\ell)}}m_{\ell}^{\eta}(\bm{y} | \bm{x}
    )\pi_{\ell}^{\gamma} \bigl(\bm{y}|\bm{x}, \vmu_{s^-}^{\zeta}(dx')\bigr) d\bm{y}\biggr)\delta_{\mathcal{P}^{(\ell)}(\gamma\mu_{s^-}^{\zeta}, \bm{i})}(d \bm{x}) d s\biggr]\\
    &\phantom{=}  +\sum_{\ell=1}^{L} \mathbb{E}\biggl[\int_{0}^{t}\frac{1}{\gamma} \int_{\mathbb{X}^{(\ell)}}\sum_{i \in \Omega^{(\ell)}(\gamma\mu_{s^-}^{\zeta})} K^{\gamma}_{\ell}(\bm{x})\\
    &\phantom{=} \qquad\qquad \times\biggl(\int_{\mathbb{Y}^{(\ell)}}\bigl(\sum_{r=1}^{\beta_{\ell j}} f(y_{r}^{(j)})\bigr) m_{\ell}^{\eta}(\bm{y} | \bm{x}
    ) \pi_{\ell}^{\gamma} \bigl(\bm{y}|\bm{x}, \vmu_{s^-}^{\zeta}(dx')\bigr) d\bm{y}\biggr)\delta_{\mathcal{P}^{(\ell)}(\gamma\mu_{s^-}^{\zeta}, \bm{i})}(d \bm{x}) d s\biggr].
\end{align*}
Define the operator
\begin{equation} \label{eq:GammaFDef}
    \fsum{f} := \sum_{r=1}^{\beta_{\ell,j}} f(y_r^{(j)}) - \sum_{r=1}^{\alpha_{\ell,j}} f(x_r^{(j)}),
\end{equation}
and note in the remainder that for any test functions we consider
\begin{equation*}
    \abs{\fsum{f}} \leq (\beta_{\ell,j} + \alpha_{\ell,j}) \norm{f}_{\infty}.
\end{equation*}
Using the operators $(\mathcal{L}_{j}f)(x)$ and $(\tilde{\mathcal{L}}_{j,j'} f)(x,y)$ defined in (\ref{Eq:Operators}) we then obtain
\begin{align*}
    \mathbb{E}[\langle f, \mu_{t}^{\bm{\zeta}, j}\rangle]
    &= \mathbb{E}[\langle f, \mu_{0}^{\zeta, j}\rangle]+\mathbb{E}\biggl[\int_{0}^{t}\langle(\mathcal{L}_{j} f)(x), \mu_{s^-}^{\zeta, j}(d x)\rangle d s\biggr]\\
    &\phantom{=} +\mathbb{E}\biggl[\int_{0}^{t}\bigl\langle\sum_{j'=1}^{J}\langle(\tilde{\mathcal{L}}_{j,j'} f)(x,y), \mu_{s^-}^{\zeta, j'}(dy)\rangle,\mu_{s^-}^{\zeta, j}(d x)\bigr\rangle d s\biggr]\\
    &\phantom{=} +\frac{1}{\gamma}\mathbb{E}\biggl[\int_{0}^{t} \langle \frac{\partial f}{\partial Q}(x) \cdot \nabla u_{j,j}(\|0\|), \mu_{s^-}^{\zeta, j}(d x)\rangle d s\biggr] \\
    &\phantom{=}  +\sum_{\ell=1}^{L} \mathbb{E}\biggl[\int_{0}^{t} \int_{\tilde{\mathbb{X}}(\ell)} \frac{1}{\bm{\alpha}^{(\ell)}!} K_{\ell}(\bm{x})\biggl(\int_{\mathbb{Y}^{(\ell)}}\fsum{f}{m_{\ell}^{\eta}(\bm{y} | \bm{x})}\\
    &\phantom{=} \qquad\qquad\qquad\qquad\qquad\qquad\qquad \times \pi_{\ell}^{\gamma} \bigl(\bm{y}|\bm{x}, \vmu_{s^-}^{\zeta}(dx')\bigr) d\bm{y}\biggr) \lambda^{(\ell)}[\mu_{s^-}^{\zeta}](d \bm{x}) d s\biggr]. \numberthis \label{expectation}
\end{align*}
For this last equality, we switch integrals of the form $\int_{\mathbb{X}^{(\ell)}} \sum_{i \in \Omega^{(\ell)}(\gamma \mu_{s-}^{\zeta})} \cdots \delta_{\mathcal{P}^{(\ell)}(\gamma \mu_{s-}^{\zeta}, i)}(d\bm{x})$ to $\int_{\tilde{\mathbb{X}}^{(\ell)}} \frac{1}{\alpha^{(\ell)} !} \cdots \lambda^{(\ell)}[\mu_{s-}^{\zeta}](d\bm{x})$ using the definitions of $\mu_{s}^{\zeta, j}(d\bm{x})$ and $\lambda^{(\ell)}[\cdot]$ (see equation \eqref{mu} and Definition \ref{lambda}), and removing probability zero sets where two particles with the same type are simultaneously located at the same spatial location (see Definition \ref{xtilde}). Note that by definition indices for particles of the same species are ordered by the allowable substrate index sampling space $\Omega^{(\ell)}$ (see Definition \ref{omega}). In converting from integrals involving the positions of individual particles with $\delta_{\mathcal{P}^{(\ell)}(\gamma \mu_{s-}^{\zeta}, i)}(d\bm{x})$ to integrals involving product measures $\lambda^{(\ell)}[\mu_{s-}^{\zeta}](d\bm{x})$, we need to remove the ``diagonal" indices by means of integrating on $\tilde{\mathbb{X}}^{(\ell)}$ (see Definition \ref{xtilde}) and normalizing by the total number of index orderings $\bm{\alpha}^{(\ell)}!$.

\subsection{Tightness and Identification of the limit.}\label{SS:LimitIdentification}
We start by discussing relative compactness of the sequence $\{\mu_{t}^{\zeta, j}\}_{t \in[0, T]}, j=1,2, \cdots, J$. Recall that $M_{F}(\mathbb{R}^{d})$ denotes the space of finite measures endowed with the weak topology.
\begin{lemma}\label{L:Tightness}
The measure-valued processes $\{\mu_{t}^{\zeta, j}\}_{t \in[0, T]}, j=1,2, \cdots, J$ are relatively compact on $\mathbb{D}_{M_{F}(\mathbb{R}^{d})}[0, T]$, the space of c\`adl\`ag paths with values in $M_{F}(\mathbb{R}^{d})$ endowed with Skorokhod topology. Further, the corresponding weak limit of any convergent subsequence of $\{\vmu^{\zeta}_t\}_{t \in [0,T]}$ as $\zeta \to 0$ is in $C_{\otimes_{j=1}^J M_F(\mathbb{R}^d)}([0,T])$.
\end{lemma}
\begin{proof}
The proof is exactly as in \cite{IsaacsonMa2022} with minor adjustments to account for the presence of the bounded one-body and two-body potential interactions. We omit the details for the sake of brevity.
\end{proof}

Given a sequence $\zeta_k \to 0$ as $k \to \infty$, by Lemma \ref{L:Tightness} the sequence of marginal distribution vectors $\{\bm{\mu}_{t}^{\zeta_k}\}_k$ has a weakly convergent subsequence. Recall that we let $\bm{\xi}_{t}:=(\xi_{t}^{1}, \xi_{t}^{2}, \cdots, \xi_{t}^{J})$ be the corresponding limiting marginal distribution vector on $\mathbb{R}^{J \times d}$ and $\xi_{t}=\sum_{j=1}^{J} \xi_{t}^{j} \delta_{S_{j}}$ the corresponding limiting particle distribution on $\hat{P}$. We claim that for each $1 \leq j \leq J$, $\xi_t^j$ is then the unique solution to~\eqref{evl}, which in the case that $\tilde{L} = 0$ becomes
\begin{equation} \label{evl_noLtilde}
\begin{aligned}
\langle f, \xi_{t}^{j}\rangle &=\langle f, \xi_{0}^{j}\rangle+\int_{0}^{t}\langle(\mathcal{L}_{j} f)(x), \xi_{s}^{j}(d x)\rangle d s +\int_{0}^{t}\bigl\langle\sum_{j'=1}^J\langle(\tilde{\mathcal{L}}_{j,j'} f)(x,y), \xi_s^{j'}(dy)\rangle,\xi_{s}^{j}(d x)\bigr\rangle d s\\
&\phantom{=} +\sum_{\ell=1}^{L}\int_{0}^{t} \int_{\tilde{\mathbb{X}}(\ell)} \frac{1}{\bm{\alpha}^{(\ell)}!} K_{\ell}(\bm{x})\biggl(\int_{\mathbb{Y}^{(\ell)}}\fsum{f} m_{\ell}(\bm{y} | \bm{x})\pi_{\ell}\bigl(\bm{y}|\bm{x}, \vxi_{s}(dx')\bigr)d \bm{y}\biggr) \lambda^{(\ell)}[\xi_{s}](d \bm{x}) ds,
\end{aligned}
\end{equation}
where $\Gamma^{\ell,j}$ is defined in~\eqref{eq:GammaFDef}. Uniqueness of solutions to this equation is shown in Subsection \ref{SS:Uniqueness}. We now identify that the limit $\xi^j_t$ satisfies this equation.


Let $\mathcal{S}$ be the collection of elements $\Phi$ in the space of bounded functionals, $B(\otimes_{j=1}^{J} M_{F}(\mathbb{R}^{d})\bigr)$, of the form
\begin{equation}\label{mar}
\Phi(\bm{\mu})=\varphi(\langle f_{1}, \bm{\mu}\rangle,\langle f_{2}, \bm{\mu}\rangle \ldots\langle f_{M}, \bm{\mu}\rangle)
\end{equation}
for some $M \in \mathbb{N}$, some $\varphi \in C^{\infty}(\mathbb{R}^{J \times M})$, and $\langle f_{m}, \bm{\mu}\rangle=(\langle f_{1, m}, \mu^{1}\rangle, \cdots,\langle f_{J, m}, \mu^{J}\rangle)$ where each $\{f_{j, m}\} \in C_{b}^{2}(\mathbb{R}^{d})$ for $j=1, \cdots, J$ and $m=1, \cdots, M$. Then $\mathcal{S}$ separates points in $\otimes_{j=1}^{J} M_{F}(\mathbb{R}^{d})$ (see Chapter $3.4$ of \cite{EthierKurtz1986} and Proposition $3.3$ of \cite{CapponiSun2020}). To identify the limit, it suffices to show convergence of the martingale problem for functions of the form \eqref{mar}, given the existence and uniqueness of the limiting process.

For $\Phi \in \mathcal{S}$ of the form $\eqref{mar}, \bm{\mu}:=(\mu^{1}, \mu^{2} \cdots, \mu^{J}) \in \otimes_{j=1}^{J} M_{F}(\mathbb{R}^{d})$ and $\mu=\sum_{j=1}^{J} \mu^{j} \delta_{S_{j}} \in M_{F}(\hat{P})$ with each $\mu^{j} \in M_{F}(\mathbb{R}^{d})$, the generator $\mathcal{A}$ of \eqref{evl_noLtilde} and of the limiting martingale problem for $1 \leq j \leq J$, is defined as
\begin{equation*}
\begin{aligned}
(\mathcal{A} \Phi)(\bm{\mu}) &\coloneqq
\sum_{m=1}^{M} \sum_{j=1}^{J} \frac{\partial \varphi}{\partial x_{(m-1) * J+j}}(\langle f_{1}, \bm{\mu}\rangle,\langle f_{2}, \bm{\mu}\rangle \ldots\langle f_{M}, \bm{\mu}\rangle)\\
&\times\biggl[\langle (\mathcal{L}_{j} f_{j, m})(x), \mu^{j}(dx)\rangle+\bigl\langle\sum_{j'=1}^{J}\langle\tilde{\mathcal{L}}_{j,j'} f_{j, m}(x,y), \mu^{j'}(d y)\rangle,\mu^{j}(d x)\bigr\rangle\\
&+\sum_{\ell=1}^{L} \int_{\tilde{\mathbb{X}}(\ell)} \frac{1}{\bm{\alpha}^{(\ell)}!} K_{\ell}(\bm{x})\biggl(\int_{\mathbb{Y}^{(\ell)}}\Gamma^{\ell,j}[f_{j,m}](\bm{x},\bm{y})m_{\ell}(\bm{y} | \bm{x})\pi_{\ell}\bigl(\bm{y}|\bm{x}, \vmu(dx')\bigr) d \bm{y}\biggr) \lambda^{(\ell)}[\mu](d \bm{x})\biggr].
\end{aligned}
\end{equation*}

\begin{lemma}(Weak Convergence). For any $\Phi \in \mathcal{S}$ and $0 \leq r_{1} \leq r_{2} \cdots \leq r_{W}=s<t<T$ and $\{\psi_{w}\}_{w=1}^{W} \subset$ $B(\otimes_{j=1}^{J} M_{F}(\mathbb{R}^{d})\bigr)$, we have that
\begin{equation}
\lim_{\zeta \rightarrow 0} \mathbb{E}\biggl[\{\Phi(\bm{\mu}_{t}^{\zeta})-\Phi(\mu_{s}^{\zeta})-\int_{s}^{t}(\mathcal{A} \Phi)(\bm{\mu}_{r}^{\zeta}) d r\} \prod_{w=1}^{W} \psi_{w}(\bm{\mu}_{r_{w}}^{\zeta})\biggr]=0.\label{il}
\end{equation}
\end{lemma}
\begin{proof}
For each $j=1, \cdots J$, we can rewrite \eqref{system} as
$$
\langle f, \mu_{t}^{\zeta, j}\rangle=\langle f, \mu_{0}^{\zeta, j}\rangle+M_{t}^{f, j}+A_{t}^{f, j}
$$
where
\begin{equation*}
\begin{aligned}
A_{t}^{f, j}&= \int_{0}^{t}\langle(\mathcal{L}_{j} f)(x), \mu_{s-}^{\zeta, j}(d x)\rangle d s + \int_{0}^{t}\bigl\langle\sum_{j'=1}^{J}\langle(\tilde{\mathcal{L}}_{j,j'} f)(x,y), \mu_{s-}^{\zeta, j'}(d y)\rangle,\mu_{s-}^{\zeta, j}(d x)\bigr\rangle d s\\
&\phantom{=} +\frac{1}{\gamma}\int_{0}^{t} \langle \frac{\partial f}{\partial Q}(x) \cdot \nabla u_{j,j}(\|0\|), \mu_{s-}^{\zeta, j}(d x)\rangle d s\\
&\phantom{=}  +\sum_{\ell=1}^{L}\int_{0}^{t} \int_{\tilde{\mathbb{X}}(\ell)} \frac{1}{\bm{\alpha}^{(\ell)}!} K_{\ell}(\bm{x})\biggl(\int_{\mathbb{Y}^{(\ell)}}\fsum{f}{m_{\ell}^{\eta}(\bm{y} | \bm{x})}\pi_{\ell}^{\gamma} \bigl(\bm{y}|\bm{x}, \vmu_{s-}^{\zeta}(dx')\bigr) d \bm{y}\biggr) \lambda^{(\ell)}[\mu_{s-}^{\zeta}](d \bm{x}) ds
\end{aligned}
\end{equation*}
and
\begin{equation*}
\begin{aligned}
M_{t}^{f, j} &= \frac{1}{\gamma} \sum_{i \geq 1} \int_{0}^{t} 1_{\{i \leq \gamma\langle 1, \mu_{s-}^{\zeta, j}\rangle\}} \sqrt{2 D_{j}} \frac{\partial f}{\partial Q}\bigl(H^{i}(\gamma \mu_{s-}^{\zeta, j})\bigr) d W_{s}^{i} \\
&\phantom{=} +\frac{1}{\gamma }\sum_{\ell=1}^{L} \int_{0}^{t} \int_{\mathbb{I}^{(\ell)}} \int_{\mathbb{Y}^{(\ell)}} \int_{\mathbb{R}_{+}^{3}}\biggl(\langle f, -\sum_{r=1}^{\alpha_{\ell j}} \delta_{H^{(j)}}(\gamma \mu_{s-}^{\zeta, j})+ \sum_{r=1}^{\beta_{\ell, j}} \delta_{y_{r}^{(j)}}\rangle\biggr) 1_{\{\bm{i} \in \Omega^{(\ell)}(\gamma\mu_{s-}^{\zeta})\}} \\
&\phantom{=} \times 1_{\{\theta_1 \leq K_{\ell}^{\gamma}(\mathcal{P}^{(\ell)}(\gamma\mu_{s-}^{\zeta}, \bm{i})\bigr)\}} 1_{\{\theta_2 \leq m_{\ell}^{\eta}(\bm{y}|\mathcal{P}^{(\ell)}(\gamma\mu_{s-}^{\zeta},\bm{i})\bigr)\}} 1_{\{\theta_3 \leq \pi_{\ell}^{\gamma} \bigl(\bm{y}|\bm{x}, \vmu_{s^-}^{\zeta}(dx')\bigr)\}} d\tilde{N}_{\ell}(s, \bm{i, y}, \theta_1, \theta_2, \theta_3).
\end{aligned}
\end{equation*}
Note that $M_{t}^{f, j}$ is a square integrable martingale (see Proposition $2.4$ in \cite{IkedaWatanabe2014}) with quadratic variation
\begin{equation*}
\begin{aligned}
\langle M^{f, j}\rangle_{t}&=\frac{1}{\gamma^{2}} \int_{0}^{t} \sum_{i=1}^{\gamma\langle 1, \mu_{s-}^{\zeta, j}\rangle}\biggl(\sqrt{2 D_{j}} \frac{\partial f}{\partial Q}\bigl(H^{i}(\gamma \mu_{s-}^{\zeta, j})\bigr)\biggr)^{2} d s\\
&\phantom{=} +\frac{1}{\gamma^{2}}\sum_{\ell=1}^{L} \int_{0}^{t} \int_{\mathbb{Y}^{(\ell)}} \sum_{\{i \in \Omega^{(\ell)}(\gamma \mu_{s-)}^{\zeta})\}}\biggl(-\sum_{r=1}^{\alpha_{\ell j}} f\bigl(H^{i_{r}^{j}}(\gamma \mu_{s-}^{\zeta, j})\bigr)+\sum_{r=1}^{\beta_{\ell j}} f(y_{r}^{j})\biggr)^{2} \\
&\phantom{=} \qquad\qquad\qquad\qquad \times K_{\ell}^{\gamma}\bigl(\mathcal{P}^{(\ell)}(\gamma \mu_{s-}^{\zeta}, \bm{i})\bigr)  m_{\ell}^{\eta}\bigl(\bm{y} | \mathcal{P}^{(\ell)}(\gamma \mu_{s-}^{\zeta}, \bm{i})\bigr)\pi_{\ell}^{\gamma} \bigl(\bm{y}|\bm{x}, \vmu_{s^-}^{\zeta}(dx')\bigr) d\bm{y} \, ds
\end{aligned}
\end{equation*}
The quadratic variation of $M_{t}^{f, j}$ is thus uniformly bounded and goes to 0 as $\zeta \rightarrow 0$, due to the uniform boundedness of $f$ and its partial derivatives, and by Assumptions \ref{A:K_bound} and \ref{A:PlacementMeasure}.

Now define $M_{t}^{f, j}=\mathcal{C}_{t}^{f, j}+\mathcal{D}_{t}^{f, j}$, where
\begin{equation*}
\mathcal{C}_{t}^{f, j}=\frac{1}{\gamma} \sum_{i \geq 1} \int_{0}^{t} 1_{\{i \leq \gamma\langle 1, \mu_{s-}^{\zeta, j}\rangle\}} \sqrt{2 D_{j}} \frac{\partial f}{\partial Q}\bigl(H^{i}(\gamma \mu_{s-}^{\zeta, j})\bigr) d W_{s}^{i}
\end{equation*}
is the continuous martingale part, and
\begin{equation}\label{dfj}
\begin{aligned}
\mathcal{D}_{t}^{f, j}&= \frac{1}{\gamma} \sum_{\ell=1}^{L} \int_{0}^{t} \int_{\mathbb{I}^{(\ell)}} \int_{\mathbb{Y}^{(\ell)}} \int_{\mathbb{R}_{+}^{3}}\biggl(\langle f, - \sum_{r=1}^{\alpha_{\ell j}} \delta_{H^{i_{r}^{(j)}}(\gamma \mu_{s-}^{\zeta, j})}+ \sum_{r=1}^{\beta_{\ell j}} \delta_{y_{r}^{(j)}}\rangle\biggr) 1_{\{\bm{i} \in \Omega^{(\ell)}(\gamma \mu_{s-}^{\zeta})\}} \\
&\phantom{=} \times 1_{\{\theta_{1} \leq K_{\ell}^{\gamma}\bigl(\mathcal{P}^{(\ell)}(\gamma \mu_{s-}^{\zeta}, i)\bigr)\}} 1_{\{\theta_{2} \leq m_{\ell}^{\eta}\bigl(\bm{y} | \mathcal{P}^{(\ell)}(\gamma \mu_{s-}^{\zeta}, i)\bigr)\}} 1_{\{\theta_3 \leq \pi_{\ell}^{\gamma} \bigl(\bm{y}|\bm{x}, \vmu_{s^-}^{\zeta}(dx')\bigr)\}} d \tilde{N}_{\ell}(s, \bm{i}, \bm{y}, \theta_{1}, \theta_{2},\theta_{3}),
\end{aligned}
\end{equation}
is the martingale part coming from the stochastic integral with respect to the Poisson point processes.

Let $\theta = (\theta_1,\theta_2,\theta_3)$. For simplicity of notation, we define
\begin{align*}
g^{\ell, f, \mu^{\zeta, j}}(s, \bm{i}, \bm{y}, \theta) &= \biggl(\langle f, \mu_{s-}^{\zeta, j}-\frac{1}{\gamma} \sum_{r=1}^{\alpha_{\ell j}} \delta_{H^{i_{r}^{(j)}}(\gamma \mu_{s-}^{\zeta, j})}+\frac{1}{\gamma} \sum_{r=1}^{\beta_{\ell j}} \delta_{y_{r}^{(j)}}\rangle-\langle f, \mu_{s-}^{\bm{\zeta}, j}\rangle\biggr) 1_{\{\bm{i} \in \Omega^{(\ell)}(\gamma \mu_{s-}^{\zeta})\}} \\
&\phantom{=} \qquad\qquad\times 1_{\{\theta_{1} \leq K_{\ell}^{\gamma}\bigl(\mathcal{P}^{(\ell)}(\gamma \mu_{s-}^{\zeta}, i)\bigr)\}}  1_{\{\theta_{2} \leq m_{\ell}^{\eta}\bigl(\bm{y} | \mathcal{P}^{(\ell)}(\gamma \mu_{s-}^{\zeta}, i)\bigr)\}} 1_{\{\theta_3 \leq \pi_{\ell}^{\gamma} \bigl(\bm{y}|\bm{x}, \vmu_{s^-}^{\zeta}(dx')\bigr)\}}\\
&=\frac{1}{\gamma}\biggl(-\sum_{r=1}^{\alpha_{\ell j}} f\left(H^{i_{r}^{(j)}}(\gamma \mu_{s-}^{\zeta, j})\right)+\sum_{r=1}^{\beta_{\ell j}} f(y_{r}^{(j)})\biggr) 1_{\{\bm{i} \in \Omega^{(\ell)}(\gamma \mu_{s-}^{\zeta})\}} 1_{\{\theta_{1} \leq K_{\ell}^{\gamma}\bigl(\mathcal{P}^{(\ell)}(\gamma \mu_{s-}^{\zeta}, i)\bigr)\}} \\
&\phantom{=} \qquad\qquad\qquad\qquad\qquad\qquad\times 1_{\{\theta_{2} \leq m_{\ell}^{\eta}\bigl(\bm{y} | \mathcal{P}^{(\ell)}(\gamma \mu_{s-}^{\zeta}, i)\bigr)\}} 1_{\{\theta_3 \leq \pi_{\ell}^{\gamma} \bigl(\bm{y}|\bm{x}, \vmu_{s^-}^{\zeta}(dx')\bigr)\}},
\end{align*}
which represents the jumps and is uniformly bounded by $\mathcal{O}(\frac{1}{\gamma})$. With some abuse of notation we shall write $\bm{g}^{\ell, f, \bm{\mu}^{\zeta}}$ for the vector $(g^{\ell, f, \mu^{\zeta, 1}}, \cdots, g^{\ell, f, \mu^{\zeta, J}})$. Then \eqref{dfj} becomes
$$
\mathcal{D}_{t}^{f, j}=\sum_{\ell=1}^{L} \int_{0}^{t} \int_{\mathbb{I}^{(\ell)}} \int_{\mathbb{Y}^{(\ell)}} \int_{\mathbb{R}_{+}^{3}} g^{\ell, f, \mu^{\zeta, j}}(s, \bm{i}, \bm{y}, \theta) d \tilde{N}_{\ell}(s, \bm{i}, \bm{y}, \theta_{1}, \theta_{2},\theta_{3}).
$$
Applying Itô's formula (see Theorem $5.1$ in \cite{IkedaWatanabe2014}) to $\Phi(\bm{\mu}_{t}^{\zeta})$ we obtain
\begin{align*}
& \Phi(\bm{\mu}_{t}^{\zeta})-\Phi(\bm{\mu}_{s}^{\zeta})-\int_{s}^{t}(\mathcal{A} \Phi)(\bm{\mu}_{r}^{\zeta}) d r=\int_{s}^{t} \sum_{m=1}^{M} \sum_{j=1}^{J} \frac{\partial \varphi}{\partial x_{(m-1) * J+j}}\bigl(\langle f_{1}, \bm{\mu}_{r}^{\zeta}\rangle,\langle f_{2}, \bm{\mu}_{r}^{\zeta}\rangle \ldots\langle f_{M}, \bm{\mu}_{r}^{\zeta}\rangle\bigr)d \mathcal{C}_{r}^{f_{j, m}, j}
\\
&+\frac{1}{2} \int_{s}^{t} \sum_{m=1}^{M} \sum_{j=1}^{J} \frac{\partial^{2} \varphi}{\partial x_{(m-1) * J+j}^{2}}\bigl(\langle f_{1}, \bm{\mu}_{r}^{\zeta}\rangle,\langle f_{2}, \bm{\mu}_{r}^{\zeta}\rangle \cdots\langle f_{M}, \bm{\mu}_{r}^{\zeta}\rangle\bigr) d\langle\mathcal{C}^{f_{j, m}, j}\rangle_{r} 
\\
& +\sum_{\ell=1}^{L} \int_{s}^{t} \int_{\mathbb{I}^{(\ell)}} \int_{\mathbb{I}^{(\ell)}} \int_{\mathbb{R}_{+}^{2}}\biggl(\varphi\bigl(\langle f_{1}, \bm{\mu}_{r}^{\bm{\zeta}}\rangle+\bm{g}^{\ell, f_{1}, \bm{\mu}^{\zeta}}(r, \bm{i}, \bm{y}, \theta), \ldots,\langle f_{M}, \bm{\mu}_{r}^{\zeta}\rangle +\bm{g}^{\ell, f_{M}, \bm{\mu}^{\zeta}}(r, \bm{i}, \bm{y}, \theta)\bigr)\biggr.\\
& \qquad\qquad\qquad\qquad\qquad\qquad   -\varphi\bigl(\langle f_{1}, \bm{\mu}_{r}^{\zeta}\rangle,\langle f_{2}, \bm{\mu}_{r}^{\zeta}\rangle \ldots\langle f_{M}, \bm{\mu}_{r}^{\zeta}\rangle\bigr)\biggr) d \tilde{N}_{\ell}(r, \bm{i}, \bm{y}, \theta) 
\\
& +\sum_{\ell=1}^{L} \int_{s}^{t} \int_{\mathbb{I}^{(\ell)}} \int_{\mathbb{Y}^{(\ell)}} \int_{\mathbb{R}_{+}^{2}}\biggl(\varphi\bigl(\langle f_{1}, \bm{\mu}_{r}^{\bm{\zeta}}\rangle+\bm{g}^{\ell, f_{1}, \bm{\mu}^{\zeta}}(s, \bm{i}, \bm{y}, \theta), \ldots,\langle f_{M}, \bm{\mu}_{r}^{\zeta}\rangle+\bm{g}^{\ell, f_{M}, \bm{\mu}^{\zeta}}(s, \bm{i}, \bm{y}, \theta)\bigr)\biggr. \\
& \qquad\qquad -\varphi\bigl(\langle f_{1}, \bm{\mu}_{r}^{\zeta}\rangle,\langle f_{2}, \bm{\mu}_{r}^{\zeta}\rangle \ldots\langle f_{M}, \bm{\mu}_{r}^{\zeta}\rangle\bigr) \\
& \qquad\qquad -\sum_{m=1}^{M} \sum_{j=1}^{J} g^{\ell, f_{j, m}, \mu^{\zeta,j}}(s, \bm{i}, \bm{y}, \theta) \frac{\partial \varphi}{\partial x_{(m-1) J+j}}\bigl(\langle f_{1}, \bm{\mu}_{r}^{\zeta}\rangle,\ldots,\langle f_{M}, \bm{\mu}_{r}^{\zeta}\rangle\bigr)\biggr) d \bar{N}_{\ell}(s, \bm{i}, \bm{y}, \theta) 
\\
& +\sum_{m=1}^{M} \sum_{j=1}^{J} \sum_{\ell=1}^{L} \int_{s}^{t} \frac{\partial \varphi}{\partial x_{(m-1) J+j}}\bigl(\langle f_{1}, \bm{\mu}_{r}^{\zeta}\rangle,\langle f_{2}, \bm{\mu}_{r}^{\zeta}\rangle \ldots\langle f_{M}, \bm{\mu}_{r}^{\zeta}\rangle\bigr) \\
& \qquad \times \int_{\tilde{\mathbb{X}}(\ell)} \frac{1}{\bm{\alpha}^{(\ell)} !} K_{\ell}(\bm{x})\biggl[\int_{\mathbb{Y}^{(\ell)}}\fsum{f_{j,m}} \biggl(m_{\ell}^{\eta}(\bm{y} | \bm{x})\pi_{\ell}^{\gamma}\bigl(\bm{y}|\bm{x}, \vmu_{r-}^{\zeta}(dx)\bigr)\\
& \qquad\qquad\qquad\qquad\qquad\qquad\qquad\qquad\qquad -m_{\ell}(\bm{y} | \bm{x})\pi_{\ell}\bigl(\bm{y}|\bm{x}, \vmu_{r-}^{\zeta}(d x)\bigr)\biggr) d \bm{y}\biggr] \lambda^{(\ell)}[\mu_{r-}^{\zeta}](d \bm{x}) d r 
\\
& +\int_s^t \sum_{m=1}^{M} \sum_{j=1}^{J} \frac{\partial \varphi}{\partial x_{(m-1) * J+j}}\bigl(\langle f_{1}, \bm{\mu}_{r}^{\zeta}\rangle,\ldots, \langle f_{M}, \bm{\mu}_{r}^{\zeta}\rangle\bigr)\biggl(\frac{1}{\gamma}\int_{\mathbb{R}^{d}} \frac{\partial f_{j,m}}{\partial Q}(x) \cdot \nabla u_{j,j}(\|0\|) \mu_{r-}^{\zeta, j}(d x) \biggr)dr 
\\
& =\sum_{\kappa=1}^{6} \Lambda_{\kappa}^{\zeta}(t), \numberthis\label{ito}
\end{align*}
where $\Lambda_{\kappa}^{\zeta}(t)$ represents the $\kappa$th additive term on the right hand side. We now use the Skorokhod representation theorem (Theorem $1.8$ in \cite{EthierKurtz1986}) which, for the purposes of identifying the limit and proving \eqref{il}, allows us to assume that the aforementioned claimed convergence of $\bm{\mu}_{t}^{\bm{\zeta}}$ holds with probability one in the topology of weak convergence of measures. The Skorokhod representation theorem involves the introduction of another probability space, but we ignore this distinction in the notation. To show \eqref{il}, it is then sufficient to prove that the left-hand side of \eqref{ito} goes to zero in probability. With this goal in mind, we proceed to prove convergence in probability to zero for $\Lambda_{\kappa}^{\zeta}(t)$ for $\kappa=1, \cdots, 6$.

The fact that $\Lambda_{\kappa}^{\zeta}(t)$ for $\kappa=1, \cdots, 4$ converge to zero in probability follows from Section 7.3 in \cite{IsaacsonMa2022}. It remains to show that $\Lambda_5^\zeta$ and $\Lambda_6^\zeta$ converge to zero in probability. Note that
\begin{align*}
\Lambda_5^\zeta &= \sum_{m=1}^{M} \sum_{j=1}^{J} \sum_{\ell=1}^{L} \int_{s}^{t} \frac{\partial \varphi}{\partial x_{(m-1) J+j}}\bigl(\langle f_{1}, \bm{\mu}_{r}^{\zeta}\rangle,\langle f_{2}, \bm{\mu}_{r}^{\zeta}\rangle \ldots\langle f_{M}, \bm{\mu}_{r}^{\zeta}\rangle\bigr) \int_{\tilde{\mathbb{X}}(\ell)} \frac{1}{\bm{\alpha}^{(\ell)} !} K_{\ell}(\bm{x})\\
&\phantom{=} \times \biggl[\int_{\mathbb{Y}^{(\ell)}}\fsum{f_{j,m}} \biggl( \bigl(\pi_{\ell}^{\gamma}\bigl(\bm{y}|\bm{x}, \vmu_{r-}^{\zeta}(dx)\bigr)-\pi_{\ell}\bigl(\bm{y}|\bm{x}, \vmu_{r-}^{\zeta}(dx)\bigr) \bigr) m_{\ell}^{\eta}(\bm{y} | \bm{x}) \\
&\phantom{=} \qquad\qquad\qquad\qquad\qquad\qquad+\pi_{\ell}\bigl(\bm{y}|\bm{x}, \vmu_{r-}^{\zeta}(dx)\bigr)\bigl(m_{\ell}^{\eta}(\bm{y} | \bm{x})-m_{\ell}(\bm{y}|\bm{x})\bigr)\biggr)d \bm{y}\biggr]\lambda^{(\ell)}[\mu_{r-}^{\zeta}](d \bm{x}) d r\\
&\leq \sum_{m=1}^{M} \sum_{j=1}^{J} \sum_{\ell=1}^{L} \int_{s}^{t} \biggl|\frac{\partial \varphi}{\partial x_{(m-1) J+j}}\bigl(\langle f_{1}, \bm{\mu}_{r}^{\zeta}\rangle,\langle f_{2}, \bm{\mu}_{r}^{\zeta}\rangle \ldots\langle f_{M}, \bm{\mu}_{r}^{\zeta}\rangle\bigr)\biggr| \int_{\tilde{\mathbb{X}}(\ell)} \frac{1}{\bm{\alpha}^{(\ell)} !} K_{\ell}(\bm{x})\\
&\phantom{=}  \times \biggl[\biggl|\int_{\mathbb{Y}^{(\ell)}}\fsum{f_{j,m}}m_{\ell}^{\eta}(\bm{y} | \bm{x}) \bigl(\pi_{\ell}^{\gamma}\bigl(\bm{y}|\bm{x}, \vmu_{r-}^{\zeta}(dx)\bigr)-\pi_{\ell}\bigl(\bm{y}|\bm{x}, \vmu_{r-}^{\zeta}(dx)\bigr)\bigr)d \bm{y}\biggr|\\
&\phantom{=} +\biggl|\int_{\mathbb{Y}^{(\ell)}}\fsum{f_{j,m}}\pi_{\ell}\bigl(\bm{y}|\bm{x}, \vmu_{r-}^{\zeta}(dx)\bigr)\bigl(m_{\ell}^{\eta}(\bm{y} | \bm{x})-m_{\ell}(\bm{y} | \bm{x})\bigr)d \bm{y}\biggr|\biggr]\lambda^{(\ell)}[\mu_{r-}^{\zeta}](d \bm{x}) d r\\
&\leq \sum_{m=1}^{M} \sum_{j=1}^{J} \sum_{\ell=1}^{L} \int_{s}^{t} \biggl|\frac{\partial \varphi}{\partial x_{(m-1) J+j}}\bigl(\langle f_{1}, \bm{\mu}_{r}^{\zeta}\rangle,\langle f_{2}, \bm{\mu}_{r}^{\zeta}\rangle \ldots\langle f_{M}, \bm{\mu}_{r}^{\zeta}\rangle\bigr)\biggr| \int_{\tilde{\mathbb{X}}(\ell)} \frac{1}{\bm{\alpha}^{(\ell)} !} K_{\ell}(\bm{x})\\
&\phantom{=}  \times \biggl[\displaystyle\sup_{\bm{y} \in \mathbb{Y}^{(\ell)}, \bm{x} \in \mathbb{X}^{(\ell)}}\bigl|\pi_{\ell}\bigl(\bm{y}|\bm{x}, \vmu_{r-}^{\zeta}(dx)\bigr)-\pi_{\ell}^{\gamma}\bigl(\bm{y}|\bm{x}, \vmu_{r-}^{\zeta}(dx)\bigr) \bigr| \int_{\mathbb{Y}^{(\ell)}}\bigl|\fsum{f_{j,m}}\bigr|m_{\ell}^{\eta}(\bm{y}|\bm{x})d \bm{y}\\
&\phantom{=} +\biggl|\int_{\mathbb{Y}^{(\ell)}}\fsum{f_{j,m}} \pi_{\ell}\bigl(\bm{y}|\bm{x}, \vmu_{r-}^{\zeta}(dx)\bigr)\bigl(m_{\ell}^{\eta}(\bm{y} | \bm{x})-m_{\ell}(\bm{y} | \bm{x})\bigr)d \bm{y}\biggr|\biggr]\lambda^{(\ell)}[\mu_{r-}^{\zeta}](d \bm{x}) d r.\numberthis
\end{align*}
We note the following estimate proven in Appendix~\ref{S:mollifierconvlemmaproof}.
\begin{lemma} 
For any $\eta \geq 0$ small enough, $\tilde{L}+1 \leq \ell \leq L, \bm{y} \in \mathbb{Y}^{(\ell)}, \bm{x} \in \mathbb{X}^{(\ell)}$, and $f \in C_{b}^{2}(\mathbb{Y}^{(\ell)})$, there exists a constant $C$ such that
\begin{equation*}
\biggl|\int_{\mathbb{Y}^{(\ell)}} f(\bm{y})\pi_{\ell}\bigl(\bm{y}|\bm{x}, \vmu_{r-}^{\zeta}(dx)\bigr) \bigl(m_{\ell}^{\eta}(\bm{y} | \bm{x})-m_{\ell}(\bm{y} | \bm{x})\bigr) d \bm{y}\biggr| \leq C\|f\|_{C_{b}^{1}(\mathbb{Y}^{(\ell)})}{\eta}. \label{m-m}
\end{equation*}
\end{lemma}
With the aid of Assumption \ref{convpi}, Lemma \ref{m-m}, and the assumptions that $f$ and its partial derivatives are uniformly bounded, we then have $\displaystyle\lim _{\zeta\rightarrow 0} \displaystyle\sup _{t \in[0, T]} \mathbb{E}|\Lambda_{5}^{\zeta}(t)|=0$. Incorporating Assumption \eqref{anorm} we obtain
\begin{equation*}
      \abs{\frac{1}{\gamma}\int_{\mathbb{R}^{d}} \frac{\partial f_{j,m}}{\partial Q}(x) \cdot \nabla u_{j,j}(\|0\|) \mu_{r-}^{\zeta, j}(d x)}
  \leq \frac{C_{\circ}}{\gamma}\|f_{j,m}\|_{C^1_b(\mathbb{R^d})}\|u\|_{C^1(\mathbb{R}^{2d})}.
\end{equation*}
We then have that $\displaystyle\lim _{\zeta\rightarrow 0} \displaystyle\sup _{t \in[0, T]} \mathbb{E}|\Lambda_{6}^{\zeta}(t)|=0.$ Thus, the left-hand side of \eqref{ito} goes to zero in probability, concluding the proof of the lemma.

\end{proof}
We have shown that if the weak limit of the marginal distribution $(\mu_{t}^{\zeta, 1}, \mu_{t}^{\zeta, 2} \ldots, \mu_{t}^{\zeta, J})$ exists and is unique (uniqueness is shown in Subsection \ref{SS:Uniqueness}), then as $\zeta$ goes to zero, it will converge to the limiting particle distribution $(\xi_{t}^{1}, \xi_{t}^{2}, \cdots, \xi_{t}^{J})$ in distribution with respect to the topology of weak convergence of measures.

\subsection{Uniqueness of the limiting solution. }\label{SS:Uniqueness}
We now show that the solution to~\eqref{evl_noLtilde} is unique in $C_{M_{F}(\mathbb{R}^d)}([0,T])$. $C$ will subsequently denote a generic constant. Suppose, by contradiction, that we have two different solutions to~\eqref{evl_noLtilde}, $\{\vxi_t \coloneqq (\xi_{t}^{1}, \xi_{t}^{2}, \cdots, \xi_{t}^{J})\}_{t \in[0, T]}$ and $\{\vxibar_t \coloneqq (\bar{\xi}_{t}^{1}, \bar{\xi}_{t}^{2}, \cdots, \bar{\xi}_{t}^{J})\}_{t \in[0, T]}$, with the same initial condition $\vxi_{0}=\vxibar_{0}$. Parallel to Eq (\ref{evl_noLtilde}), for a test function of the form of $\psi_{t}(x) \in C_{b}^{1,2}(\mathbb{R}_{+} \times \mathbb{R}^{d})$, we get
\begin{equation}\label{sol}
\begin{aligned}
\langle\psi_{t}, \xi_{t}^{j}\rangle &= \langle\psi_{0}, \xi_{0}^{j}\rangle+\int_{0}^{t}\langle\partial_{s} \psi_{s}+(\mathcal{L}_{j} \psi_{s})(x), \xi_{s}^{j}(d x)\rangle d s +\int_{0}^{t}\bigl\langle\sum_{j'=1}^{J}\langle(\tilde{\mathcal{L}}_{j,j'} \psi_s)(x,y), \xi_{s}^{j'}(dy)\rangle,\xi_{s}^{j}(d x)\bigr\rangle d s\\
&+\sum_{\ell=1}^{L}\int_{0}^{t} \int_{\tilde{\mathbb{X}}(\ell)} \frac{1}{\bm{\alpha}^{(\ell)}!} K_{\ell}(\bm{x})\biggl(\int_{\mathbb{Y}^{(\ell)}}\fsum{\psi_s} m_{\ell}(\bm{y} | \bm{x})
\pi_{\ell} \bigl(\bm{y}|\bm{x}, \vxi_{s}(dx)\bigr) d \bm{y}\biggr)\lambda^{(\ell)}[\xi_{s}](d \bm{x}) ds,
\end{aligned}
\end{equation}
recalling~\eqref{eq:GammaFDef}.

Let $\mathcal{P}_{j, t}, t \geq 0$, be the semigroup generated by $\mathcal{L}_{j}, \text{ with } (\mathcal{L}_{j}f)(x) = D_j \Delta_x f(x) - \nabla_x f(x) \cdot \nabla_x v_{j}(x) \text{ for } j=1,2, \cdots, J$. Choose $\psi_{s}(x;t)=\mathcal{P}_{j, t-s} f(x)$, respectively for each $1 \leq j \leq J$, where $f \in C_{b}^{2}(\mathbb{R}^{d}) \text{ and }\|f\|_{L^{\infty}} \leq 1$, with $s\leq t$. Using the semigroup property, we then have
\begin{align*}
    \partial_{s} \psi_{s}(x;t) &= \partial_{s} (\mathcal{P}_{j,t-s}f)(x)\\
    &= - \displaystyle \lim_{h\rightarrow 0} \frac{(\mathcal{P}_{j,t-s+h}f)(x)-(\mathcal{P}_{j,t-s}f)(x)}{h} \\
    &= - \displaystyle \lim_{h\rightarrow 0}\frac{\mathcal{P}_{j,h}(\mathcal{P}_{j,t-s}f)(x)-(\mathcal{P}_{j,t-s}f)(x)}{h}\\
    &= -(\mathcal{L}_{j}\mathcal{P}_{j,t-s}f)(x)\\
    &= -(\mathcal{L}_{j}\psi_s)(x;t),
\end{align*}
and Eq \eqref{sol} becomes
\begin{equation}\label{group}
\begin{aligned}
\langle &f, \xi_{t}^{j}\rangle =\langle \mathcal{P}_{j,t}f, \xi_{0}^{j}\rangle+\int_{0}^{t}\bigl\langle\sum_{j'=1}^J\langle(\tilde{\mathcal{L}}_{j,j'} \mathcal{P}_{j,t-s}f)(x,y), \xi_{s}^{j'}(dy)\rangle,\xi_{s}^{j}(d x)\bigr\rangle d s\\
&\phantom{=} +\sum_{\ell=1}^{L}\int_{0}^{t} \int_{\tilde{\mathbb{X}}(\ell)} \frac{1}{\bm{\alpha}^{(\ell)}!} K_{\ell}(\bm{x})\biggl(\int_{\mathbb{Y}^{(\ell)}}\fsum{\mathcal{P}_{j,t-s}f} m_{\ell}(\bm{y} | \bm{x}) \pi_{\ell} \bigl(\bm{y}|\bm{x}, \vxi_{s}(dx)\bigr) d \bm{y}\biggr)\lambda^{(\ell)}[\xi_{s}](d \bm{x}) d s.
\end{aligned}
\end{equation}
We then get
\begin{align*}
&|\langle f, \xi_{t}^{j}-\bar{\xi}_{t}^{j}\rangle|  \leq \int_{0}^{t}\biggl|\bigl\langle\sum_{j'=1}^J\langle(\tilde{\mathcal{L}}_{j,j'} \mathcal{P}_{j,t-s}f)(x,y), \bigl(\xi_s^{j'}(dy)-\bar{\xi}_s^{j'}(dy)\bigr)\rangle,\xi_{s}^{j}(d x)-\bar{\xi}_s^j(dx)\bigr\rangle\biggr| d s\\
&\phantom{=} +\int_{0}^{t}\biggl|\bigl\langle\sum_{j'=1}^J\langle(\tilde{\mathcal{L}}_{j,j'} \mathcal{P}_{j,t-s}f)(x,y), \bigl(\xi_{s}^{j'}(d y)-\bar{\xi}_s^{j'}(dy)\bigr)\rangle,\bar{\xi}_s^{j}(dx)\bigr\rangle\biggr| d s\\
&\phantom{=} +\int_{0}^{t}\biggl|\bigl\langle\sum_{j'=1}^J\langle(\tilde{\mathcal{L}}_{j,j'} \mathcal{P}_{j,t-s}f)(x,y), \bar{\xi}_s^{j'}(dy)\rangle,\xi_{s}^{j}(d x)-\bar{\xi}_s^j(dx)\bigr\rangle\biggr| d s\\
&\phantom{=} +\sum_{\ell=1}^{L} \int_{0}^{t} \biggl| \int_{\tilde{\mathbb{X}}(\ell)} \frac{1}{\bm{\alpha}^{(\ell)!}}K_{\ell}(\bm{x})\int_{\mathbb{Y}^{(\ell)}}\fsum{\mathcal{P}_{j,t-s}f} m_{\ell}(\bm{y} | \bm{x})\pi_{\ell} \bigl(\bm{y}|\bm{x}, \vxi_{s}(dx)\bigr) d\bm{y}\lambda^{(\ell)}[\xi_{s}](d \bm{x})\\
&\phantom{=} -\int_{\tilde{\mathbb{X}}(\ell)}\frac{1}{\bm{\alpha}^{(\ell)!}}K_{\ell}(\bm{x})\int_{\mathbb{Y}^{(\ell)}}\fsum{\mathcal{P}_{j,t-s}f} m_{\ell}(\bm{y} | \bm{x})\pi_{\ell} \bigl(\bm{y}|\bm{x}, \vxibar_{s}(dx)\bigr)d \bm{y}\lambda^{(\ell)}[\bar{\xi}_{s}](d \bm{x})\biggr| d s\numberthis\label{gw}.
\end{align*}

Recall that $\psi_{s}(x;t)$ solves
\begin{align*}
    \partial_{s} \psi_{s}(x;t) +(\mathcal{L}_{j}\psi_s)(x;t)&=0, \quad s\in[0,t]\nonumber\\
    \psi_{t}(x;t)&=f(x).
\end{align*}
By the Bismut-Elworthy formula (see Proposition 9.22 and Corollary 9.23 of \cite{DaPrato2014} or Theorem 3.2.4 of \cite{Stroock2008}), we have the following estimate for its gradient $\nabla_{x}\psi_{s}(x)$
\begin{align*}
|\nabla_{x}\psi_{s}(x;t)|&\leq C \frac{1}{1\wedge (t-s)^{1/2}}\|f\|_{C^{0}_{b}(\mathbb{R}^{d})},\quad \text{ for every }x\in\mathbb{R}^{d},
\end{align*}
for some finite constant $C<\infty$. Consequently, we obtain
\begin{align*}
     |(\tilde{\mathcal{L}}_{j,j'}\mathcal{P}_{j,t-s}f)(x,y)|&\leq C\norm{f}_{C^0_b(\mathbb{R^d})}\|u\|_{C^1(\mathbb{R}^{d\times2})} \frac{1}{1\wedge (t-s)^{1/2}}\nonumber\\
     &\leq C_3 \frac{1}{1\wedge (t-s)^{1/2}}, \quad \text{ for every }x,y\in\mathbb{R}^{d},
\end{align*}
where $C_3=C \|u\|_{C^1(\mathbb{R}^{d\times2})}$. Note that $C_3<\infty$ as $\norm{f}_{L^\infty} \leq 1$ and $\|u\|_{C^1(\mathbb{R}^{d\times2})}$ is bounded by Assumption \ref{anorm}.

Additionally, since $\sup_{t \in [0,T]} \|\mathcal{P}_{j,t} f\|_{L^\infty} \leq C<\infty$ as $\norm{f}_{L^\infty} \leq 1$ (see Chapter 4 of \cite{Pazy1983}), we get
\begin{align*}
&\sum_{\ell=1}^{L} \int_{0}^{t}  \biggr|\int_{\tilde{\mathbb{X}}(\ell)} \frac{1}{\bm{\alpha}^{(\ell)!}}K_{\ell}(\bm{x}) \int_{\mathbb{Y}^{(\ell)}}\left(\sum_{r=1}^{\beta_{\ell j}} |\mathcal{P}_{j, t-s} f(y_{r}^{(j)})|+\sum_{r=1}^{\alpha_{\ell j}} |\mathcal{P}_{j, t-s} f(x_{r}^{(j)})|\right)\\
&\qquad\qquad\qquad\qquad \times m_{\ell}(\bm{y} | \bm{x})\pi_{\ell} \bigl(\bm{y}|\bm{x}, \vxibar_{s}(dx)\bigr)d \bm{y}\bigl(\lambda^{(\ell)}[\xi_{s}](d \bm{x})-\lambda^{(\ell)}[\bar{\xi}_{s}](d \bm{x})\bigr)\biggr|ds\\
&\leq C(K) \sum_{\ell=1}^{L} \frac{\alpha_{\ell j}+\beta_{\ell j}}{\bm{\alpha}^{(\ell)}!} \int_{0}^{t}\|\lambda^{(\ell)}[\xi_{s}]-\lambda^{(\ell)}[\bar{\xi}_{s}]\|_{M_{F}(\mathbb{X}^{(\ell)})} d s,
\end{align*}
and by Assumption \ref{lippi}
\begin{align*}
&\sum_{\ell=1}^{L} \int_{0}^{t}  \int_{\tilde{\mathbb{X}}(\ell)} \biggr|\frac{1}{\bm{\alpha}^{(\ell)!}}K_{\ell}(\bm{x}) \int_{\mathbb{Y}^{(\ell)}}\left(\sum_{r=1}^{\beta_{\ell j}} |\mathcal{P}_{j, t-s} f(y_{r}^{(j)})|+\sum_{r=1}^{\alpha_{\ell j}} |\mathcal{P}_{j, t-s} f(x_{r}^{(j)})|\right)m_{\ell}(\bm{y} | \bm{x})\\
& \qquad\qquad\qquad\qquad\qquad \times\biggl(\pi_{\ell} \bigl(\bm{y}|\bm{x}, \vxi_{s}(dx)\bigr)-\pi_{\ell} \bigl(\bm{y}|\bm{x}, \vxibar_{s}(dx)\bigr)\biggr)d \bm{y}\biggr|\lambda^{(\ell)}[\xi_{s}](d \bm{x})ds\\
&\leq C(K) P \sum_{\ell=1}^{L} \frac{\alpha_{\ell j}+\beta_{\ell j}}{\bm{\alpha}^{(\ell)}!} \int_{0}^{t}\sum_{i=1}^J\|\xi_{s}^{i}-\bar{\xi}_{s}^{i}\|_{M_{F}(\mathbb{R}^{d})} d s.
\end{align*}

With $M=1 \vee \displaystyle\sup_{\{t \in[0, T], j=1,2, \cdots, J\}}|\langle 1, \xi_{t}^{j}\rangle| \vee|\langle 1, \bar{\xi}_{t}^{j}\rangle|<\infty$ and using the above estimates, we can further bound \eqref{gw} by
\begin{align*}
|\langle f, \xi_{t}^{j}-\bar{\xi}_{t}^{j}\rangle|
& \leq C_3\int_0^t \sum_{i=1}^J\|\xi_{s}^{i}-\bar{\xi}_{s}^{i}\|_{M_{F}(\mathbb{R}^{d})}\bigl(\langle1,\xi_{s}^{j}\rangle +2\langle1,\bar{\xi}_s^j\rangle\bigr) \frac{1}{1\wedge (t-s)^{1/2}} ds \\
&+ C_3\int_0^t \bigl(\sum_{i=1}^J\|\bar{\xi}_{s}^{i}\|_{M_{F}(\mathbb{R}^{d})}\bigr) \|\xi_{s}^{j}-\bar{\xi}_{s}^{j}\|_{M_{F}(\mathbb{R}^{d})} \frac{1}{1\wedge (t-s)^{1/2}}ds \\
&+ 2C(K) \sum_{\ell=1}^{L} \frac{\alpha_{\ell j}+\beta_{\ell j}}{\bm{\alpha}^{(\ell)}!} \int_{0}^{t}\|\otimes_{i=1}^{J}(\otimes_{r=1}^{\alpha_{\ell i}} \xi_{s}^{i})-\otimes_{i=1}^{J}(\otimes_{r=1}^{\alpha_{\ell i}} \bar{\xi}_{s}^{i})\|_{M_{F}(\mathbb{X}^{(\ell)})} d s \\
&+C(K)P \sum_{\ell=1}^{L} \frac{\alpha_{\ell j}+\beta_{\ell j}}{\bm{\alpha}^{(\ell)}!} \int_{0}^{t}\sum_{i=1}^J\|\xi_{s}^{i}-\bar{\xi}_{s}^{i}\|_{M_{F}(\mathbb{R}^{d})} d s\\
& \leq (3MC_3+C(C_{\circ})C_3)\int_0^t \sum_{i=1}^J\|\xi_{s}^{i}-\bar{\xi}_{s}^{i}\|_{M_{F}(\mathbb{R}^{d})}\frac{1}{1\wedge (t-s)^{1/2}}ds \\
&+2C(K) \sum_{\ell=1}^{L} \frac{\alpha_{\ell j}+\beta_{\ell j}}{\bm{\alpha}^{(\ell)}!} \int_{0}^{t} M^{|\alpha^{(\ell)}|-1} \sum_{i=1}^{J} \alpha_{\ell i}\|\xi_{s}^{i}-\bar{\xi}_{s}^{i}\|_{M_{F}(\mathbb{R}^{d})} d s\\
& +C(K)P\sum_{\ell=1}^{L} \frac{\alpha_{\ell j}+\beta_{\ell j}}{\bm{\alpha}^{(\ell)}!}\int_0^t \sum_{i=1}^J\|\xi_{s}^{i}-\bar{\xi}_{s}^{i}\|_{M_{F}(\mathbb{R}^{d})}ds.\numberthis
\end{align*}

Summing over all particle types, we get
\begin{align*}
&\sum_{j=1}^{J}\|\xi_{t}^{j}-\bar{\xi}_{t}^{j}\|_{M_{F}(\mathbb{R}^{d})} \leq \sum_{j=1}^{L} (3MC_3+\textcolor{black}{C(C_{\circ})C_3}+C(K)P\sum_{\ell=1}^{L} \frac{\alpha_{\ell j}+\beta_{\ell j}}{\bm{\alpha}^{(\ell)}!})\\
&\qquad \times \int_0^t \sum_{i=1}^J\|\xi_{s}^{i}-\bar{\xi}_{s}^{i}\|_{M_{F}(\mathbb{R}^{d})}\frac{1}{1\wedge (t-s)^{1/2}}ds \\
&\qquad+2C(K) \sum_{\ell=1}^{L} \sum_{j=1}^{J} \frac{\alpha_{\ell j}+\beta_{\ell j}}{\bm{\alpha}^{(\ell)}!} \int_{0}^{t} M^{| \alpha^{(\ell)}|-1} \sum_{i=1}^{J} \alpha_{\ell i}\|\xi_{s}^{i}-\bar{\xi}_{s}^{i}\|_{M_{F}(\mathbb{R}^{d})}\frac{1}{1\wedge (t-s)^{1/2}} d s \\
&\quad \leq \biggl(3MLC_3+\textcolor{black}{C(C_{\circ})C_3}+\bigl(LJC(K)P+2LJC(K)\bigr) \max _{1 \leq \ell \leq L, 1 \leq j \leq J}\bigl(\frac{\alpha_{\ell j}+\beta_{\ell j}}{\bm{\alpha}^{(\ell)}!} M^{| \alpha^{(\ell)} |-1} \alpha_{\ell j}\bigr)\biggr)\\
&\qquad\times\int_{0}^{t} \sum_{i=1}^{J}\|\xi_{s}^{i}-\bar{\xi}_{s}^{i}\|_{M_{F}(\mathbb{R}^{d})}\frac{1}{1\wedge (t-s)^{1/2}} d s\nonumber\\
&\quad\leq C \int_{0}^{t} \sum_{i=1}^{J}\|\xi_{s}^{i}-\bar{\xi}_{s}^{i}\|_{M_{F}(\mathbb{R}^{d})}\frac{1}{1\wedge (t-s)^{1/2}} d s,
\end{align*}
for some unimportant constant $C<\infty$. The integrable singularity of the kernel, $(1\wedge \sqrt{t-s})^{-1}$, can be handled using either of the weakly singular Gr\"{o}nwall Inequalities Lemma 7.1.1 of \cite{Henry1981} or Theorem 3.2 in \cite{Webb2019}. We then obtain that \[
\sum_{j=1}^{J}\sup_{t\in[0,T]}\|\xi_{t}^{j}-\bar{\xi}_{t}^{j}\|_{M_{F}(\mathbb{R}^{d})} = 0,
\]
  concluding the proof of the uniqueness of the limiting solution.

\section{Appendix}\label{S:Appendix}
\subsection{Verifications of assumptions for the  Fr\"{o}hner-No\'e model}\label{SS:MF_AcceptanceProbability}

We verify here that all forms of the acceptance probabilities for the generalized Fr\"{o}hner-No\'e model, i.e., \eqref{pduf}, \eqref{pdub}, and \eqref{pd}, satisfy Assumptions \eqref{lippi}, \eqref{lippiiny} and \eqref{convpi}. As such, they satisfy the conditions under which our main result, Theorem \ref{T:MainTheorem}, holds.

We will prove the next two lemmas for the limit of the acceptance probability given by \eqref{pduf}. Proofs for the other two cases are analogous. Consider $\pi_{\ell} = \pi_1$ for the following two proofs, and recall the definition of $\hat{M}_F$ in~\eqref{eq:Mhatdef}.
\begin{lemma}
Consider any $\gamma \geq 0, \tilde{L}+1 \leq \ell \leq L, \bm{y} \in \mathbb{Y}^{(\ell)}, \bm{x} \in \mathbb{X}^{(\ell)}$ and $\vxi = (\xi^{1},\xi^{2},\cdots,\xi^{J})$, $\vxibar = (\bxi^{1},\bxi^{2},\cdots,\bxi^{J})$ both in $\otimes_{j=1}^J \hat{M}_{F}(\mathbb{R}^{d}; C_{\circ})$. Then the acceptance probability $\pi_{\ell}\bigl(\bm{y}|\bm{x}, \vxi(dx)\bigr)= \min \{1, e^{-(\Phi_{1}^{+}(\bm{y},\vxi)-\Phi_{1}^{-}(\bm{x},\vxi))}\}$ is bounded and Lipschitz continuous with Lipschitz constant $P$, i.e.,
\begin{equation}
     \displaystyle\sup_{\bm{y} \in \mathbb{Y}^{(\ell)}, \bm{x} \in \mathbb{X}^{(\ell)}}\bigl|\pi_{\ell}\bigl(\bm{y}|\bm{x}, \vxi(dx)\bigr)-\pi_{\ell} \bigl(\bm{y}|\bm{x}, \vxibar(dx)\bigl)\bigr|\leq P \sum_{i=1}^J\|\xi^{i}-\bar{\xi}^{i}\|_{M_{F}(\mathbb{R}^{d})}.
\end{equation}
\end{lemma}
\begin{proof} Recalling that the potentials are all assumed to be uniformly bounded, we have
\begin{align*}
    \sup_{\bm{y} \in \mathbb{Y}^{(\ell)}, \bm{x} \in \mathbb{X}^{(\ell)}}&\bigl|\pi_{\ell}\bigl(\bm{y}|\bm{x}, \vxi(dx)\bigr) -\pi_{\ell} \bigl(\bm{y}|\bm{x}, \vxibar(dx)\bigl)\bigr|\\
    &\leq \sup_{\bm{y} \in \mathbb{Y}^{(\ell)}, \bm{x} \in \mathbb{X}^{(\ell)}}\biggl|e^{\sum_{j=1}^{J}\sum_{r=1}^{\alpha_{\ell j}}v_j(x^{(j)}_{r})-\sum_{j=1}^{J}\sum_{r=1}^{\beta_{\ell j}}v_j(y_{r}^{(j)})}\\
    &\phantom{=} \times \biggl(e^{\sum_{j=1}^J\sum_{j'=1}^J \bigl(\sum_{r=1}^{\alpha_{\ell j}}\int_{\mathbb{R}^d}u_{j,j'}(x^{(j)}_{r},x)\xi^{j'}(d x)-\sum_{r=1}^{\beta_{\ell j}}\int_{\mathbb{R}^d}u_{j,j'}(y_r^{(j)},x)\xi^{j'}(d x)\bigr)}\\
    &\phantom{=}  -e^{\sum_{j=1}^J\sum_{j'=1}^J \bigl(\sum_{r=1}^{\alpha_{\ell j}}\int_{\mathbb{R}^d}u_{j,j'}(x^{(j)}_{r},x)\bar{\xi}^{j'}(d x)-\sum_{r=1}^{\beta_{\ell j}}\int_{\mathbb{R}^d}u_{j,j'}(y_r^{(j)},x)\bar{\xi}^{j'}(d x)\bigr)}\biggr)\biggr|\\
    &\leq C \sup_{\bm{y} \in \mathbb{Y}^{(\ell)}, \bm{x} \in \mathbb{X}^{(\ell)}}\biggl|\sum_{j=1}^J\sum_{j'=1}^J \bigl(\sum_{r=1}^{\alpha_{\ell j}}\int_{\mathbb{R}^d}u_{j,j'}(x^{(j)}_{r},x)\xi^{j'}(d x)-\sum_{r=1}^{\beta_{\ell j}}\int_{\mathbb{R}^d}u_{j,j'}(y_r^{(j)},x)\xi^{j'}(d x)\bigr)\\
    &\phantom{=} - \sum_{j=1}^J\sum_{j'=1}^J \bigl(\sum_{r=1}^{\alpha_{\ell j}}\int_{\mathbb{R}^d}u_{j,j'}(x^{(j)}_{r},x)\bar{\xi}^{j'}(d x)-\sum_{r=1}^{\beta_{\ell j}}\int_{\mathbb{R}^d}u_{j,j'}(y_r^{(j)},x)\bar{\xi}^{j'}(d x)\bigr)\biggr|\\
    &\leq C \sup_{\bm{y} \in \mathbb{Y}^{(\ell)}, \bm{x} \in \mathbb{X}^{(\ell)}}
    \sum_{j=1}^J\sum_{j'=1}^J \Bigg[\sum_{r=1}^{\alpha_{\ell j}} \abs{\langle u_{j,j'}(x^{(j)}_{r},\cdot), \xi^{j'} - \bar{\xi}^{j'} \rangle} + \sum_{r=1}^{\beta_{\ell j}} \abs{\langle u_{j,j'}(y_r^{(j)},\cdot), \xi^{j'} - \bar{\xi}^{j'} \rangle} \Bigg]\\
    &\leq C P \sum_{i=1}^J\|\xi^{i}-\bar{\xi}^{i}\|_{M_{F}(\mathbb{R}^{d})}.
\end{align*}
The second inequality stems from the local Lipschitz continuity of exponential functions, the uniform bound on $\langle 1, \xi^j\rangle$ and $\langle 1, \bar{\xi}^j \rangle$, and the boundedness of the potentials implying global Lipschitz continuity. The last inequality is due to the uniform boundedness of the two-body potentials and the definition of the variation norm.
\end{proof}

\begin{lemma}
Consider any $\gamma \geq 0, \tilde{L}+1 \leq \ell \leq L, \bm{y}, \bm{y'} \in \mathbb{Y}^{(\ell)}, \bm{x} \in \mathbb{X}^{(\ell)}$ and $\vxi \coloneqq (\xi^{1},\xi^{2},\cdots,\xi^{J})$ in $\otimes_{j=1}^J \hat{M}_{F}(\mathbb{R}^{d}; C_{\circ})$. Then the acceptance probability $\pi_{\ell}\bigl(\bm{y}|\bm{x}, \vxi(dx')\bigr)= \min \{1, e^{-\bigl((\Phi_{1}^{+}(\bm{y},\vxi)-\Phi_{1}^{-}(\bm{x},\vxi)\bigr)}\}$ is bounded and Lipschitz continuous with Lipschitz constant $\tilde{P}$, i.e.,
\begin{equation}
    \sup_{\substack{\bm{x} \in \mathbb{X}^{(\ell)}\\\vxi \in \otimes_{j=1}^J M_{F}(\mathbb{R}^{d})}}\bigl|\pi_{\ell}\bigl(\bm{y}|\bm{x}, \vxi(dx')\bigr)-\pi_{\ell} \bigl(\bm{y'}|\bm{x}, \vxi(dx')\bigr)\bigr|\leq \tilde{P} \norm{\bm{y}-\bm{y'}}.
\end{equation}
\end{lemma}
\begin{proof}
\begin{align*}
&\sup_{\substack{\bm{x} \in \mathbb{X}^{(\ell)}\\\vxi \in \otimes_{j=1}^J \hat{M}_{F}(\mathbb{R}^{d}; C_{\circ})}} \bigl|\pi_{\ell}\bigl(\bm{y}|\bm{x}, \vxi(dx')\bigr)-\pi_{\ell} \bigl(\bm{y'}|\bm{x}, \vxi(dx')\bigl)\bigr|\\
\leq&\sup_{\substack{\bm{x} \in \mathbb{X}^{(\ell)}\\\vxi \in \otimes_{j=1}^J  \hat{M}_{F}(\mathbb{R}^{d}; C_{\circ})}}\biggl|e^{\sum_{j=1}^{J}\sum_{r=1}^{\alpha_{\ell j}}\bigl(v_j(x^{(j)}_{r})+\sum_{j'=1}^J\int_{\mathbb{R}^d}u_{j,j'}(x^{(j)}_{r},x)\xi_{s}^{j'}(d x)\bigr)}\\
\phantom{=}& \times \biggl(e^{-\sum_{j=1}^{J}\sum_{r=1}^{\beta_{\ell j}}\bigl(v_j(y_{r}^{(j)})+\sum_{j'=1}^J\int_{\mathbb{R}^d}u_{j,j'}(y_r^{(j)},x)\xi_{s}^{j'}(d x)\bigr)}-\nonumber\\
&\hspace{7cm}-e^{-\sum_{j=1}^{J}\sum_{r=1}^{\beta_{\ell j}}\bigl(v_j(y_{r}^{',(j)})+\sum_{j'=1}^J\int_{\mathbb{R}^d}u_{j,j'}(y_r^{',(j)},x)\xi_{s}^{j'}(d x)\bigr)}\biggr)\biggr|\\
\leq& C\sup_{\substack{\bm{x} \in \mathbb{X}^{(\ell)}\\\vxi \in \otimes_{j=1}^J  \hat{M}_{F}(\mathbb{R}^{d}; C_{\circ})}}\biggl|\sum_{j=1}^J\sum_{r=1}^{\beta_{\ell j}}\biggl(\bigl(v_j(y_{r}^{(j)})-v_j(y_{r}^{',(j)})\bigr) + \sum_{j'=1}^J\int_{\mathbb{R}^d}\bigl(u_{j,j'}(y_r^{(j)},x)-u_{j,j'}(y_r^{',(j)},x)\bigr)\xi^{j'}(d x)\biggr)\biggr|\\
\leq& C \sum_{j=1}^J\sum_{r=1}^{\beta_{\ell j}}\bigl( |y_r^{(j)}-y_r^{',(j)}|\bigr)\\
\leq& C \norm{\bm{y}-\bm{y'}},
\end{align*}
for some constant $C<\infty$ that may be changing from line to line. The second inequality stems from the local Lipschitz continuity of exponential functions, the uniform bound on $\langle 1, \xi^j \rangle$, and the boundedness of the potentials implying global Lipschitz continuity. For the third inequality we used that the one-body potential $v_j$ is Lipschitz. We again used the boundedness of $\left\langle 1, \xi^{i}\right\rangle$ for all $1 \leq i \leq J$ and the assumed global boundedness of the two-body potential $u_{j,j'}$ in $C^1$ by Assumption \ref{anorm}.
\end{proof}

\begin{lemma}For any $\gamma > 0, \tilde{L}+1 \leq \ell \leq L$ and $\vxi \coloneqq (\xi^{1},\xi^{2},\cdots,\xi^{J}) \in \otimes_{j=1}^J \hat{M}_{F}(\mathbb{R}^{d}; C_{\circ})$, the acceptance probability $\pi_{\ell}^{\gamma}\bigl(\bm{y}|\bm{x}, \vxi(dx)\bigr)$ (given by either of \eqref{pduf}, \eqref{pdub} and \eqref{pd}) converges to $\pi_{\ell}\bigl(\bm{y}|\bm{x}, \vxi(dx)\bigr)$ (given by \eqref{pdmfl}) uniformly as $\gamma \rightarrow \infty$ with respect to any $\bm{y} \in \mathbb{Y}^{(\ell)}$ and $\bm{x} \in \mathbb{X}^{(\ell)}$.
Equivalently,
\begin{equation}
     \sup_{\bm{y} \in \mathbb{Y}^{(\ell)}, \bm{x} \in \mathbb{X}^{(\ell)}}\bigl|\pi_{\ell}\bigl(\bm{y}|\bm{x}, \vxi(dx)\bigr) -\pi_{\ell}^{\gamma}\bigl(\bm{y}|\bm{x}, \vxi(dx)\bigr)\bigr| \xrightarrow[\gamma \rightarrow \infty]{} 0.
\end{equation}
\end{lemma}
\begin{proof} We only demonstrate the proof of convergence of \eqref{pdub} to \eqref{pdmfl} as the other two cases are the same. Recall Remark~\ref{rmk:potentialwithmeasures}, justifying why $\pi^{\gamma}_{\ell}$ is well-defined as a function of general finite measures. For $\bm{x} \in \mathbb{X}^{(1)}$ and $\bm{y} \in \mathbb{Y}^{(1)}$, factoring out the one-body potential terms and using their uniform boundedness we find
    \begin{align*}
\sup_{\bm{y} \in \mathbb{Y}^{(\ell)}, \bm{x} \in \mathbb{X}^{(\ell)}}&\bigl|\pi_{\ell}^{\gamma}\bigl(\bm{y}|\bm{x}, \vxi(dx)\bigr) -\pi_{\ell}\bigl(\bm{y}|\bm{x}, \vxi(dx)\bigr)\bigr|\\
\leq&C\displaystyle\sup_{\bm{y} \in \mathbb{Y}^{(\ell)}, \bm{x} \in \mathbb{X}^{(\ell)}}\biggl|e^{\sum_{j=1}^J\sum_{j'=1}^J \bigl(\sum_{r=1}^{\alpha_{\ell j}}\int_{\mathbb{R}^d}u_{j,j'}(x^{(j)}_{r},x)\xi^{j'}(d x)-\sum_{r=1}^{\beta_{\ell j}}\int_{\mathbb{R}^d}u_{j,j'}(y_r^{(j)},x)\xi^{j'}(d x)\bigr)}\\
\times& \biggl( e^{-\sum_{j=1}^{J}\sum_{r=1}^{\beta_{\ell j}}\bigl(\sum_{j'=j+1}^{J}\sum_{r'=1}^{\beta_{\ell j'}}\frac{1}{\gamma}u_{j,j'}(y^{(j)}_{r},y^{(j')}_{r'})+\sum_{r'=1}^{r-1}\frac{1}{\gamma}u_{j,j}(y^{(j)}_{r},y^{(j)}_{r'})\bigr)}\\
\times&e^{\sum_{j=1}^J\sum_{j'=1}^J\sum_{r'=1}^{\alpha_{\ell j'}}\bigl(\sum_{r=1}^{\beta_{\ell j}}\frac{1}{\gamma}u_{j,j'}(y_r^{(j)},x^{(j')}_{r'})-\sum_{r=1}^{\alpha_{\ell j}}\frac{1}{\gamma}u_{j,j'}(x^{(j)}_{r},x^{(j')}_{r'})\bigr)}-1\biggr)\biggr|\xrightarrow[\gamma \rightarrow \infty]{} 0.
    \end{align*}
The convergence result follows from the uniform boundedness of the two-body potentials and $\langle 1, \xi^j \rangle$.
\end{proof}

\subsection{Proof of Lemma 8.3.}\label{S:mollifierconvlemmaproof}
For any $\eta \geq 0$ small enough, $\tilde{L}+1 \leq \ell \leq L, \bm{y} \in \mathbb{Y}^{(\ell)}, \bm{x} \in \mathbb{X}^{(\ell)}$, and $f \in C_{b}^{2}(\mathbb{Y}^{(\ell)})$, there exists a constant $C$ such that
\begin{equation*}
\biggl|\int_{\mathbb{Y}^{(\ell)}} f(\bm{y})\pi_{\ell}\bigl(\bm{y}|\bm{x}, \vmu_{r-}^{\zeta}(dx')\bigr) \bigl(m_{\ell}^{\eta}(\bm{y} | \bm{x})-m_{\ell}(\bm{y} | \bm{x})\bigr) d \bm{y}\biggr| \leq C\|f\|_{C_{b}^{2}(\mathbb{Y}^{(\ell)})}{\eta}.
\end{equation*}

\begin{proof}
Case 1: Reaction of the form $S_{i} \rightarrow S_{j}$.
\begin{align*}
&\biggl|\int_{\mathbb{Y}^{(\ell)}} f(\bm{y})\pi_{\ell}\bigl(\bm{y}|\bm{x}, \vmu_{r-}^{\zeta}(dx')\bigr) \bigl(m_{\ell}^{\eta}(\bm{y} | \bm{x})-m_{\ell}(\bm{y} | \bm{x})\bigr) d \bm{y}\biggr|\\
=&\biggl|\int_{\mathbb{R}^{d}} f(y)\pi_{\ell}\bigl(y|x, \vmu_{r-}^{\zeta}(dx')\bigr)  G_{\eta}(y-x) d y-f(x)\pi_{\ell}\bigl(x|x, \vmu_{r-}^{\zeta}(dx')\bigr) \biggr| \\
=&\biggl|\int_{\mathbb{R}^{d}}\biggl(f(y)\pi_{\ell}\bigl(y|x, \vmu_{r-}^{\zeta}(dx')\bigr)-f(x)\pi_{\ell}\bigl(x|x, \vmu_{r-}^{\zeta}(dx')\bigr)\biggr) G_{\eta}(y-x) d y\biggr| \\
\leq&\biggl|\int_{\mathbb{R}^{d}}f(y)\biggl(\pi_{\ell}\bigl(y|x, \vmu_{r-}^{\zeta}(dx')\bigr)-\pi_{\ell}\bigl(x|x, \vmu_{r-}^{\zeta}(dx')\bigr)\biggr) G_{\eta}(y-x) d y\biggr|\\
&+\biggl|\int_{\mathbb{R}^{d}}\pi_{\ell}\bigl(x|x, \vmu_{r-}^{\zeta}(dx')\bigr)\bigl(f(y)-f(x)\bigr) G_{\eta}(y-x) d y\biggr|\\
\leq&\int_{B(x, \eta)}|f(y)||\pi_{\ell}\bigl(y|x, \vmu_{r-}^{\zeta}(dx')\bigr)-\pi_{\ell}\bigl(x|x, \vmu_{r-}^{\zeta}(dx')\bigr)|G_{\eta}(y-x) d y\\
&+\int_{B(x, \eta)}\pi_{\ell}\bigl(x|x, \vmu_{r-}^{\zeta}(dx')\bigr)| f(y)-f(x)|G_{\eta}(y-x) d y\\
\leq& C'\int_{B(x, \eta)}\|f\|_{C_{b}^{0}(\mathbb{R}^{d})} \eta \times G_{\eta}(y-x) d y+ \int_{B(x, \eta)}\|f\|_{C_{b}^{1}(\mathbb{R}^{d})} \eta \times G_{\eta}(y-x) d y \\
\leq&C\|f\|_{C_{b}^{1}(\mathbb{R}^{d})} \eta,
\end{align*}
where Assumption \ref{lippiiny} was used to derive the second to last inequality.

Case 2: Reaction of the form $S_{i} \rightarrow S_{j}+S_{k}$.
\begin{align*}
&\biggl|\int_{\mathbb{Y}^{(\ell)}} f(\bm{y})\pi_{\ell}\bigl(\bm{y}|\bm{x}, \vmu_{r-}^{\zeta}(dx')\bigr) \bigl(m_{\ell}^{\eta}(\bm{y} | \bm{x})-m_{\ell}(\bm{y} | \bm{x})\bigr) d \bm{y}\biggr|\\
=&\biggl| \sum_{i=1}^{I} p_{i} \times\biggl[\int_{\mathbb{R}^{2 d}} f\bigl(y_{1}, y_{2}\bigr) \rho\bigl(\left|y_{1}-y_{2}\right|\bigr) \pi_{\ell}\bigl(y_1,y_2|x, \vmu_{r-}^{\zeta}(dx')\bigr) G_{\eta}\bigl(x-\bigl(\alpha_{i} y_{1}+\bigl(1-\alpha_{i}\bigr) y_{2}\bigr)\bigr) d y_{1} d y_{2} \\
&-\int_{\mathbb{R}^{2 d}} f\bigl(y_{1}, y_{2}\bigr) \rho\bigl(\left|y_{1}-y_{2}\right|\bigr) \pi_{\ell}\bigl(y_1,y_2|x, \vmu_{r-}^{\zeta}(dx')\bigr) \delta\bigl(x-\bigl(\alpha_{i} y_{1}+\bigl(1-\alpha_{i}\bigr) y_{2}\bigr)\bigr) d y_{1} d y_{2}\biggr] \biggr| \\
=&\biggl| \sum_{i=1}^{I} p_{i} \times\biggl[\int_{\mathbb{R}^{2 d}} f(w+y_{2}, y_{2}) \pi_{\ell}\bigl(w+y_{2},y_2|x, \vmu_{r-}^{\zeta}(dx')\bigr)\rho(|w|) G_{\eta}(x-\alpha_{i} w-y_{2})) d w d y_{2} \\
&-\int_{\mathbb{R}^{2 d}} f(w+y_{2}, y_{2})\pi_{\ell}\bigl(w+y_{2},y_2|x, \vmu_{r-}^{\zeta}(dx')\bigr) \rho(|w|) \delta(x-\alpha_{i} w-y_{2})) d w d y_{2}\biggr] \biggr| \\
=&\biggl| \sum_{i=1}^{I} p_{i}\times\biggl[\int_{\mathbb{R}^{d}} \rho(|w|)\biggl(\int_{\mathbb{R}^{d}}\bigl(f(w+y_{2}, y_{2})\pi_{\ell}\bigl(w+y_{2},y_2|x, \vmu_{r-}^{\zeta}(dx')\bigr)\\
&-f(w+x-\alpha_{i} w, x-\alpha_{i} w)\bigr) \pi_{\ell}\bigl(w+x-\alpha_{i} w, x-\alpha_{i} w|x, \vmu_{r-}^{\zeta}(dx')\bigr)G_{\eta}(x-\alpha_{i} w-y_{2}) d y_{2}\biggr) d w\biggr] \biggr|\\
\leq& \sum_{i=1}^{I} p_{i} \times\biggl[\int_{\mathbb{R}^{d}} \rho(|w|)\biggl(\int_{B\bigl(x-\alpha_{i} w, \eta\bigr)}\bigl|f(w+y_{2}, y_{2})\pi_{\ell}\bigl(w+y_{2},y_2|x, \vmu_{r-}^{\zeta}(dx')\bigr)\\
&-f(w+x-\alpha_{i} w, x-\alpha_{i} w)\pi_{\ell}\bigl(w+x-\alpha_{i} w, x-\alpha_{i} w|x, \vmu_{r-}^{\zeta}(dx')\bigr) \bigr|G_{\eta}(x-\alpha_{i} w-y_{2}) d y_{2}\biggr) d w\biggr]\\
\leq& \sum_{i=1}^{I} p_{i} \times\biggl[\int_{\mathbb{R}^{d}} \rho(|w|)\biggl(\int_{B\bigl(x-\alpha_{i} w, \eta\bigr)}\biggl(\bigl|f(w+y_{2}, y_{2})\bigl(\pi_{\ell}\bigl(w+y_{2},y_2|x, \vmu_{r-}^{\zeta}(dx')\bigr)\\
&-\pi_{\ell}\bigl(w+x-\alpha_{i} w, x-\alpha_{i} w|x, \vmu_{r-}^{\zeta}(dx')\bigr)\bigr)\bigr|+\bigl|f(w+y_{2}, y_{2})-f(w+x-\alpha_{i} w, x-\alpha_{i} w)\bigr|\\
&\times \pi_{\ell}\bigl(w+x-\alpha_{i} w, x-\alpha_{i} w|x, \vmu_{r-}^{\zeta}(dx')\bigr)\biggr)G_{\eta}(x-\alpha_{i} w-y_{2}) d y_{2}\biggr) d w\biggr]\\
\leq& \sum_{i=1}^{I} p_{i} \times\biggl[\int_{\mathbb{R}^{d}} \rho(|w|)\biggl(C'\int_{B\bigl(x-\alpha_{i} w, \eta\bigr)}\|f\|_{C^{0}\bigl(\mathbb{R}^{2 d}\bigr)} \eta \times G_{\eta}(x-\alpha_{i} w-y_{2}) d y_{2}\\
&+\int_{B\bigl(x-\alpha_{i} w, \eta\bigr)}\|f\|_{C^{1}\bigl(\mathbb{R}^{2 d}\bigr)} \eta \times G_{\eta}(x-\alpha_{i} w-y_{2}) d y_{2}\biggr) d w\biggr] \\
\leq&C\|f\|_{C_{b}^{1}(\mathbb{R}^{d})} \eta,
\end{align*}
where Assumption \ref{lippiiny} was used to derive the second to last inequality.

Case 3: Reaction of the form $S_{i}+S_{k} \rightarrow S_{j}$.
\begin{align*}
&\biggl|\int_{\mathbb{Y}^{(\ell)}} f(\bm{y})\pi_{\ell}\bigl(\bm{y}|\bm{x}, \vmu_{r-}^{\zeta}(dx')\bigr) \bigl(m_{\ell}^{\eta}(\bm{y} | \bm{x})-m_{\ell}(\bm{y} | \bm{x})\bigr) d \bm{y}\biggr|\\
=&\biggl|\sum_{i=1}^{I} p_{i} \times\biggl[\int_{\mathbb{R}^{d}}\biggl(f(y) \pi_{\ell}\bigl(y|x_1, x_2, \vmu_{r-}^{\zeta}(dx')\bigr)-f\bigl(\alpha_{i} x_{1}+(1-\alpha_{i}) x_{2}\bigr)\times\\
&\times\pi_{\ell}\bigl(\alpha_{i} x_{1}+(1-\alpha_{i})x_2|x_1,x_2, \vmu_{r-}^{\zeta}(dx')\bigr)\biggr)G_{\eta}(y-\bigl(\alpha_{i} x_{1}+(1-\alpha_{i}) x_{2})\bigr) d y\biggr]\biggr| \\
\leq&C\|f\|_{C_{b}^{1}(\mathbb{R}^{d})} \eta.
\end{align*}

Case 4: Reaction of the form $S_{i}+S_{k} \rightarrow S_{j}+S_{r}$.
\begin{align*}
&\biggl|\int_{\mathbb{Y}^{(\ell)}} f(\bm{y})\pi_{\ell}\bigl(\bm{y}|\bm{x}, \vmu_{r-}^{\zeta}(dx')\bigr) \bigl(m_{\ell}^{\eta}(\bm{y} | \bm{x})-m_{\ell}(\bm{y} | \bm{x})\bigr) d \bm{y}\biggr|\\
=&\biggl| p \times\biggl[\int_{\mathbb{R}^{2 d}} f(y_{1}, y_{2}) \pi_{\ell}\bigl(y_1,y_2|x_1, x_2, \vmu_{r-}^{\zeta}(dx')\bigr)G_{\eta}(y_{1}-x_{1}) G_{\eta}(y_{2}-x_{2}) d y_{1} d y_{2}\\
&-f(x_{1}, x_{2})\pi_{\ell}\bigl(x_1,x_2|x_1, x_2, \vmu_{r-}^{\zeta}(dx')\bigr)\biggr]\\
+&(1-p)\times\biggl[\int_{\mathbb{R}^{2 d}} f(y_{1}, y_{2}) \pi_{\ell}\bigl(y_1, y_2|x_1, x_2, \vmu_{r-}^{\zeta}(dx')\bigr)G_{\eta}(y_{2}-x_{1}) G_{\eta}(y_{1}-x_{2}) d y_{1} d y_{2}\\
&-f(x_{2}, x_{1})\pi_{\ell}\bigl(x_2, x_1|x_1, x_2, \vmu_{r-}^{\zeta}(dx')\bigr)\biggr]\biggr|\\
\leq &p \times \int_{B\left(\left(x_{1}, x_{2}\right), \sqrt{2} \eta\right)}\bigl| f\left(y_{1}, y_{2}\right)\pi_{\ell}\bigl(y_1, y_2|x_1, x_2, \vmu_{r-}^{\zeta}(dx')\bigr)\\
&-f\left(x_{1}, x_{2}\right)\pi_{\ell}\bigl(x_1, x_2|x_1, x_2, \vmu_{r-}^{\zeta}(dx')\bigr) \bigr| G_{\eta}\left(y_{1}-x_{1}\right) G_{\eta}\left(y_{2}-x_{2}\right) d y_{1} d y_{2} \\
&+(1-p) \times \int_{B\left(\left(x_{2}, x_{1}\right), \sqrt{2} \eta\right)}\bigl|f\left(y_{1}, y_{2}\right)\pi_{\ell}\bigl(y_{1}, y_{2}|x_1, x_2, \vmu_{r-}^{\zeta}(dx')\bigr)\\
&-f\left(x_{2}, x_{1}\right)\pi_{\ell}\bigl(x_2, x_1|x_1, x_2, \vmu_{r-}^{\zeta}(dx')\bigr)\bigr| G_{\eta}\left(y_{2}-x_{1}\right) G_{\eta}\left(y_{1}-x_{2}\right) d y_{1} d y_{2}  \\
\leq &  C'p \times \int_{B\left(\left(x_{1}, x_{2}\right), \sqrt{2} \eta\right)}\|f\|_{C_{b}^{0}\left(\mathbb{R}^{2 d}\right)} \times \sqrt{2} \eta \times G_{\eta}\left(y_{1}-x_{1}\right) G_{\eta}\left(y_{2}-x_{2}\right) d y_{1} d y_{2} \\
&+C'(1-p) \times \int_{B\left(\left(x_{2}, x_{1}\right), \sqrt{2} \eta\right)}\|f\|_{C_{b}^{0}\left(\mathbb{R}^{2 d}\right)} \times \sqrt{2} \eta \times G_{\eta}\left(y_{2}-x_{1}\right) G_{\eta}\left(y_{1}-x_{2}\right) d y_{1} d y_{2}\\
&+p \times \int_{B\left(\left(x_{1}, x_{2}\right), \sqrt{2} \eta\right)}\|f\|_{C_{b}^{1}\left(\mathbb{R}^{2 d}\right)} \times \sqrt{2} \eta \times G_{\eta}\left(y_{1}-x_{1}\right) G_{\eta}\left(y_{2}-x_{2}\right) d y_{1} d y_{2} \\
&+(1-p) \times \int_{B\left(\left(x_{2}, x_{1}\right), \sqrt{2} \eta\right)}\|f\|_{C_{b}^{1}\left(\mathbb{R}^{2 d}\right)} \times \sqrt{2} \eta \times G_{\eta}\left(y_{2}-x_{1}\right) G_{\eta}\left(y_{1}-x_{2}\right) d y_{1} d y_{2}\\
\leq & C\|f\|_{C_{b}^{1}\left(\mathbb{R}^{2 d}\right)} \eta,
\end{align*}
where Assumption \ref{lippiiny} was used to derive the second to last inequality.
\end{proof}

\printbibliography

\end{document}